\documentclass[12pt,reqno]{amsart}
\usepackage{fullpage}

\usepackage{amsmath,amssymb,amsfonts,amsthm}
\usepackage{bm,bbm}
\usepackage{comment}
\usepackage{caption}
\usepackage{placeins}
\usepackage{graphicx}
\usepackage{subcaption}
\usepackage{booktabs}
\usepackage{float}

\usepackage[colorlinks=true,allcolors=blue]{hyperref}

\usepackage{enumitem}

\usepackage{algorithm}
\usepackage{algpseudocode}

\newtheorem{thm}{Theorem}[section]
\newtheorem{lem}[thm]{Lemma}
\newtheorem{remark}{Remark}

\begin{document}
\title{Inverse Scattering for Dirac Equations Arising in Waveguide Arrays}
\author{John C. Schotland}
\address{Department of Mathematics and Department of Physics, Yale University, New Haven, CT, USA}
\email{john.schotland@yale.edu}
\author{Shenwen Yu}
\address{Department of Mathematical Sciences, Tsinghua University, Beijing,  China}
\email{ysw22@mails.tsinghua.edu.cn}
\begin{abstract}
 We investigate inverse scattering problems for  Dirac equations that arise as continuum models of waveguide arrays. We first establish the well-posedness of the forward models. For the associated inverse problems, we develop the inverse Born series  and the reduced inverse Born series, providing analysis of convergence and rigorous error estimates. Numerical experiments are presented to validate the proposed algorithms and demonstrate their effectiveness.
\end{abstract}
\maketitle

\section{Introduction}
Inverse scattering problems for wave equations and related first-order systems appear naturally in numerous physical settings. In this paper, we study two such problems for Dirac equations that arise as continuum models of waveguide arrays~\cite{quantum2023}. In particular, we consider equations of the form
\begin{equation}\label{chiral case}
    -\mathrm{i}\,\partial_x \psi + \mathrm{i}\,\alpha\,\partial_y \psi + k\bigl(\beta + V(x,y)\bigr)\psi = 0,
\end{equation}
and
\begin{equation}\label{antichiral case}
    \mathrm{i}\,\beta\,\partial_x \psi + \mathrm{i}\,\alpha\,\partial_y \psi + k\bigl(V(x,y)-1\bigr)\psi = 0,
\end{equation}
Here $\alpha$ and $\beta$ are the Pauli matrices
\[
    \alpha=\begin{pmatrix}0&1\\[0.2em]1&0\end{pmatrix},
    \quad
    \beta=\begin{pmatrix}1&0\\[0.2em]0&-1\end{pmatrix},
\]
the spinor field $\psi(x,y)$ denotes the probability amplitude for creating a photon at the point $(x,y)$, $V(x,y)$ is the scattering potential for the atoms in the array, and the wavenumber $k$ is nonnegative. Eqs.~\eqref{chiral case} and \eqref{antichiral case} will be referred to as the chiral and antichiral Dirac equations, respectively. We note that Eq.~\eqref{chiral case} is hyperbolic, whereas Eq.~\eqref{antichiral case}
is elliptic. Although elliptic inverse problems are typically more ill-posed than
hyperbolic ones, this distinction must be interpreted together with the type of available
measurement data. As we shall see, the two models lead to different inverse problem
formulations and exhibit different reconstruction behavior.

The inverse problem we consider is to determine $V$ from suitable measurements of $\psi$. Our approach to this problem is to employ the inverse Born series inversion method. 
The Born series expresses the scattered field as an infinite series of multilinear operators applied to the potential. Inverting this expansion yields the inverse Born series, which allows the reconstruction of the potential as a multilinear series in the measured field. This method was analyzed in \cite{2008shari} and further analyzed in a general Banach-space setting in \cite{hoskinsAnalysisInverseBorn2022}, which describes convergence criteria and error estimates.
 In addition, we also investigate the reduced inverse Born series, following \cite{2022reduced}, in which only a distinguished subset of terms is retained. This reduction is motivated by cancellation phenomena among higher-order terms in the inverse Born series and leads to a substantial decrease in computational cost compared to the full inverse series.
 
The literature on inverse problems for Dirac equations remains relatively limited.
Existing works include inverse spectral problems
\cite{watsonInverseSpectralProblems1999,mykytyukInverseSpectralProblems2012}
and inverse boundary value problems
\cite{saloInverseProblemsPartial2010,kurylevInverseProblemsIndex2009}.
For inverse scattering problems, uniqueness results were established in
\cite{hachemPartialbarApproachInverse1995, isozakiInverseScatteringTheory1997}.
Stability estimates were obtained in
\cite{kawamotoDeterminationElectromagneticPotential2012}
using Carleman estimate techniques. Reconstruction methods
were derived in
\cite{grebertInverseScatteringDirac1992, belishevInverseProblemOnedimensional2014},
although their numerical implementation was not explored. More recently, inverse
scattering problems for nonlinear Dirac equations have been investigated in
\cite{yiBoundedTimeInverse2025}.

The paper is organized as follows. Section~2 introduces the forward problem and proves well-posedness. Section~3 develops the Born series, inverse Born series, and reduced inverse Born series. Section~4 describes the algorithms for the forward and inverse problems. Section~5 presents the results of numerical reconstructions and compares the performance of the reconstruction methods.

\section{Forward problem}
We begin by recalling the Dirac equations \eqref{chiral case} and \eqref{antichiral case} and assume that  the scattering potential $V$ is compactly supported in the rectangular domain
\[
    \Omega := [0,L_x]\times[0,L_y]\subset\mathbb{R}^2.
\]
For the chiral model, we consider the initial--boundary value problem
\begin{equation}\label{chiral model}
\left\{
\begin{aligned}
  &-\mathrm{i}\,\partial_x \psi + \mathrm{i}\,\alpha\,\partial_y \psi + k\bigl[\beta + V(x,y)\bigr]\psi = 0
  && \text{in } \Omega,\\[0.4em]
  &\psi_2(x,0)=0,\quad \psi_1(x,L_y)=0
  && \text{for } 0\le x\le L_x,\\[0.4em]
  &\psi(0,y)=g(y)
  && \text{for } 0\le y\le L_y,
\end{aligned}
\right.
\end{equation}
where $\psi=(\psi_1,\psi_2)^{\mathrm{T}}$ and $g$ is a smooth, compactly supported initial condition.
We note that $x$ plays the role of a time-like variable, whereas $y$ is space-like.
For the anti-chiral model, we decompose the total field as $\psi=\psi_0+\psi_s$, where $\psi_0$ is the incident field and $\psi_s$ denotes the scattered field. We find that $\psi_s$ obeys
\begin{equation}\label{s antichiral case}
    \mathrm{i}\,\beta\,\partial_x \psi_s + \mathrm{i}\,\alpha\,\partial_y \psi_s + k\bigl[V(x,y)-1\bigr]\psi_s
    = -k\,V(x,y)\,\psi_0 \quad  \text{in} \quad \mathbb R^2 .
    \end{equation}
Here we impose the radiation condition
\begin{equation}\label{radiation}
    \lim_{r\to\infty}\sqrt{r}\,\bigl(\partial_r\psi_s-\mathrm{i}k\psi_s\bigr)=0 ,
\end{equation}
where $r=\sqrt{x^2+y^2}$.

\subsection{Well-posedness of the forward problems}
We first consider the chiral model and rewrite Eq.~\eqref{chiral case} in the form of an evolution equation respect to the time-like variable $x$:
\[
\partial_x \psi = A\psi,
\quad
A := \alpha\,\partial_y - \mathrm{i}k\beta.
\]
We incorporate the boundary conditions into the domain of $A$ according to
\[
D(A)
:=\Bigl\{
\psi\in L^2(0,L_y;\mathbb{C}^2):
\ \psi_2(0)=0,\ \psi_1(L_y)=0
\Bigr\}.
\]
The operator $A$ is skew-adjoint on $L^2(0,L_y;\mathbb{C}^2)$, and thus it generates a strongly continuous semigroup $\{T(x)\}_{x\ge0}$ satisfying
\[
\|T(x)\|_{L^2\to L^2}=1,
\quad
T(x+x')=T(x)T(x').
\]
Eq.~\eqref{chiral model} can then be written in evolution form as
\[
\partial_x\psi = A\psi + B(x)\psi,
\quad
B(x):= -\mathrm{i}k\,V(x,\cdot).
\]

\begin{thm}
Assume that $V\in L^{\infty}(\Omega)$. Then the initial--boundary value problem \eqref{chiral model} admits a unique solution
\[
\psi\in C\bigl([0,L_x];L^2(0,L_y;\mathbb{C}^2)\bigr)
\]
satisfying the integral equation
\begin{equation}\label{iter of chiral}
    \psi(x)
    = T(x)g
    + \int_0^x T(x-x')B(x')\psi(x')\,dx'.
\end{equation}
\end{thm}

\begin{proof}
Set
\[
x_0=\frac{1}{2k\|V\|_{L^{\infty}(\Omega)}},
\quad
X_0:=C\bigl([0,x_0];L^2(0,L_y;\mathbb{C}^2)\bigr),
\]
equipped with the supremum norm. Define the mapping $\Phi_c:X_0\to X_0$ by
\[
(\Phi_c\psi)(x)
:=
T(x)g
+
\int_0^x T(x-x')B(x')\psi(x')\,dx'.
\]
Using $\|T(x)\|_{L^2\to L^2}=1$ and $\|B(s)\|_{L^2\to L^2}\le k\|V\|_{L^\infty(\Omega)}$, we obtain
\begin{align*}
    \|\Phi_c\psi-\Phi_c\varphi\|_{X_0}
&\leq
\left(\int_0^{x_0}k\|V(s,\cdot)\|_{L^\infty([0,L_y])}\,ds\right)\|\psi-\varphi\|_{X_0}\\
&\leq
 x_0 k\|V\|_{L^\infty(\Omega)}\,\|\psi-\varphi\|_{X_0}.
\end{align*}
By the choice of $x_0$,  we have $ x_0k\|V\|_{L^\infty(\Omega)}=\tfrac12$, hence $\Phi_c$ is a contraction on $X_0$. Existence and uniqueness of the solution on $[0,x_0]$ follow from Banach's fixed point theorem. We then repeat the same argument on the successive subintervals
\[
[x_0,2x_0],\ [2x_0,3x_0],\ \ldots,
\]
thereby obtaining a unique solution on each subinterval and, consequently, a unique solution on the entire interval $[0,L_x]$. 
\end{proof}

We next turn to the anti-chiral model. Define the Dirac operator
\[
\mathcal{D}:=\mathrm{i}\,\beta\,\partial_x+\mathrm{i}\,\alpha\,\partial_y-k,
\]
and assume that the incident field $\psi_0$ satisfies the homogeneous equation $\mathcal{D}\psi_0=0$. Then the scattered field satisfies
\[
\mathcal{D}\psi_s=-k\,V\,\psi,
\quad
\psi=\psi_0+\psi_s.
\]
A representation for the solution is obtained via the Green's function associated with $\mathcal{D}$. As given in \cite{quantum2023}, the Green's function is
\begin{equation}\label{Def of green}
G(\mathbf{x},\mathbf{x}^{\prime})
=\frac{k}{4}\Bigg(
\mathrm{i}H^{(1)}_0\bigl(k|\mathbf{x}-\mathbf{x}^{\prime}|\bigr)I
+
H^{(1)}_1\bigl(k|\mathbf{x}-\mathbf{x}^{\prime}|\bigr)
\Bigl(
\frac{x-x^{\prime}}{|\mathbf{x}-\mathbf{x}^{\prime}|}\beta
+
\frac{y-y^{\prime}}{|\mathbf{x}-\mathbf{x}^{\prime}|}\alpha
\Bigr)\Bigg),
\end{equation}
where $\mathbf{x}=(x,y)$ and $\mathbf{x}^{\prime}=(x',y')$.
Consequently, $\psi_s$ satisfies the Lippmann--Schwinger equation
\begin{equation}\label{iter of anti}
    \psi_s(\mathbf{x})
    =-k\int_{\Omega}G(\mathbf{x},\mathbf{x}^{\prime})\,V(\mathbf{x}^{\prime})\,\psi(\mathbf{x}^{\prime})\,d\mathbf{x}^{\prime}.
\end{equation}

\begin{thm}
Suppose that $V$ is compactly supported in $\Omega$. Define
\begin{equation}\label{def of mua}
\mu_a := k\sup _{\mathbf{x} \in \Omega} \int_{\Omega}\|G(\mathbf{x},\mathbf{x}^{\prime})\|_1 \,d\mathbf{x}^{\prime}.
\end{equation}
If
\[
\|V\|_{L^{\infty}(\Omega)}\leq \frac{1}{\mu_a},
\]
then \eqref{antichiral case} admits a unique solution in $C(\Omega; \mathbb{C}^2)$.
\end{thm}

\begin{proof}
As $x\to0$, the Hankel functions satisfy $H^{(1)}_0(x)=\mathcal{O}(\log x)$ and $H^{(1)}_1(x)=\mathcal{O}(x^{-1})$, which ensures integrability of the kernel in \eqref{Def of green} over the bounded domain $\Omega$ and hence that $\mu_a$ is well defined.

Define $\Phi_a:C(\Omega;\mathbb{C}^2)\to C(\Omega;\mathbb{C}^2)$ by
\[
(\Phi_a \psi_s)(\mathbf{x})
:=
-k\int_{\Omega}G(\mathbf{x},\mathbf{x}^{\prime})\,V(\mathbf{x}^{\prime})
\Bigl(\psi_0(\mathbf{x}^{\prime})+\psi_s(\mathbf{x}^{\prime})\Bigr)\,d\mathbf{x}^{\prime}.
\]
Under the smallness condition $\|V\|_{L^\infty(\Omega)}\le 1/\mu_a$, $\Phi_a$ is contracting in the supremum norm, and existence and uniqueness follow from Banach's fixed point theorem.
\end{proof}
In Appendix \ref{Quantitative Estimate}, we give a quantitative estimate of $\mu_a$ assuming that $\Omega$ is contained in a circle with radius $R:=\sqrt{L_X^2+L_y^2}$.

\section{Born and inverse Born series}
We now consider the inverse problems. In the chiral model, the data consists of measurements of the scattered field $\psi_s(L_x,y)$, and the objective is to reconstruct the scattering potential $V$. In the anti-chiral model, the data consists of boundary measurements $\psi_s|_{\partial\Omega}$, and the goal is also to recover $V$.
\subsection{Born series}
We first construct the Born series for chiral model. Iterating \eqref{iter of chiral} yields the series
    \begin{align*}
\psi_s(x,y)
=&-\mathrm{i}k \int_0^x T(x-x_1)\,V(x_1,y)\,\psi_0(x_1,y)\,dx_1 \\
&+(-\mathrm{i}k)^2 \int_0^x T(x-x_1)\,V(x_1,y)
     \int_0^{x_1} T(x_1-x_2)\,V(x_2,y)\,\psi_0(x_2,y)\,dx_2  dx_1 \\
&\;-\cdots .
\end{align*}
Accordingly, the Born series for the chiral model is defined by
\[
    \psi_s=K_1(V)+K_2(V,V)+K_3(V,V,V)+\cdots,
\]
where the multilinear operators $K_m:(L^{\infty}(\Omega))^m\to L^2([0,L_y])$ are given by
\begin{equation}\label{chiral def of km}
\begin{aligned}
K_m(V_1,\ldots,V_m)
={}&(-\mathrm{i}k)^m
\int_0^{L_x} T(L_x-x_1)\,V_1(x_1,\cdot) \\
&\quad \times \int_0^{x_1} T(x_1-x_2)\,V_2(x_2,\cdot) \\
&\quad \times \cdots \\
&\quad \times \int_0^{s_{m-1}} T(s_{m-1}-x_m)\,V_m(x_m,\cdot)\,\psi_0(x_m,\cdot)\,
dx_m\cdots dx_2\,dx_1 .
\end{aligned}
\end{equation}
Using the fact that $\|T\|_{Y\to Y}=1$, we immediately obtain the convergence condition for the Born series:
\begin{lem}
If $\|V\|_{L^{\infty}(\Omega)}<\frac{1}{kL_x}$, then the Born series for the chiral model converges.
\end{lem}

For a multilinear operator $K_m:X^m\to Y$ between Banach spaces, we define
\[
    |K_m|_{\infty}=\sup_{n_1, \ldots, n_m \neq 0}
    \frac{\bigl\|K_m(V_1, \ldots, V_m)\bigr\|_{Y}}{\|V_1\|_{X}\cdots\|V_m\|_{X}}.
\]

\begin{thm}\label{bound of forward chiral}
Let $\mu_c:=kL_x$ and                        $\nu_c:=\|g\|_{Y}$. With $X=L^{\infty}(\Omega)$ and $Y=L^2([0,L_y])$, one has
\[
    |K_m|_{\infty}\leq \mu_c^m\,\nu_c.
\]
\end{thm}

\begin{proof}
From \eqref{chiral def of km} and $\|T\|_{Y\to Y}=1$, we obtain
\begin{align*}
\frac{\|K_m(V_1,\cdots,V_m)\|_Y}{\|V_1\|_X\cdots\|V_m\|_X}
&\le |(\mathrm{i}k)^m| \int_0^{L_x}\int_0^{x_1}\cdots\int_0^{x_{m-1}}\|\psi_0(x_m,\cdot)\|_Y\,dx_m\cdots dx_2\,dx_1\\
&= k^m\int_0^{L_x}\int_0^{x_1}\cdots\int_0^{x_{m-1}}\|g\|_Y\,dx_m\cdots dx_2\,dx_1\\
&=  k^m\|g\|_Y \int_0^{L_x}\int_0^{x_1}\cdots\int_0^{x_{m-1}}1\,dx_m\cdots dx_2\,dx_1\\
&= \frac{(kL_x)^m}{m!}\|g\|_Y
\le \mu_c^m \nu_c,
\end{align*}
which yields the claim.
\end{proof}

Similarly, for the anti-chiral model, iterating \eqref{iter of anti} leads to
\begin{equation}\label{def of bs for anti}
    \begin{aligned}
    \psi_s=&-k\int_{\Omega}G(\mathbf{x},\mathbf{x}_1)\,V(\mathbf{x}_1)\,\psi_0(\mathbf{x}_1)\,d\mathbf{x}_1\\
    &+k^2\int_{\Omega}G(\mathbf{x},\mathbf{x}_1)\,V(\mathbf{x}_1)\int_{\Omega}G(\mathbf{x}_1,\mathbf{x}_2)\,V(\mathbf{x}_2)\,\psi_0(\mathbf{x}_2)\,d\mathbf{x}_2\,d\mathbf{x}_1+\cdots.
\end{aligned}
\end{equation}
The corresponding Born series is
\[
    \psi_s=K_1(V)+K_2(V,V)+K_3(V,V,V)+\cdots,
\]
where $K_m:(L^{\infty}(\Omega))^m\to C(\partial\Omega)$ is defined by
\begin{equation}\label{antichiral def of km}
\begin{aligned}
K_m(V_1,\ldots,V_m)
={}& (-k)^m
\int_{\Omega}
G(\mathbf{x},\mathbf{x}_1)\,V_1(\mathbf{x}_1)
\int_{\Omega}
G(\mathbf{x}_1,\mathbf{x}_2)\,V_2(\mathbf{x}_2) \\
&\quad \times \cdots\times
\int_{\Omega}
G(\mathbf{x}_{m-1},\mathbf{x}_m)\,
V_m(\mathbf{x}_m)\,\psi_0(\mathbf{x}_m)\,
d\mathbf{x}_m\cdots d\mathbf{x}_1 .
\end{aligned}
\end{equation}
Based on the definition of $\mu_a$ and the Born series \eqref{def of bs for anti}, the following two results are obtained without difficulty.

\begin{lem}
If $\|V\|_{L^{\infty}(\Omega)}<\frac{1}{\mu_a}$, then the Born series for the anti-chiral model converges.
\end{lem}

\begin{thm}\label{bound of forward anti}
Let $X=L^{\infty}(\Omega)$, $Y=C(\partial\Omega)$, and set $\nu_a:=\|\psi_0\|_{X}$. Then
\[
    |K_m|_{\infty}\leq \mu_a^m\,\nu_a.
\]
\end{thm}

\subsection{Inverse Series}
The inverse problem is to recover the scattering potential $V$ from measurements. To this end, we employ the inverse Born series (IBS), whose definition and analysis are developed in \cite{hoskinsAnalysisInverseBorn2022}. The IBS reconstructs the potential by means of the expansion
\[
    \tilde{V}
    =
    \mathcal{K}_1(\psi_s)+\mathcal{K}_2(\psi_s)+\mathcal{K}_3(\psi_s)+\cdots,
\]
where the operators $\{\mathcal{K}_m\}_{m\ge1}$ are defined recursively by
\begin{align*}
& \mathcal{K}_1(\psi_s)=K_1^{+}(\psi_s), \\
& \mathcal{K}_2(\psi_s)=-\mathcal{K}_1\!\left(K_2\!\left(\mathcal{K}_1(\psi_s), \mathcal{K}_1(\psi_s)\right)\right), \\
& \mathcal{K}_m(\psi_s)=-\sum_{n=2}^m \sum_{i_1+\cdots+i_n=m}
\mathcal{K}_1 K_n\!\left(\mathcal{K}_{i_1}(\psi_s), \ldots, \mathcal{K}_{i_n}(\psi_s)\right).
\end{align*}
Here $K_1^{+}$ denotes the bounded inverse of $K_1$ when it exists and otherwise its pseudoinverse.

We next recall the radius of convergence and an error estimate for the IBS, based on \cite{hoskinsAnalysisInverseBorn2022}.

\begin{thm}\label{thm:ibs_convergence_error}
For the chiral model, set
\[
(X,\mu,\nu)=(L^{\infty}(\Omega),\mu_c,\nu_c),
\]
where $\mu_c,\nu_c$ are given in Theorem~\ref{bound of forward chiral}. For the
anti-chiral model, set
\[
(X,\mu,\nu)=(L^{\infty}(\Omega),\mu_a,\nu_a),
\]
where $\nu_a$ is given in Theorem~\ref{bound of forward anti}, and $\mu_a$ is
defined in \eqref{def of mua}. Define
\[
C := \max\left\{2,\ \|\mathcal{K}_1\|\,\nu\right\},
\qquad
r := \frac{1}{2\mu}\left(\sqrt{16C^2+1}-4C\right).
\]
If
\[
\|\mathcal{K}_1 \psi_s\|_{X} < r,
\]
then the inverse Born series associated with the chiral or anti-chiral model converges
in $X$.

Moreover, suppose that the corresponding Born series also converges. Let $\tilde{V}$
denote the sum of the inverse Born series and set
\[
V_1=\mathcal{K}_1\psi_s .
\]
Define
\[
\mathcal{M}:=\max\{\|V\|_X,\|\tilde{V}\|_X\},
\]
and assume that
\[
\mathcal{M}
<
\frac{1}{\mu}
\left(
1-\sqrt{\frac{\nu\|\mathcal{K}_1\|}{1+\nu\|\mathcal{K}_1\|}}
\right).
\]
Then the approximation error satisfies
\[
\begin{aligned}
\left\|V-\sum_{m=1}^N \mathcal{K}_m(\psi_s)\right\|_X
\leqslant\,
& M\left(\frac{\|V_1\|_X}{r}\right)^{N+1}
\frac{1}{1-\frac{\|V_1\|_X}{r}} \\
&+
\left(
1-\frac{\nu\|\mathcal{K}_1\|}{(1-\mu \mathcal{M})^2}
+\nu\|\mathcal{K}_1\|
\right)^{-1}
\left\|\left(I-\mathcal{K}_1 K_1\right)V\right\|_X,
\end{aligned}
\]
where
\[
M=\frac{2\mu}{\sqrt{16C^2+1}}.
\]
\end{thm}

We also consider the reduced inverse Born series (RIBS), defined by
\begin{align*}
& \mathcal{K}^r_1(\psi_s)=K_1^{+}(\psi_s), \\
& \mathcal{K}^r_2(\psi_s)=-\mathcal{K}^r_1\!\left(K_2\!\left(\mathcal{K}^r_1(\psi_s), \mathcal{K}^r_1(\psi_s)\right)\right), \\
& \mathcal{K}^r_m(\psi_s)= - \mathcal{K}^r_1 K_2\!\left(\mathcal{K}_{m-1}^r(\psi_s), \mathcal{K}^r_1(\psi_s)\right).
\end{align*}
The RIBS was introduced in \cite{2022reduced}, where cancellations among IBS terms were proved in settings with a single source or a single detector. Although such cancellations are not generally expected for multiple sources or detectors, our numerical experiments indicate that a similar cancellation phenomenon can still be observed.

\section{Numerical methods}
In this section, we detail the numerical algorithms developed for the forward and inverse problems.

\subsection{Forward solver for the chiral model}

To solve \eqref{chiral case}, we discretize the space-like variable $y$
using a finite-difference scheme and treat the chiral Dirac equation as an
evolution problem in the time-like variable $x$.
The interval $(y_{\min},y_{\max})$ is discretized with a uniform grid
\[
y_j = y_{\min} + j\,\Delta y, \quad j=0,\dots,N_y-1,
\]
where $\Delta y=(y_{\max}-y_{\min})/(N_y-1)$.
We impose homogeneous boundary conditions on different components at
opposite boundaries, namely
\[
\psi_1(x,y_0)=0,
\quad
\psi_2(x,y_{N_y-1})=0.
\]
For each fixed $x$, the discrete field is represented by stacking all
interior values of the first component followed by those of the second
component,
\[
\psi(x)
=
\begin{pmatrix}
\psi_1(x)\\
\psi_2(x)
\end{pmatrix}
\in \mathbb C^{2(N_y-1)},
\]
where
\[
\psi_1(x)
=
\big(\psi_1(x,y_1),\dots,\psi_1(x,y_{N_y-1})\big)^{T},\,
\psi_2(x)
=
\big(\psi_2(x,y_0),\dots,\psi_2(x,y_{N_y-2})\big)^{T}.
\]

The differential operator
$A=\alpha\,\partial_y-\mathrm{i}k\beta$
is approximated on the interior grid by a sparse matrix
$L\in\mathbb C^{2(N_y-1)\times 2(N_y-1)}$.
In block form, the discrete operator reads
\begin{equation*}
L
=
\begin{pmatrix}
-\mathrm{i} kI & D^- \\
- D^+ & \mathrm{i} kI
\end{pmatrix},
\end{equation*}
where $I$ denotes the identity matrix of size $(N_y-1)\times(N_y-1)$.
The matrices $D^-$ and $D^+$ are first-order, finite-difference operators
approximating the derivative $\partial_y$.
Specifically, $D^-$ is a backward difference operator acting on $\psi_2$,
while $D^+$ is a forward difference operator acting on $\psi_1$:
\[
(D^- v)_j = \frac{v_j - v_{j-1}}{\Delta y},
\quad
(D^+ v)_j = \frac{v_{j+1} - v_j}{\Delta y},
\quad j=1,\dots,N_y-2.
\]
To propagate the solution in the direction $x$, we employ a
Crank--Nicolson marching scheme.
Let $\psi^i$ denote the approximation at $x_i=i\,\Delta x$.
The homogeneous propagation step is of the form
\begin{equation*}
\bigl(I-\tfrac{\Delta x}{2}L\bigr)\psi^{i+1}
=
\bigl(I+\tfrac{\Delta x}{2}L\bigr)\psi^{i}.
\end{equation*}
This scheme is second-order accurate in $\Delta x$.

For scattering equations, the potential term
$-\mathrm{i}\,V(x,y)\psi$ is incorporated through a source term:
\begin{equation*}
\bigl(I-\tfrac{\Delta x}{2}L\bigr)\psi^{i+1}
=
\bigl(I+\tfrac{\Delta x}{2}L\bigr)\psi^{i}
+
\Delta x\bigl(-\mathrm{i}\,k\,V^i\,\psi^{i}\bigr),
\end{equation*}
where $V^i$ denotes the discrete potential evaluated at $x=x_i$.
At each marching step, the sparse linear system with coefficient matrix
$I-\tfrac{\Delta x}{2}L$ is solved using a precomputed sparse LU factorization.

\subsection{Forward solver for the anti-chiral model}
We now reformulate the Lippmann--Schwinger equation in terms an auxiliary spinor field $\sigma$. Following standard procedures, the scattered field is represented as
\[
\psi_s(\mathbf{x})
=
\int_{\mathbb{R}^2}
G(\mathbf{x},\mathbf{x}^{\prime})
\,\sigma(\mathbf{x}^{\prime})
\,d\mathbf{x}^{\prime}.
\]
Substituting the above into Eq.~\eqref{s antichiral case} and using $\psi=\psi_0+\psi_s$, we obtain an integral equation for $\sigma$:
\begin{equation}\label{sigma_equation}
\sigma(\mathbf{x})
+
k\,V(\mathbf{x})
\int_{\Omega}
G(\mathbf{x},\mathbf{x}^{\prime})
\,\sigma(\mathbf{x}^{\prime})
\,d\mathbf{x}^{\prime}
=
-\,k\,V(\mathbf{x})\,\psi_0(\mathbf{x}),
\quad
\mathbf{x}\in\Omega.
\end{equation}
Since $V$ is supported in $\Omega$, it follows that the unknown $\sigma$ is also supported in $\Omega$. Evidently, it is thus possible to reduce the computational domain to $\Omega$ instead of all of space. The linear system associated with \eqref{sigma_equation} is solved using the generalized minimal residual (GMRES) method. To be more specific, we introduce a uniform Cartesian grid on a rectangular box
containing $\Omega$. Let
\[
\mathbf{x}_{ij}=(x_i,y_j),\qquad i,j=1,\ldots,N,
\]
denote the grid points, with mesh size $h$ and cell area $h^2$. The auxiliary field is
approximated by its grid values
\[
\sigma_{ij}\approx \sigma(\mathbf{x}_{ij}).
\]
Using a quadrature-based discretization of the volume integral, we obtain 
\[
\sigma_{ij}
+
k\,V_{ij}
\sum_{m,n}
G(\mathbf{x}_{ij}-\mathbf{x}_{mn})\,\sigma_{mn}\,h^2
=
-\,k\,V_{ij}\psi_0(\mathbf{x}_{ij}),
\]
where $V_{ij}=V(\mathbf{x}_{ij})$. In matrix form, this system can be written as
\[
(I+kVG_h)\sigma=-kV\psi_0,
\]
where $V$ is the diagonal matrix formed by the values of the potential on the grid, and
$G_h$ denotes the discretized Green's convolution operator. Moreover, by translation invariance,
$G(\mathbf{x},\mathbf{x}^{\prime})$ depends only on $\mathbf{x}-\mathbf{x}^{\prime}$,
which means that the integral operator is of convolution type. After discretization, the resulting matrix $G_h$ has Toeplitz structure, enabling the use of fast matrix--vector products by means of the FFT. 

\subsection{Inverse solver}
To construct the pseudoinverse of the leading-order forward operator $K_1$, we employ a conjugate-gradient (CG) method applied to a Tikhonov-regularized normal equation. The IBS and RIBS reconstructions are then assembled by evaluating the corresponding inverse series terms up to a prescribed order.

\begin{algorithm}[H]
\caption{Inverse Solver using  the Inverse Born Series}
\label{alg:ibs}
\begin{algorithmic}
\Require Data $\psi_s$, number of terms $m$
\Ensure Reconstruction of scattering potential $n$

\Function{IBS}{}
  \State $n_1 \gets \mathcal{K}_1(\psi_s)$
  \For{$j=2$ to $m$}
    \State $b_j \gets \displaystyle \sum_{l=2}^{m}\ \sum_{i_1+\cdots+i_l=m}
      K_l\!\bigl(V_{i_1},\cdots, V_{i_l}\bigr)$
    \State $n_j \gets -\mathcal{K}_1(b_j)$
  \EndFor
  \State \Return $V_1+\cdots+V_m$
\EndFunction

\Statex
\end{algorithmic}
\end{algorithm}
\begin{algorithm}[H]
\caption{Inverse Solver using the Reduced Inverse Born Series}
\label{alg:ribs}
\begin{algorithmic}
\Require Data $\psi_s$, number of terms $m$
\Ensure Reconstruction of scattering potential $n$

\Function{IBS}{}
  \State $n_1 \gets \mathcal{K}_1(\psi_s)$
  \For{$j=2$ to $m$}
    \State $b_j \gets 
      K_2\!\bigl(n_{j-1},n_1\bigr)$
    \State $n_j \gets -\mathcal{K}_1(b_j)$
  \EndFor
  \State \Return $V_1+\cdots+V_m$
\EndFunction

\Statex
\end{algorithmic}
\end{algorithm}

\section{Numerical Results}
In this section, we present the results of our numerical experiments.

\begin{figure}[htbp]
\centering
\begin{subfigure}[b]{0.9\textwidth}
\includegraphics[width=\textwidth]{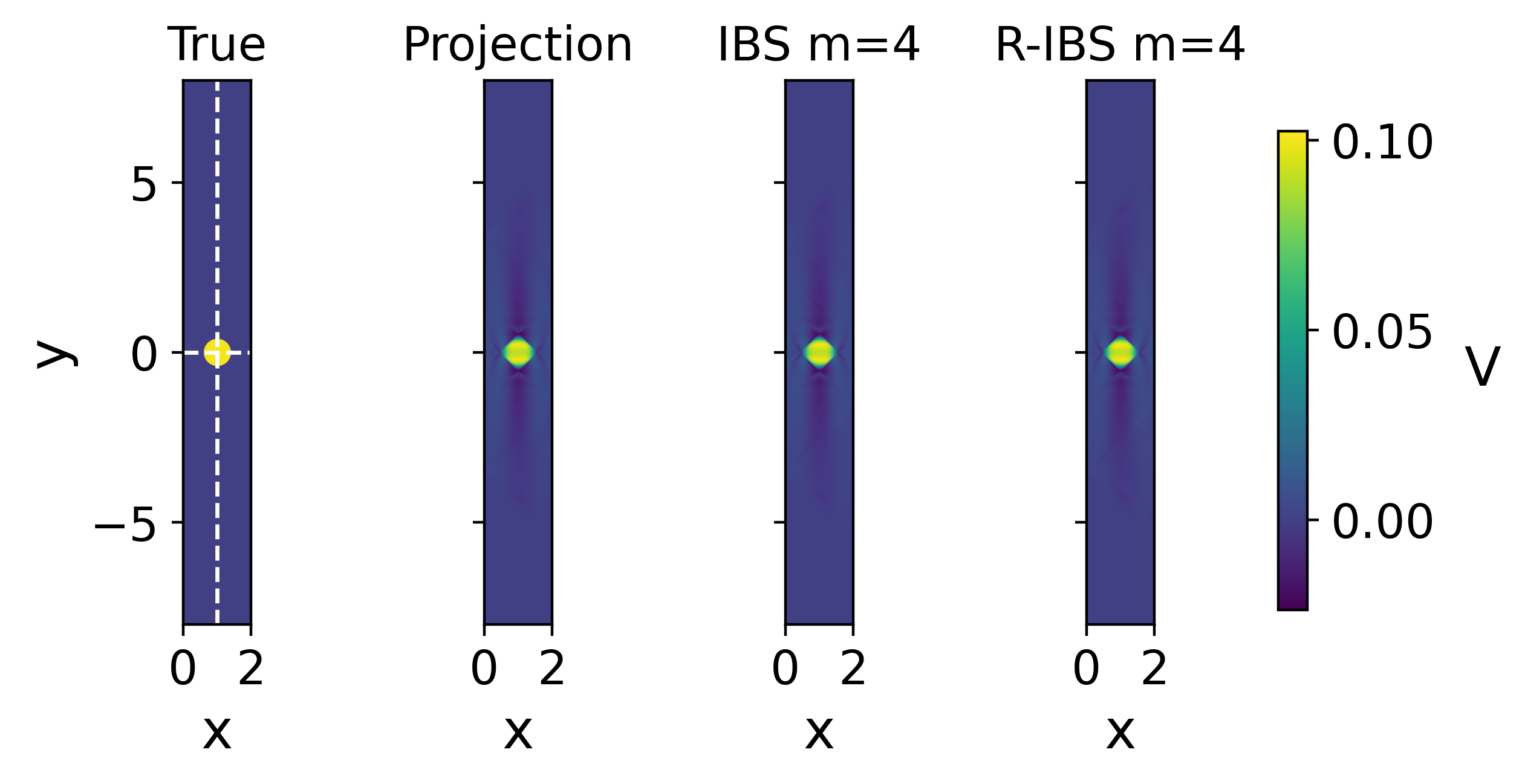}
\caption{Reconstructions of $V$}
\label{fig:chiral-disk-low-global}
\end{subfigure}

\vspace{0.0em}

\begin{subfigure}[b]{0.43\textwidth}
\includegraphics[width=\textwidth]{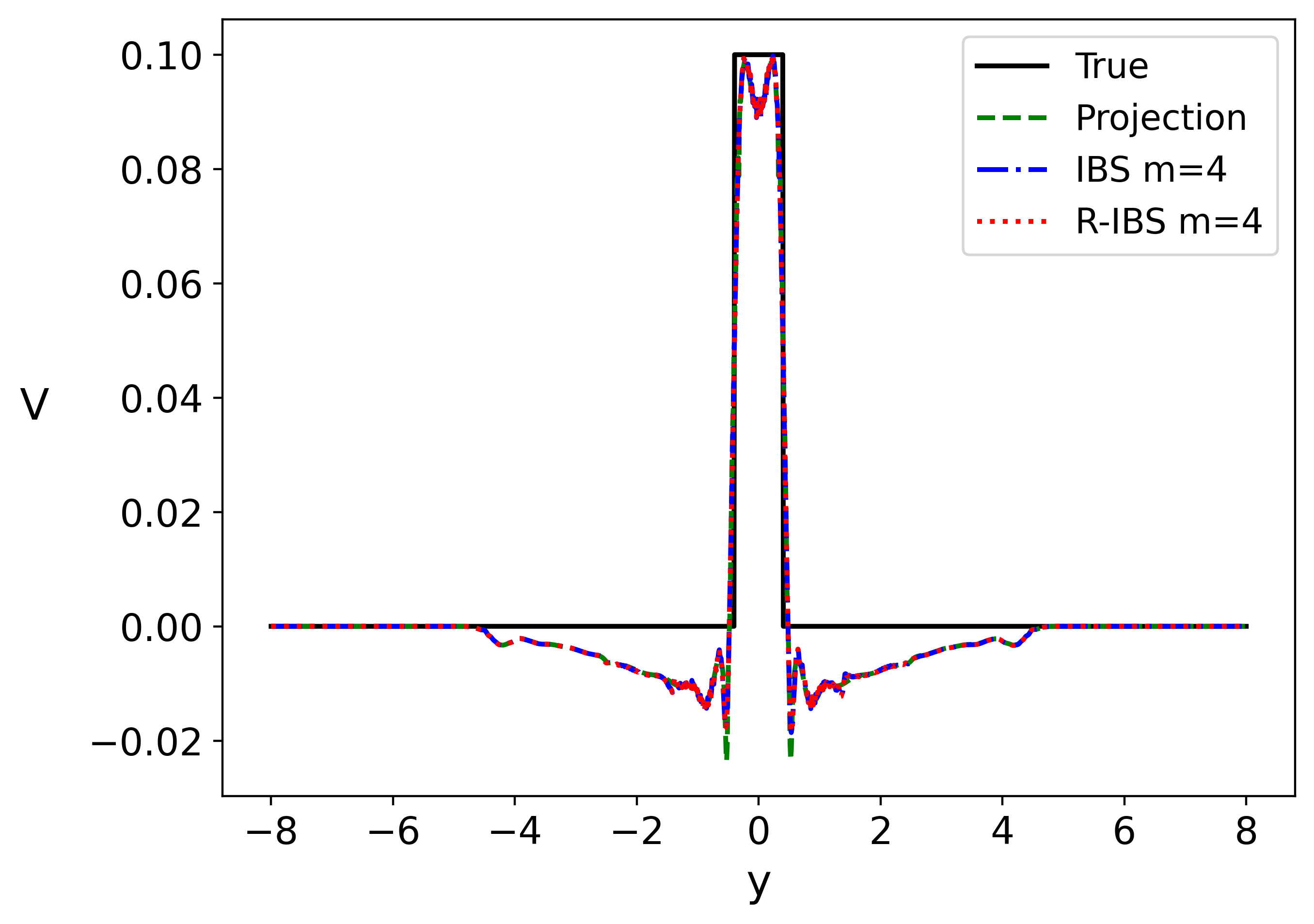}
\caption{Cross section at $x=1.0$}
\label{fig:chiral-disk-low-slice-x}
\end{subfigure}\hfill
\begin{subfigure}[b]{0.43\textwidth}
\includegraphics[width=\textwidth]{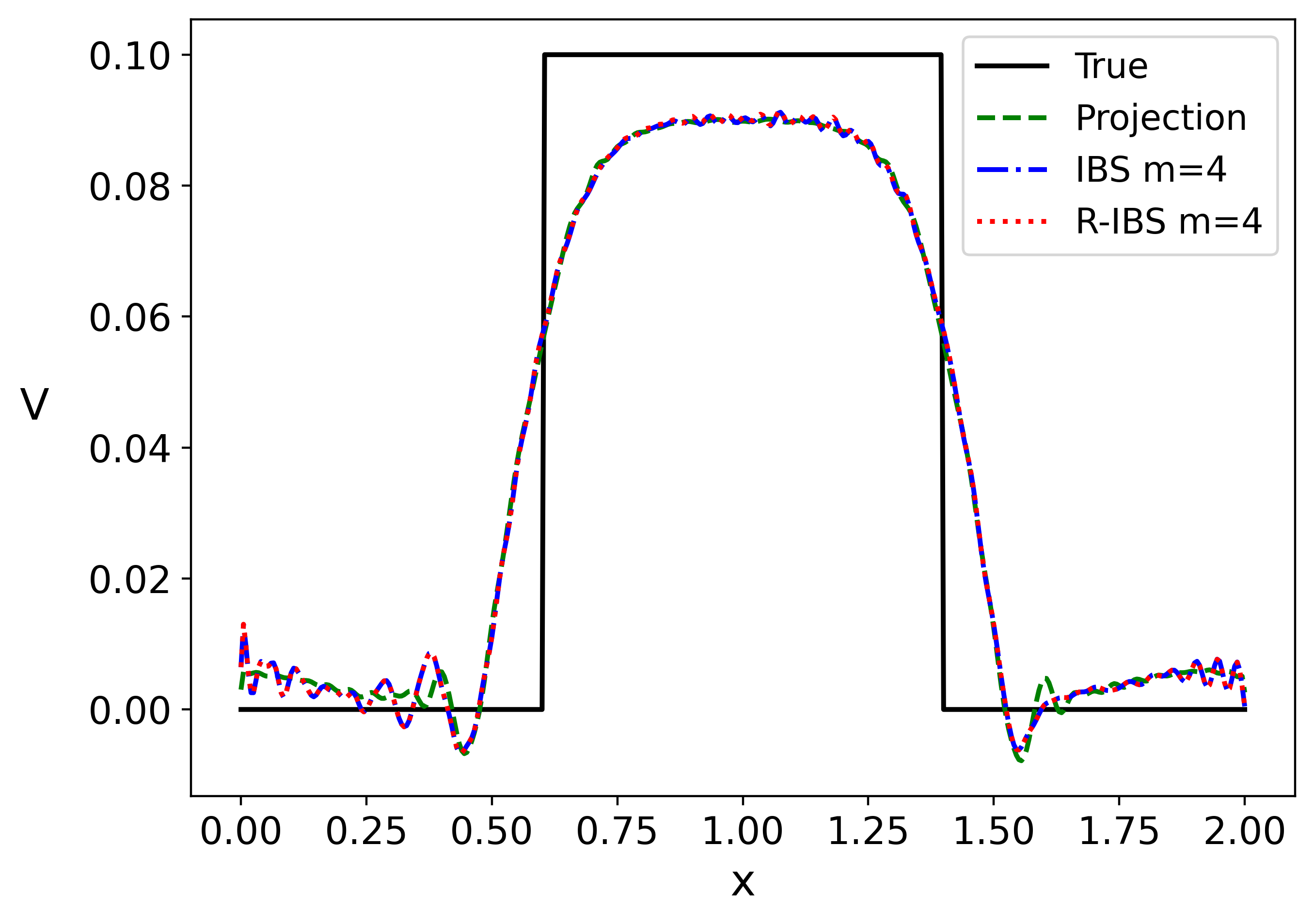}
\caption{Cross section at $y=0.0$}
\label{fig:chiral-disk-low-slice-y}
\end{subfigure}

\vspace{0.0em}

\begin{subfigure}[b]{0.43\textwidth}
\includegraphics[width=\textwidth]{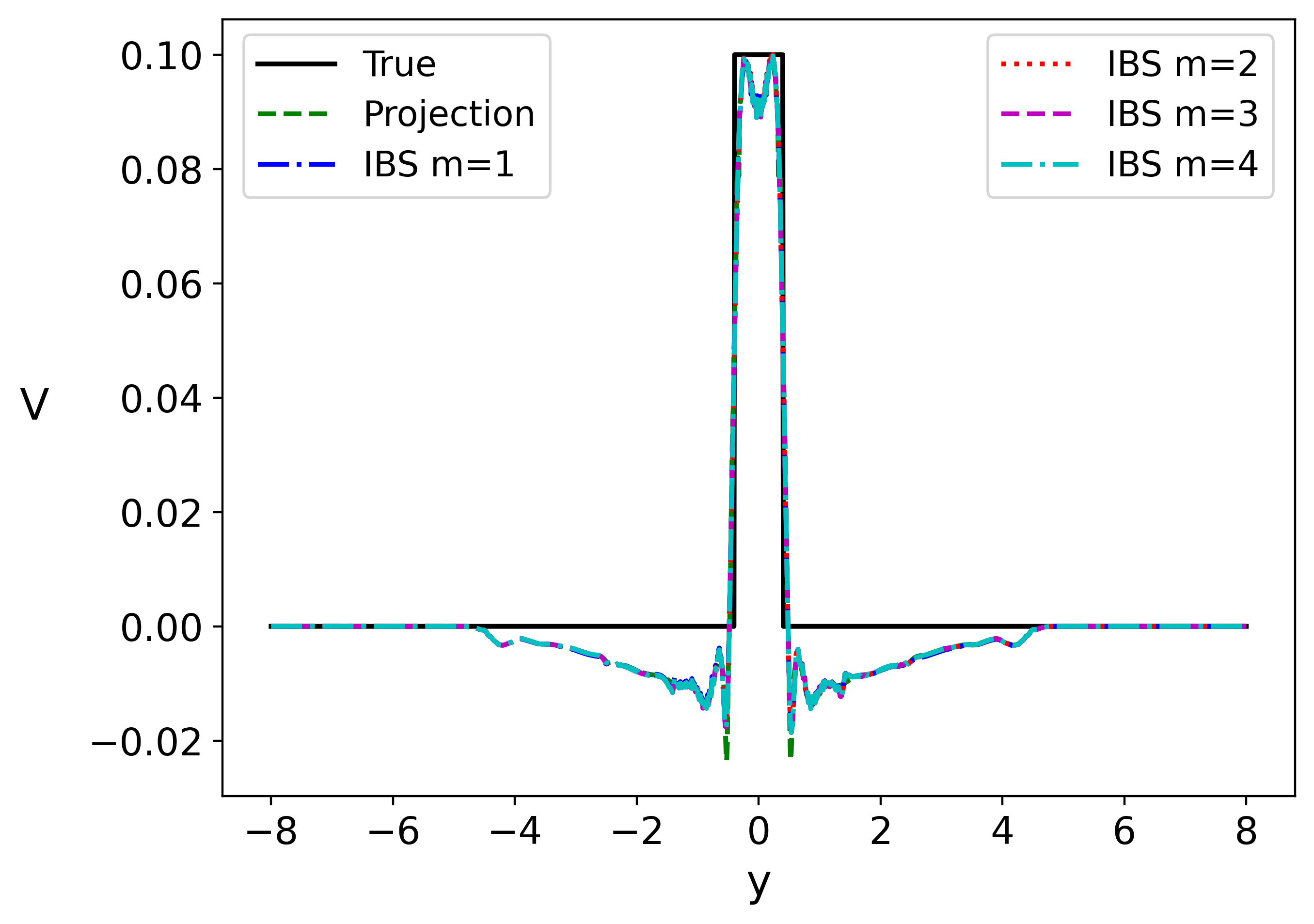}
\caption{IBS cross section at $x=1.0$}
\label{fig:chiral-disk-low-ibs-x}
\end{subfigure}\hfill
\begin{subfigure}[b]{0.43\textwidth}
\includegraphics[width=\textwidth]{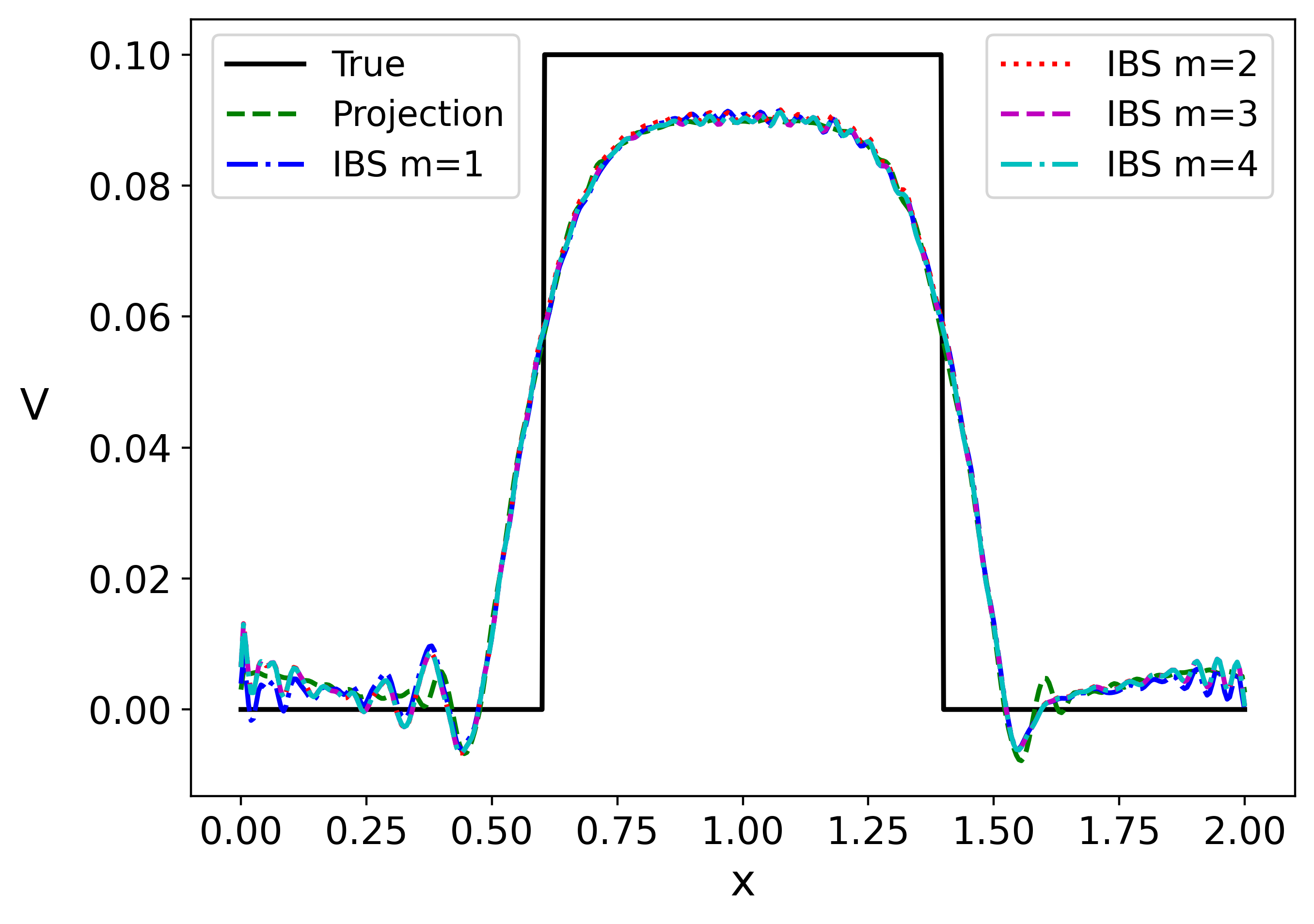}
\caption{IBS cross section at $y=0.0$}
\label{fig:chiral-disk-low-ibs-y}
\end{subfigure}

\scriptsize
\setlength{\tabcolsep}{3pt}
\begin{tabular}{@{}l|ccccc|cccc@{}}
\toprule
 & Projection & IBS1 & IBS2 & IBS3 & IBS4 &
 RIBS1 & RIBS2 & RIBS3 & RIBS4 \\
\midrule
Relative error &
0.3639 &   
0.3750 &   
0.3724 &   
0.3728 &   
0.3728 &   
0.3750 &   
0.3724 &   
0.3727 &   
0.3727 \\  
\bottomrule
\end{tabular}

\caption{Reconstructions of low contrast disk (chiral model)}
\label{fig:chiral-disk-low}
\end{figure}
\begin{figure}[htbp]
\centering
\begin{subfigure}[b]{0.9\textwidth}
\includegraphics[width=\textwidth]{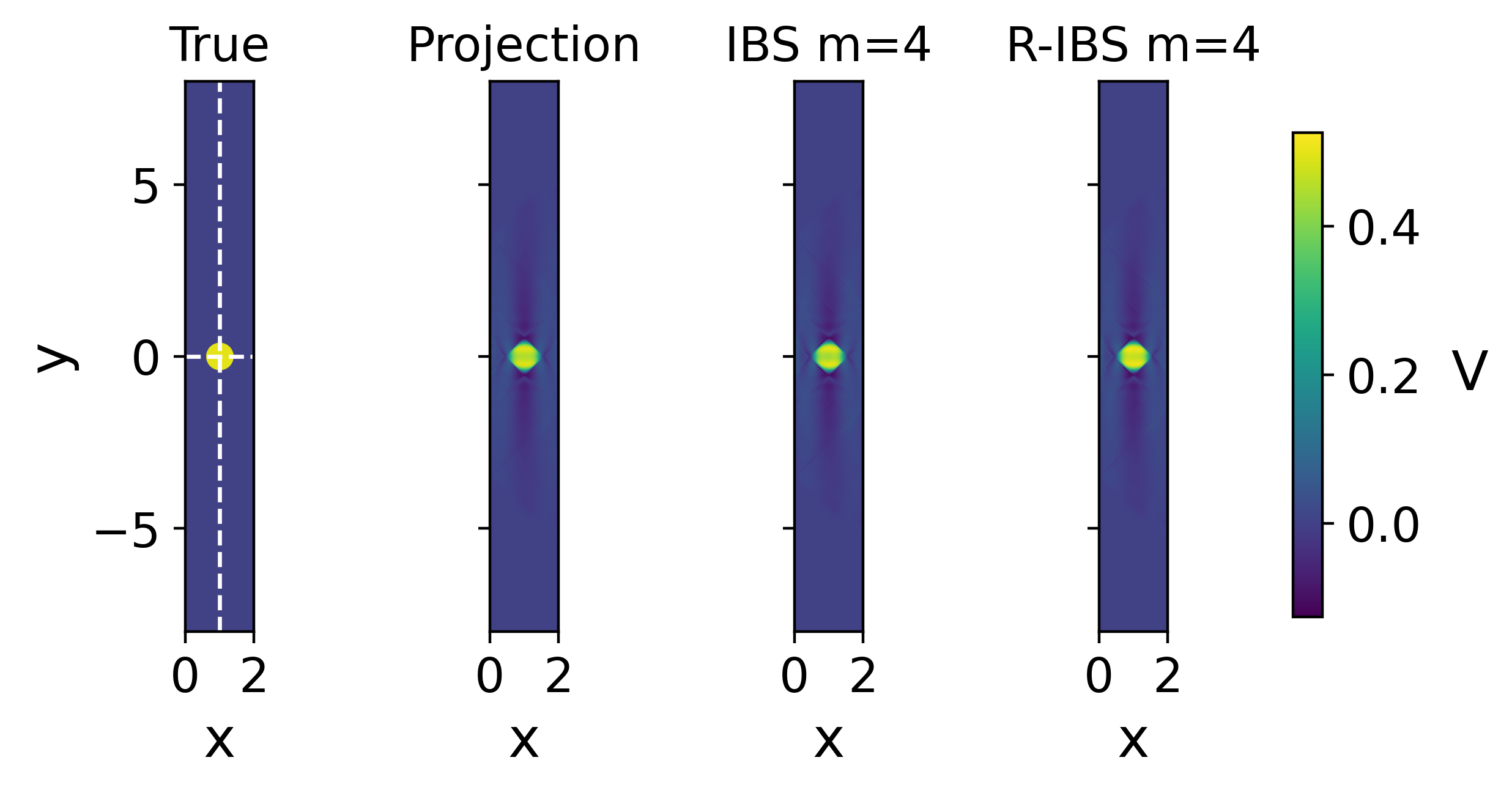}
\caption{Reconstructions of $V$}
\label{fig:chiral-disk-mid-global}
\end{subfigure}

\vspace{0.0em}

\begin{subfigure}[b]{0.43\textwidth}
\includegraphics[width=\textwidth]{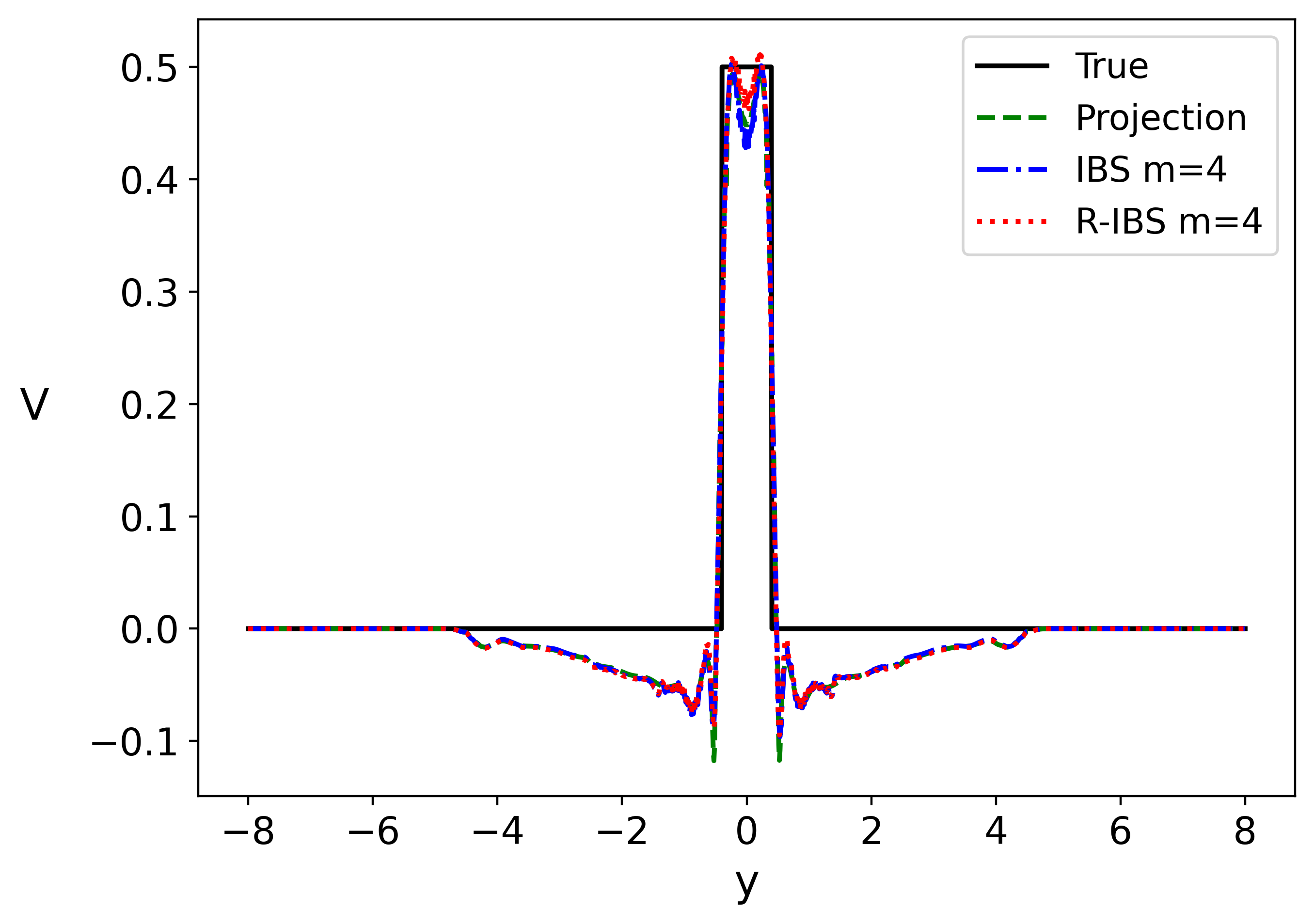}
\caption{Cross section at $x=1.0$}
\label{fig:chiral-disk-mid-slice-x}
\end{subfigure}\hfill
\begin{subfigure}[b]{0.43\textwidth}
\includegraphics[width=\textwidth]{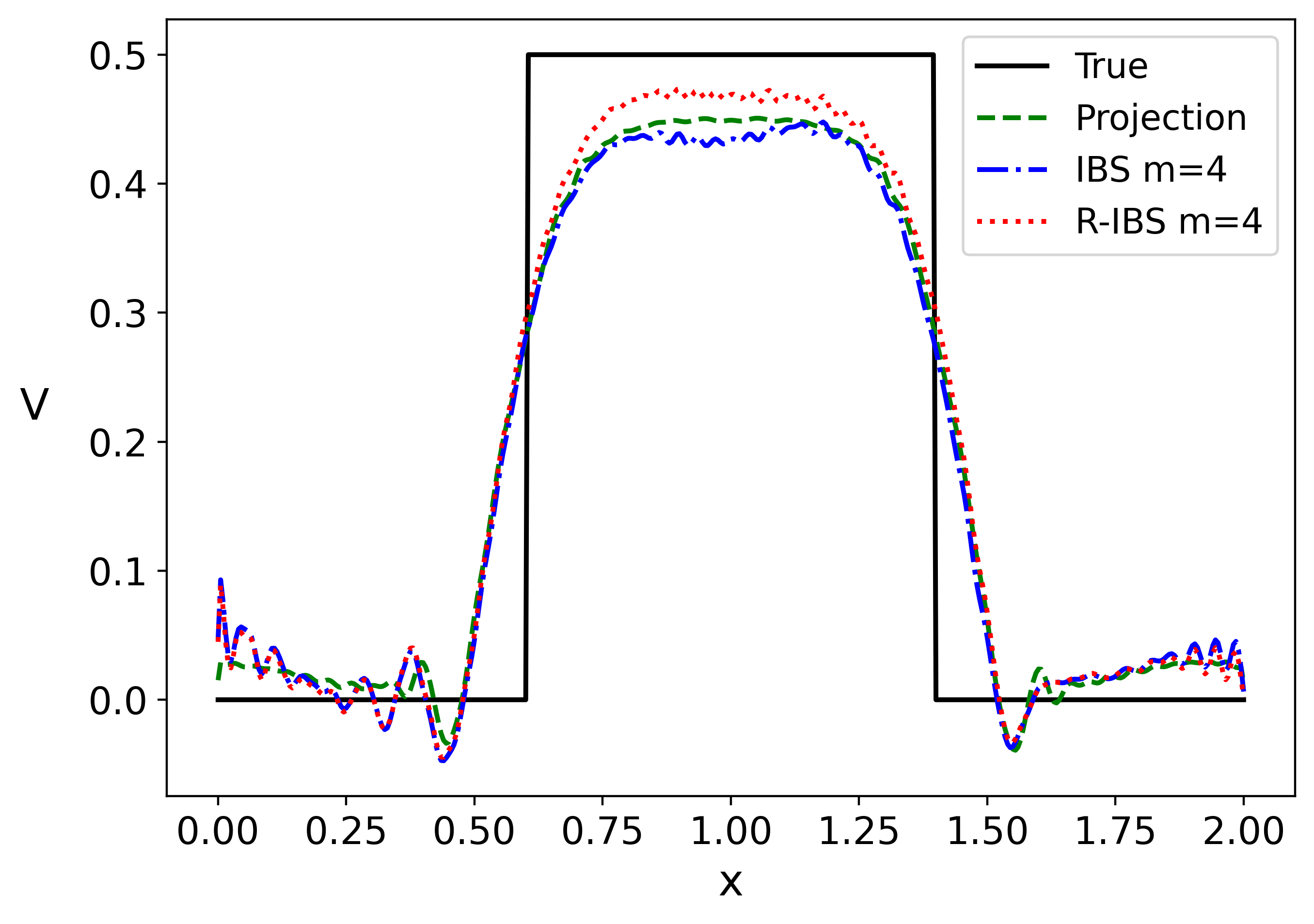}
\caption{Cross section at $y=0.0$}
\label{fig:chiral-disk-mid-slice-y}
\end{subfigure}

\vspace{0.0em}

\begin{subfigure}[b]{0.43\textwidth}
\includegraphics[width=\textwidth]{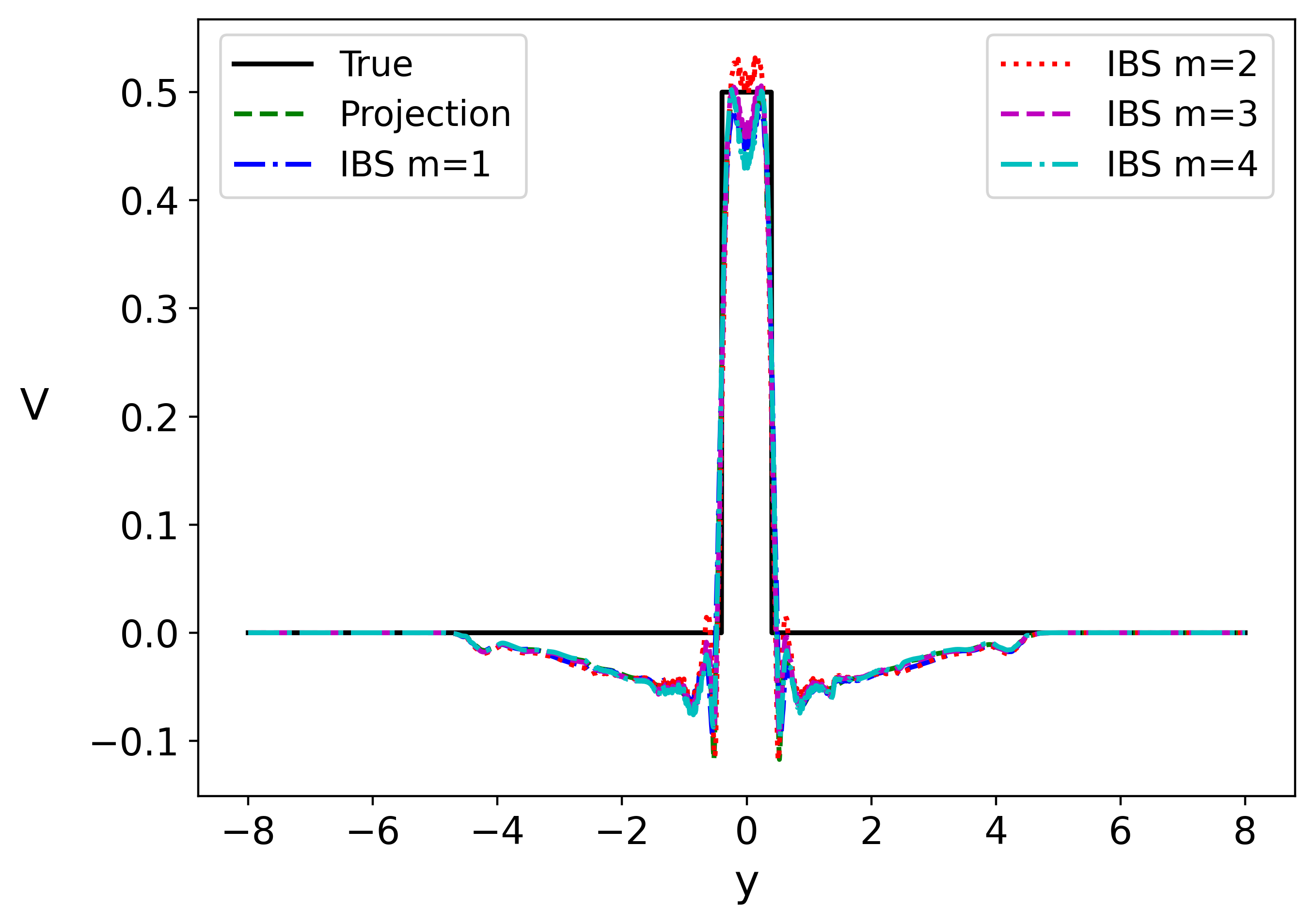}
\caption{IBS cross section at $x=1.0$}
\label{fig:chiral-disk-mid-ibs-x}
\end{subfigure}\hfill
\begin{subfigure}[b]{0.43\textwidth}
\includegraphics[width=\textwidth]{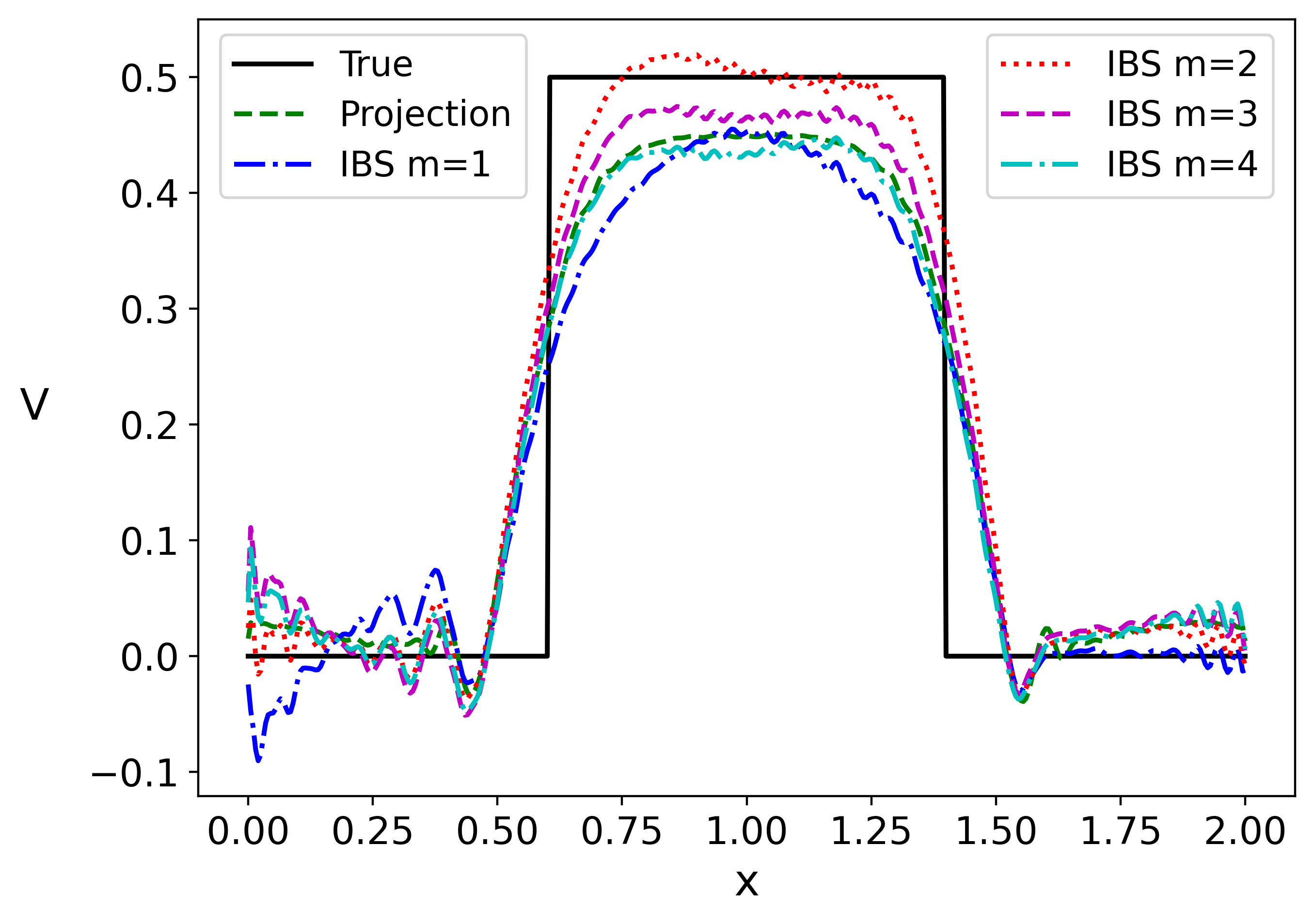}
\caption{IBS cross section at $y=0.0$}
\label{fig:chiral-disk-mid-ibs-y}
\end{subfigure}

\scriptsize
\setlength{\tabcolsep}{3pt}
\begin{tabular}{@{}l|ccccc|cccc@{}}
\toprule
 & Projection & IBS1 & IBS2 & IBS3 & IBS4 &
 RIBS1 & RIBS2 & RIBS3 & RIBS4 \\
\midrule
Relative error &
0.3639 &   
0.4801 &   
0.3916 &   
0.3831 &   
0.3812 &   
0.4801 &   
0.3916 &   
0.3796 &   
0.3752 \\  
\bottomrule
\end{tabular}

\caption{Reconstructions of medium contrast disk (chiral model)}
\label{fig:chiral-disk-mid}
\end{figure}
\begin{figure}[htbp]
\centering
\begin{subfigure}[b]{0.9\textwidth}
\includegraphics[width=\textwidth]{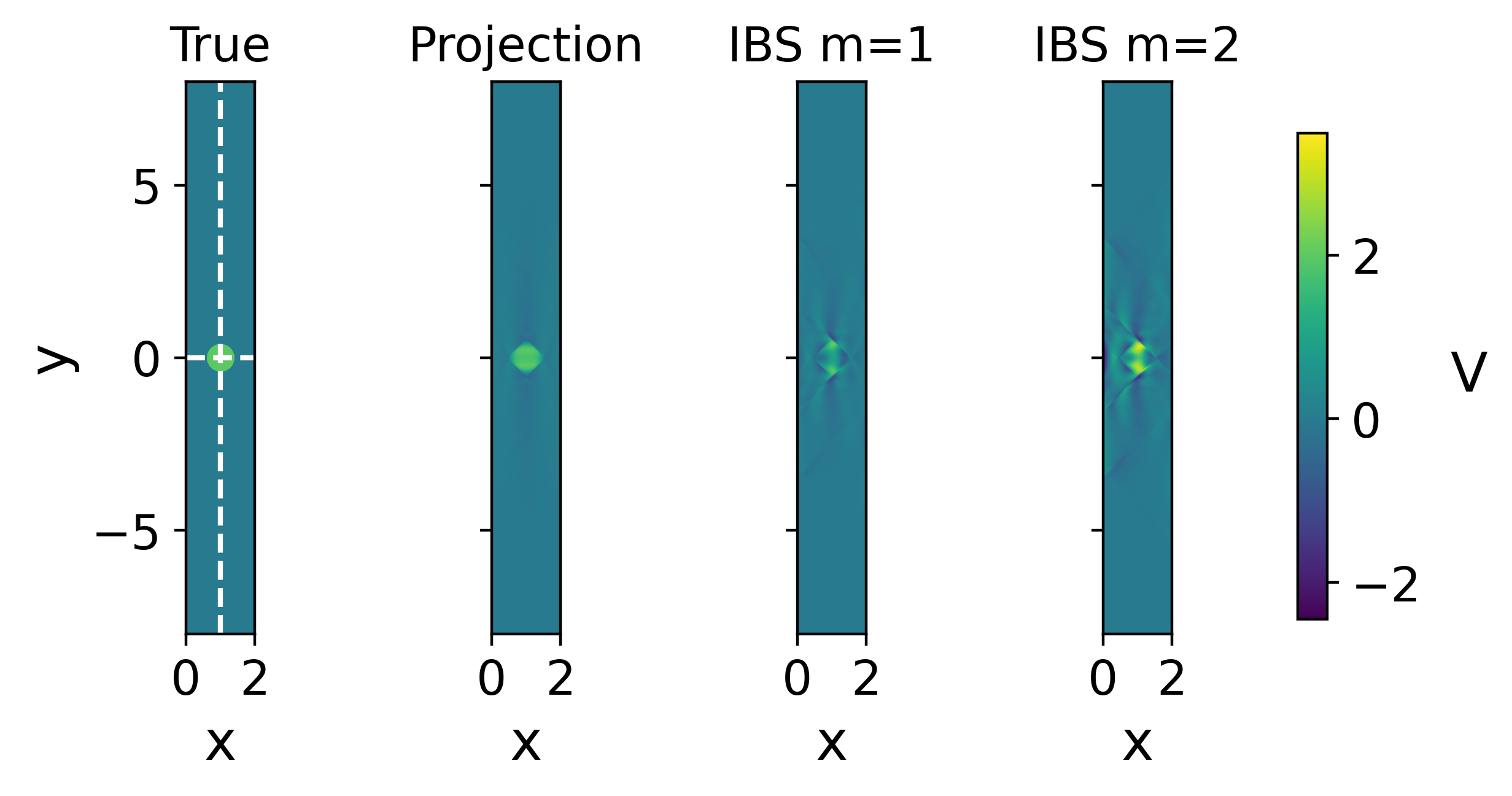}
\caption{Reconstructions of $V$}
\label{fig:chiral-disk-high-global}
\end{subfigure}

\vspace{0.6em}

\begin{subfigure}[b]{0.45\textwidth}
\includegraphics[width=\textwidth]{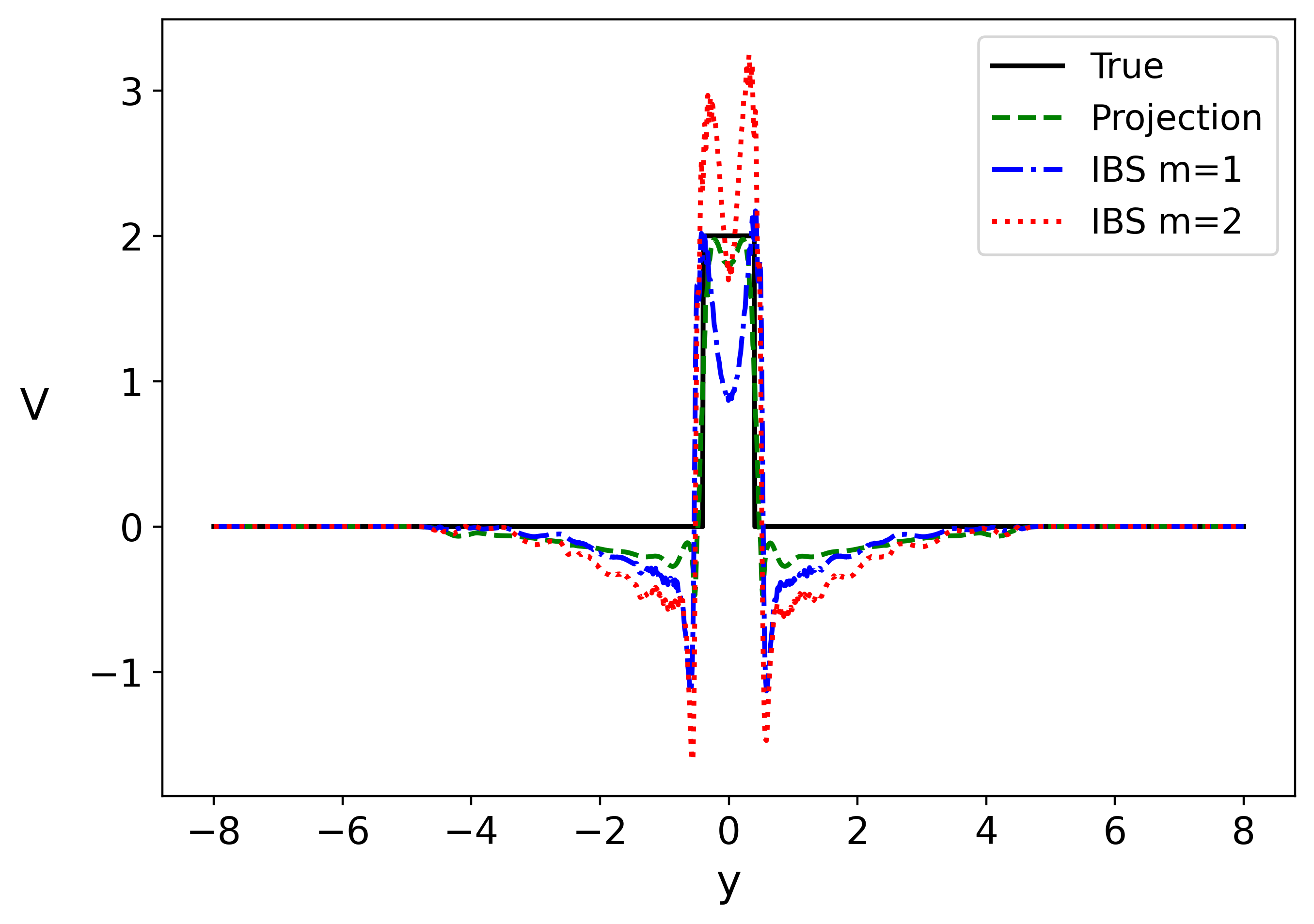}
\caption{Cross section at $x=1.0$}
\label{fig:chiral-disk-high-slice-x}
\end{subfigure}

\vspace{0.6em}

\begin{subfigure}[b]{0.45\textwidth}
\includegraphics[width=\textwidth]{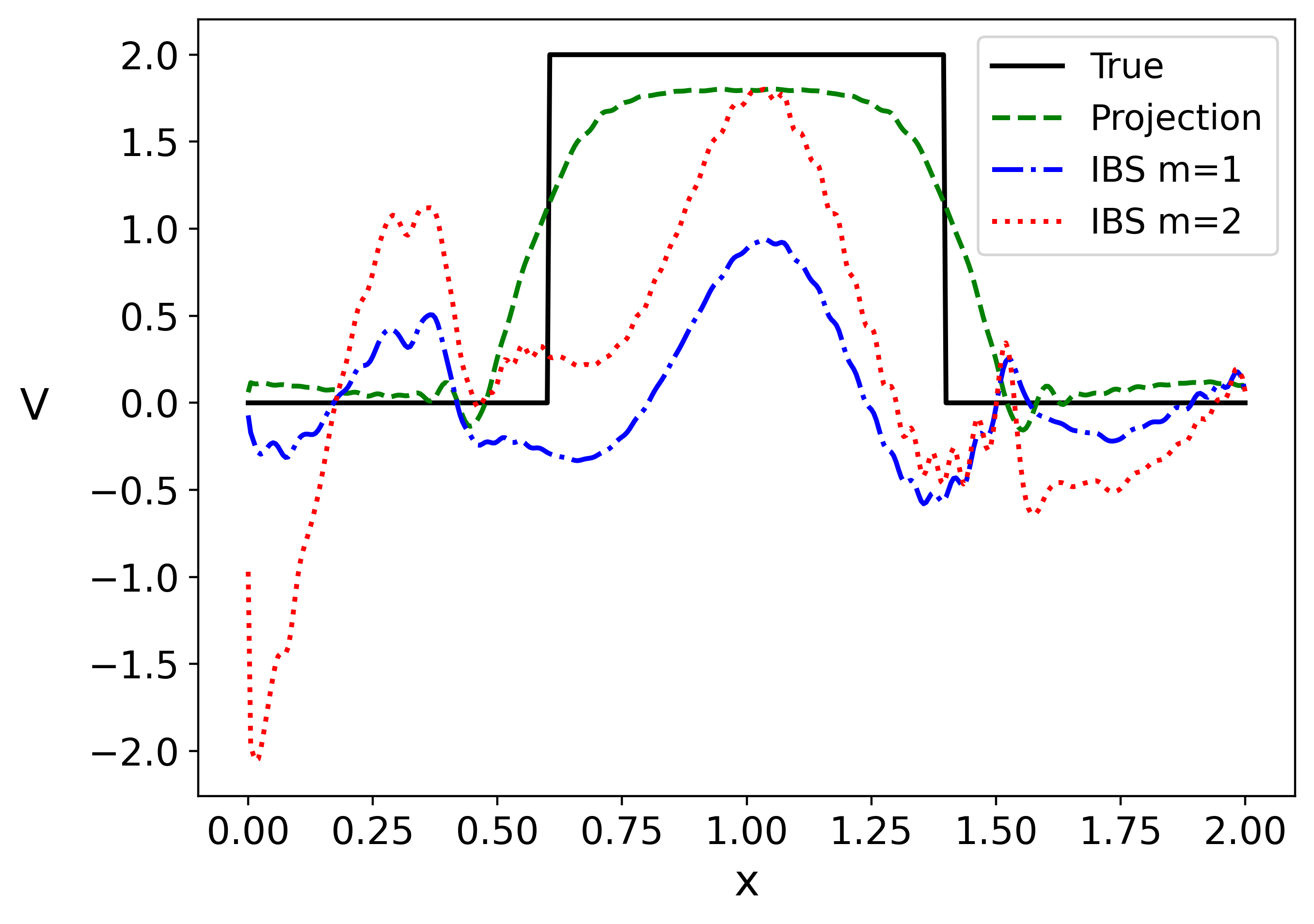}
\caption{Cross section at $y=0.0$}
\label{fig:chiral-disk-high-slice-y}
\end{subfigure}

\caption{Reconstructions of high contrast disk (chiral model)}
\label{fig:chiral-disk-high}
\end{figure}
\begin{figure}[htbp]
\centering
\begin{subfigure}[b]{0.88\textwidth}
\includegraphics[width=\textwidth]{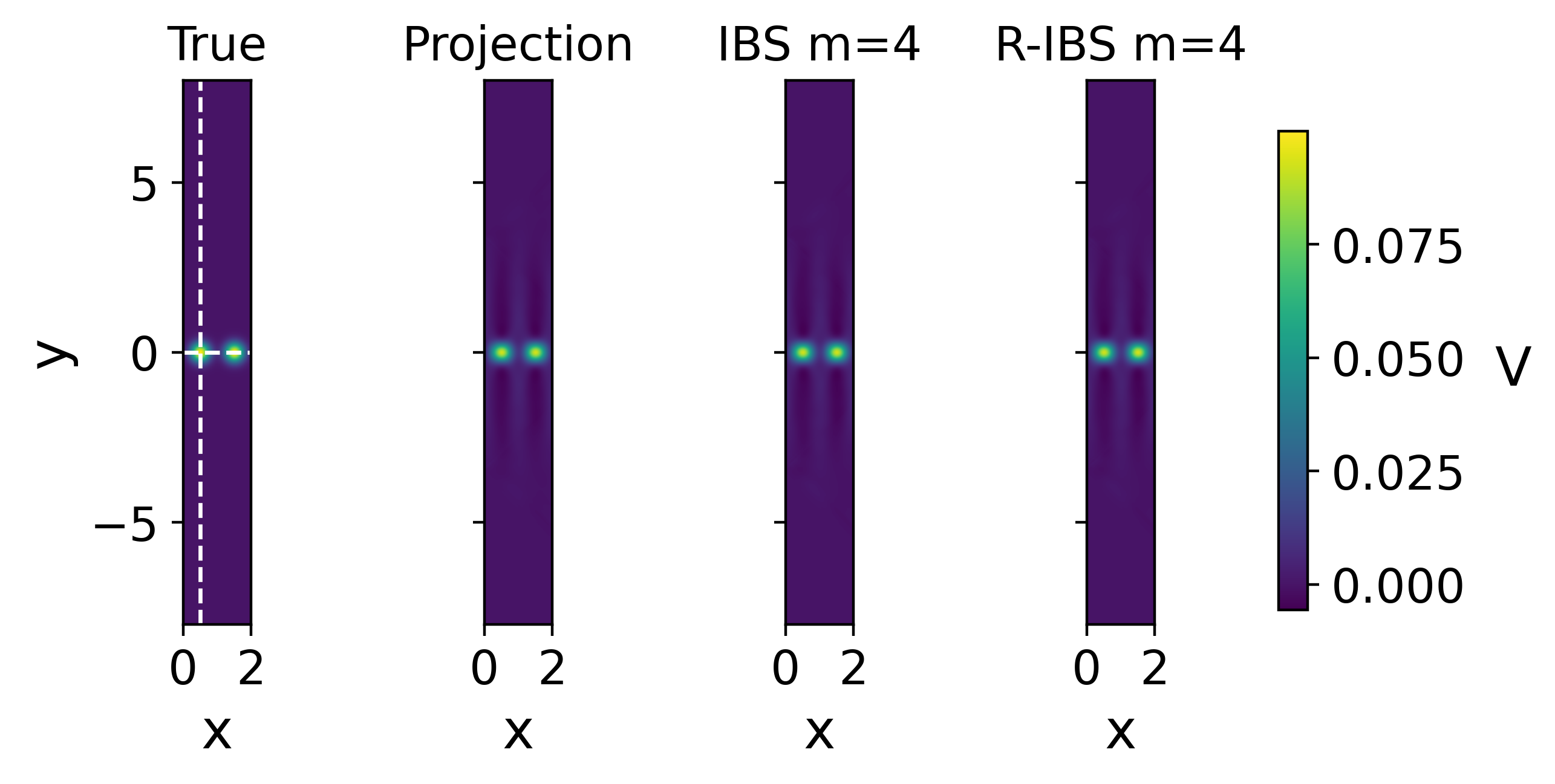}
\caption{Reconstructions of $V$}
\label{fig:chiral-gauss-low-global}
\end{subfigure}

\vspace{0.0em}

\begin{subfigure}[b]{0.43\textwidth}
\includegraphics[width=\textwidth]{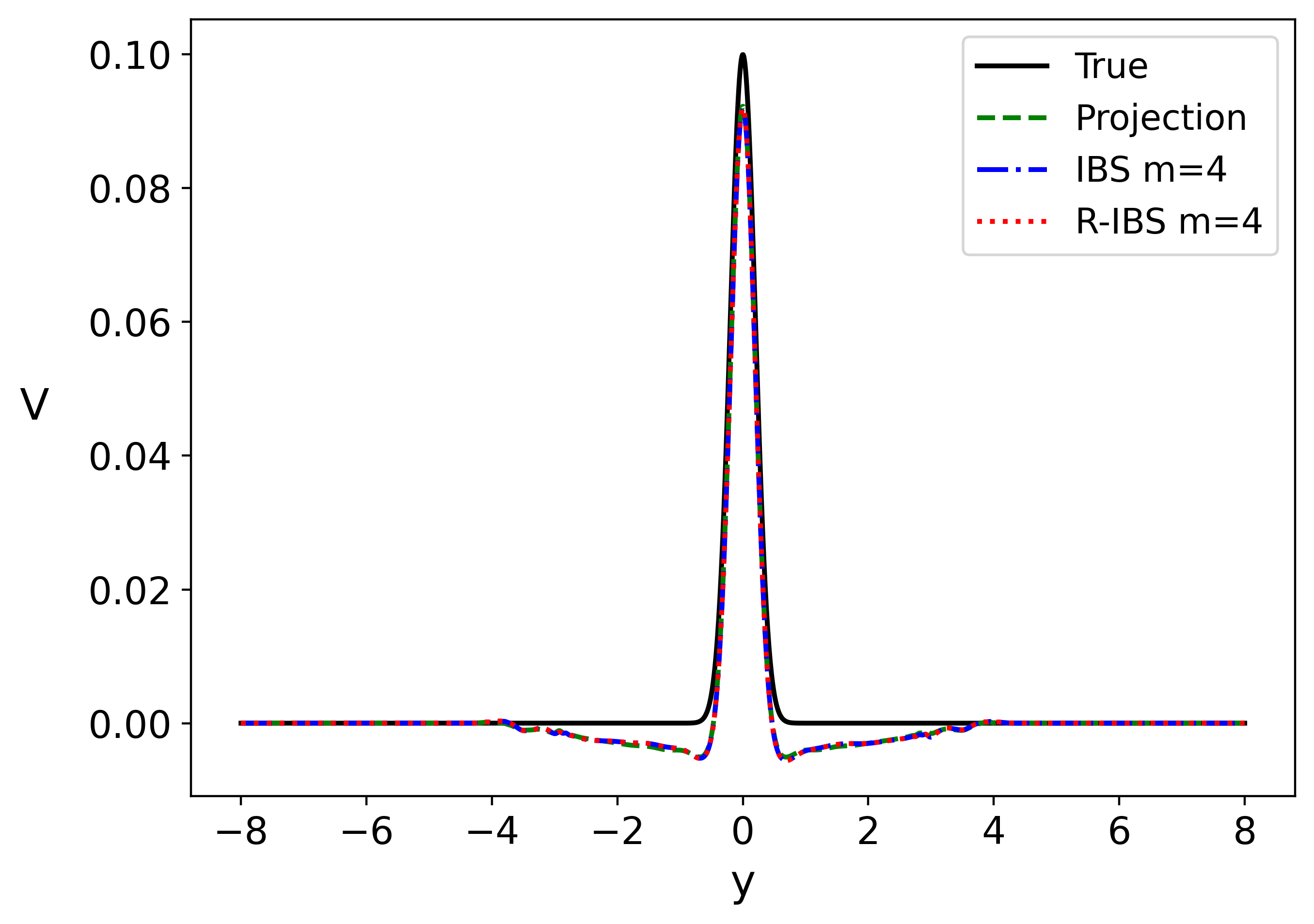}
\caption{Cross section at $x=0.5$}
\label{fig:chiral-gauss-low-slice-x}
\end{subfigure}\hfill
\begin{subfigure}[b]{0.43\textwidth}
\includegraphics[width=\textwidth]{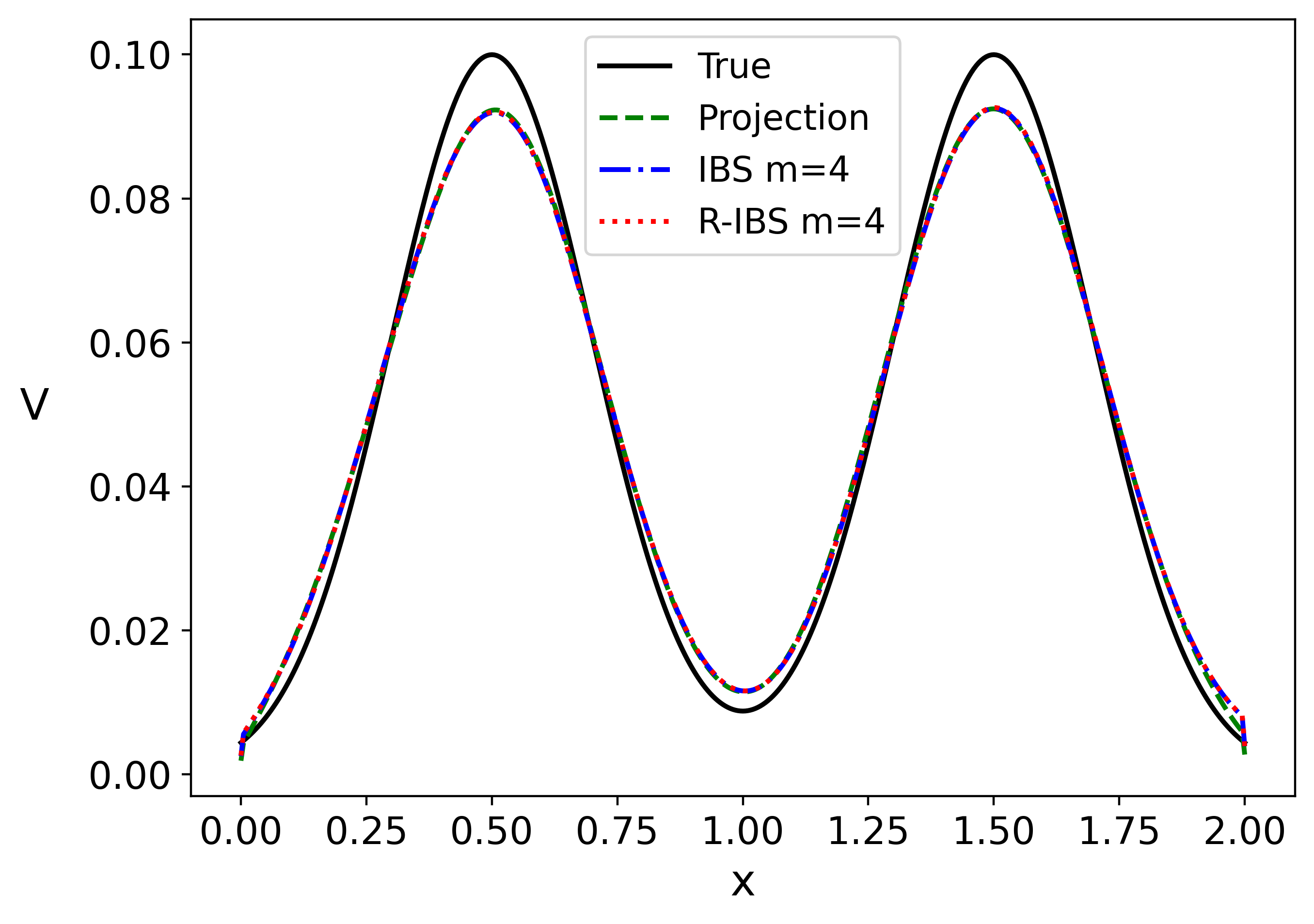}
\caption{Cross section at $y=0.0$}
\label{fig:chiral-gauss-low-slice-y}
\end{subfigure}

\vspace{0.0em}

\begin{subfigure}[b]{0.43\textwidth}
\includegraphics[width=\textwidth]{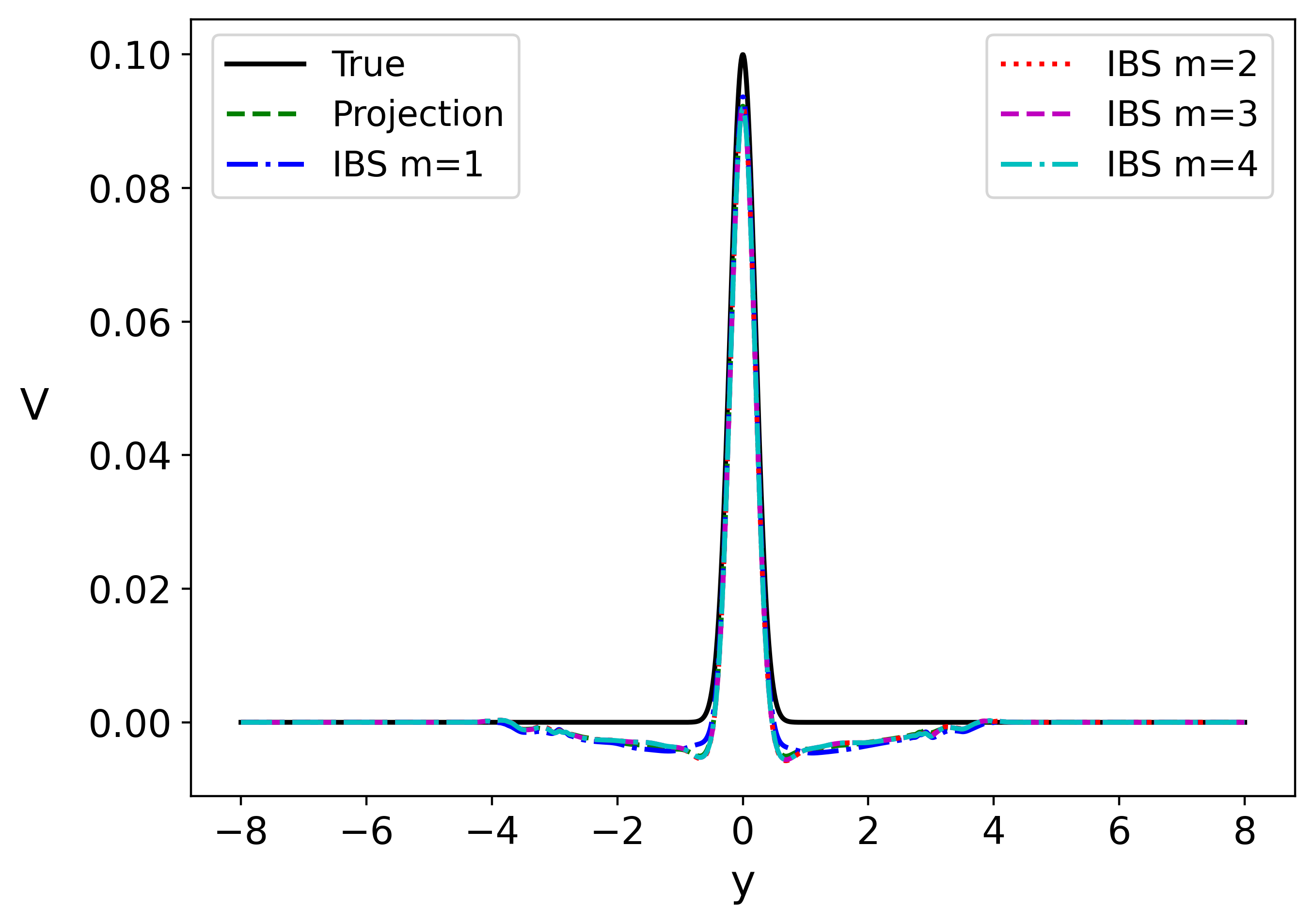}
\caption{IBS cross section at $x=0.5$}
\label{fig:chiral-gauss-low-ibs-x}
\end{subfigure}\hfill
\begin{subfigure}[b]{0.43\textwidth}
\includegraphics[width=\textwidth]{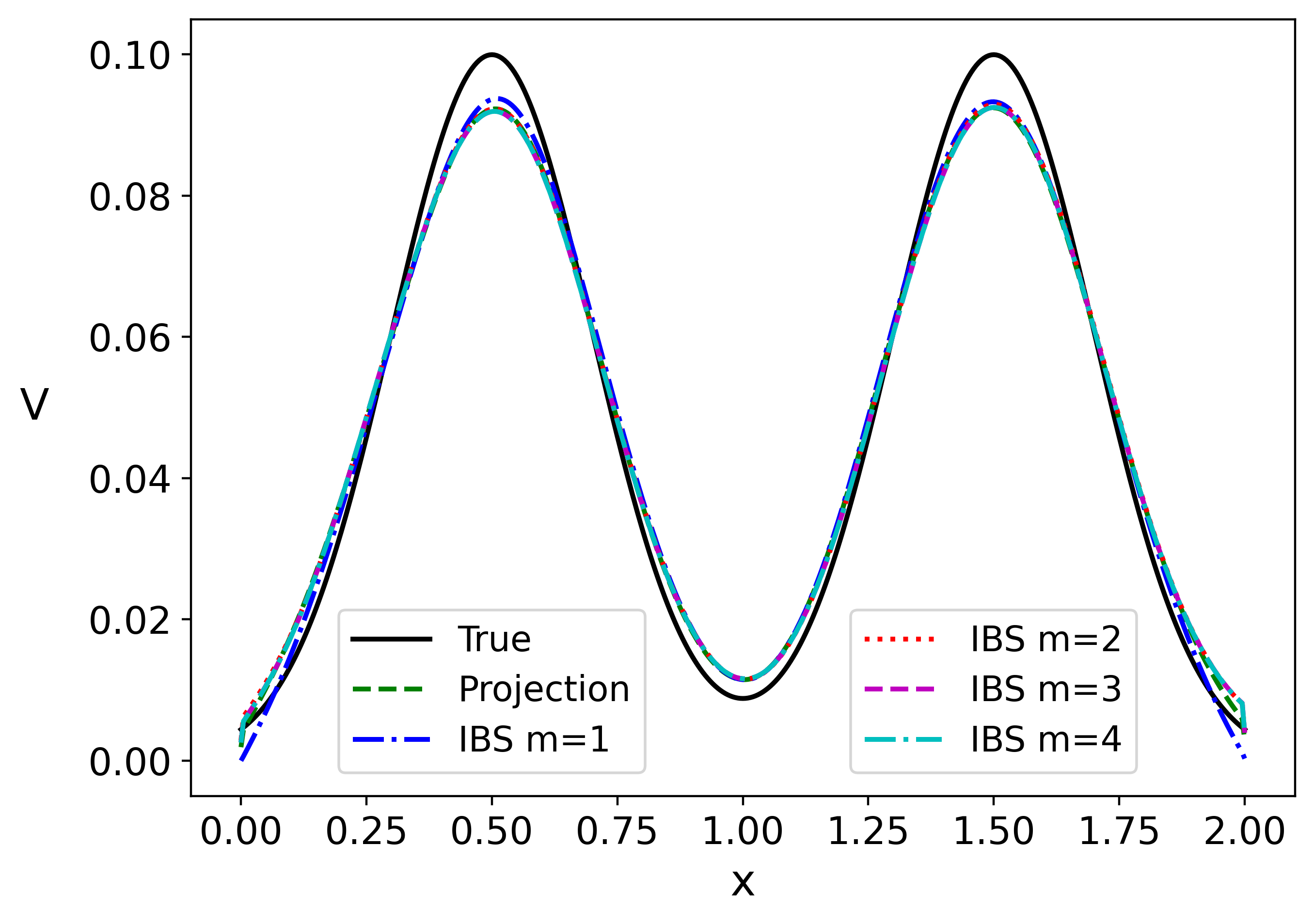}
\caption{IBS cross section at $y=0.0$}
\label{fig:chiral-gauss-low-ibs-y}
\end{subfigure}

\scriptsize
\setlength{\tabcolsep}{3pt}%
\begin{tabular}{@{}l|ccccc|cccc@{}}
    \toprule
     & Projection &
     IBS1 & IBS2 & IBS3 & IBS4 &
     RIBS1 & RIBS2 & RIBS3 & RIBS4 \\
    \midrule
    Relative error &
    0.2027 &   
    0.2107 &   
    0.2104 &   
    0.2099 &   
    0.2098 &   
    0.2107 &   
    0.2104 &   
    0.2099 &   
    0.2098 \\  
    \bottomrule
\end{tabular}
\caption{Reconstructions of two low contrast Gaussian scatterers (chiral model)}
\label{fig:chiral-gauss-low}
\end{figure}

\begin{figure}[htbp]
\centering
\begin{subfigure}[b]{0.87\textwidth}
\includegraphics[width=\textwidth]{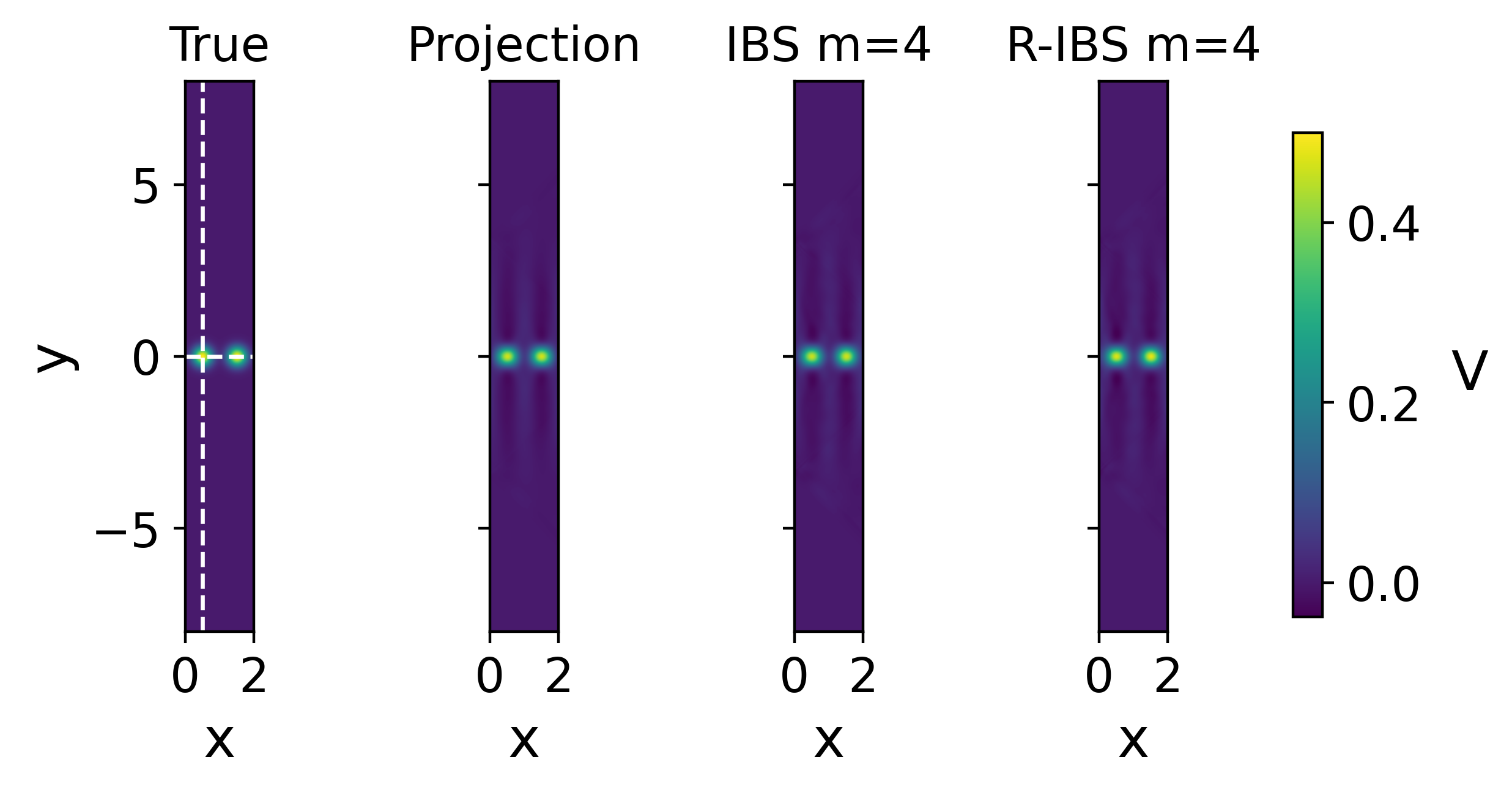}
\caption{Reconstructions of $V$}
\label{fig:chiral-gauss-mid-global}
\end{subfigure}

\vspace{0.0em}

\begin{subfigure}[b]{0.42\textwidth}
\includegraphics[width=\textwidth]{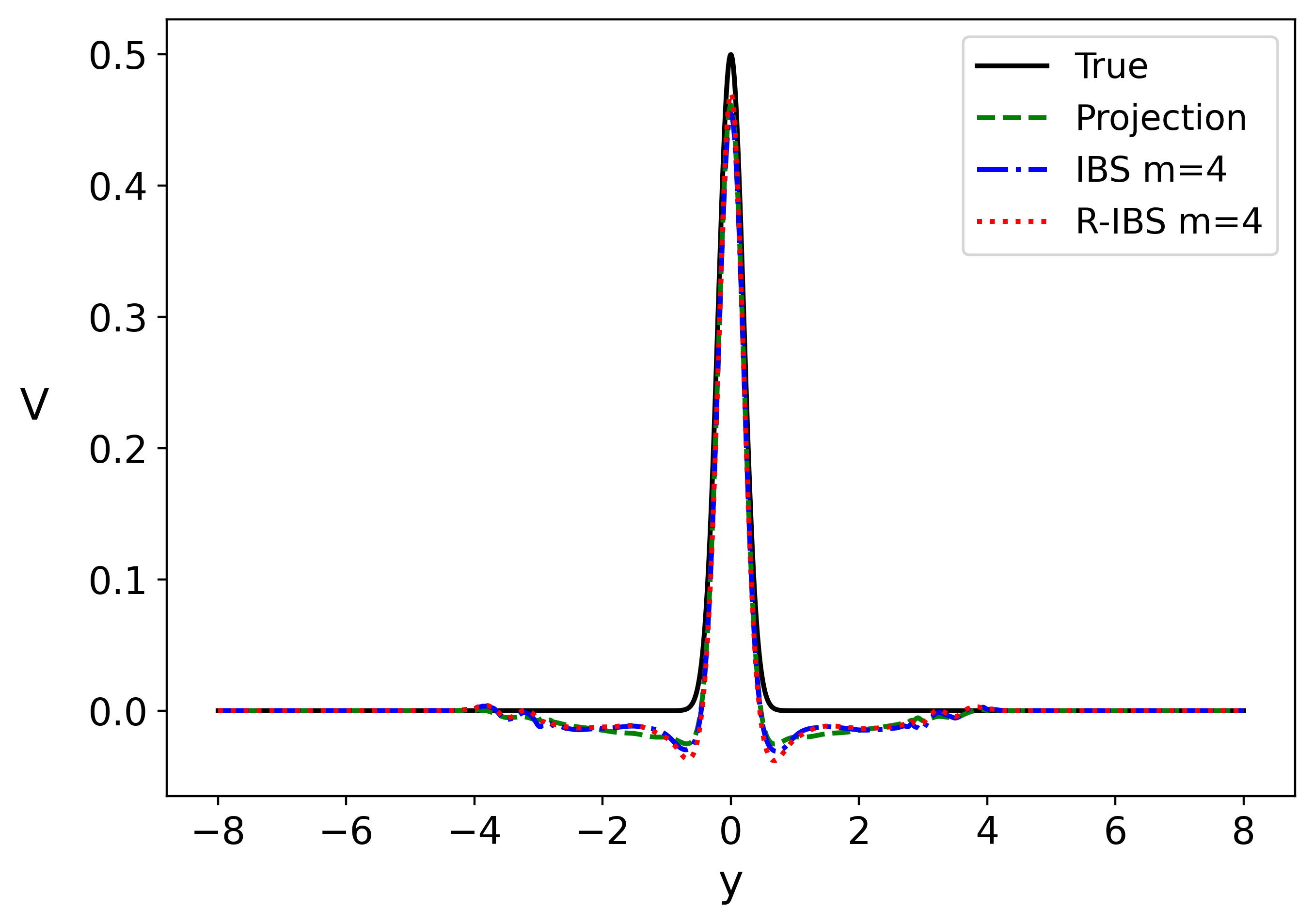}
\caption{Cross section at $x=0.5$}
\label{fig:chiral-gauss-mid-slice-x}
\end{subfigure}\hfill
\begin{subfigure}[b]{0.42\textwidth}
\includegraphics[width=\textwidth]{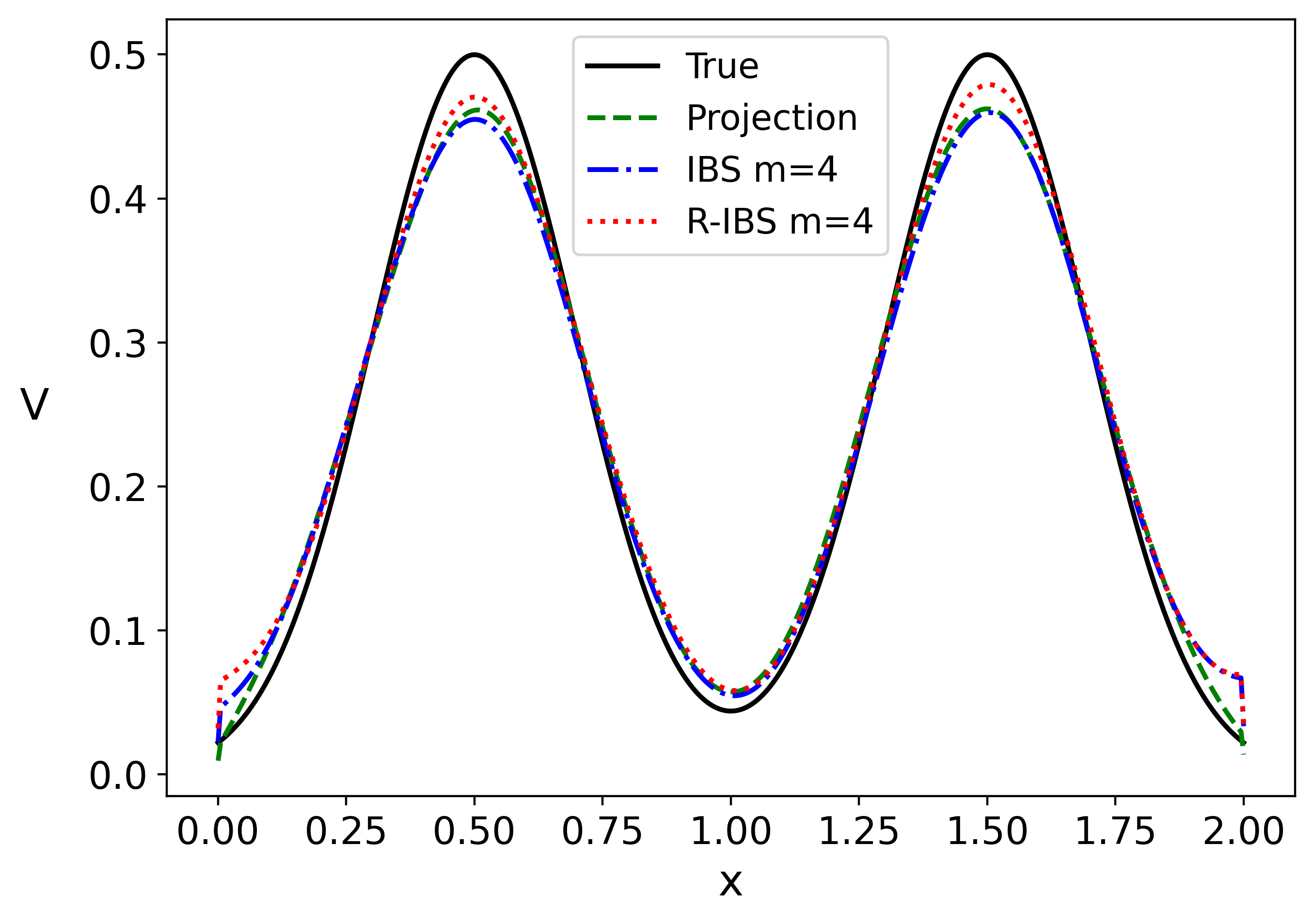}
\caption{Cross section at $y=0.0$}
\label{fig:chiral-gauss-mid-slice-y}
\end{subfigure}
\vspace{0.0em}
\begin{subfigure}[b]{0.42\textwidth}
\includegraphics[width=\textwidth]{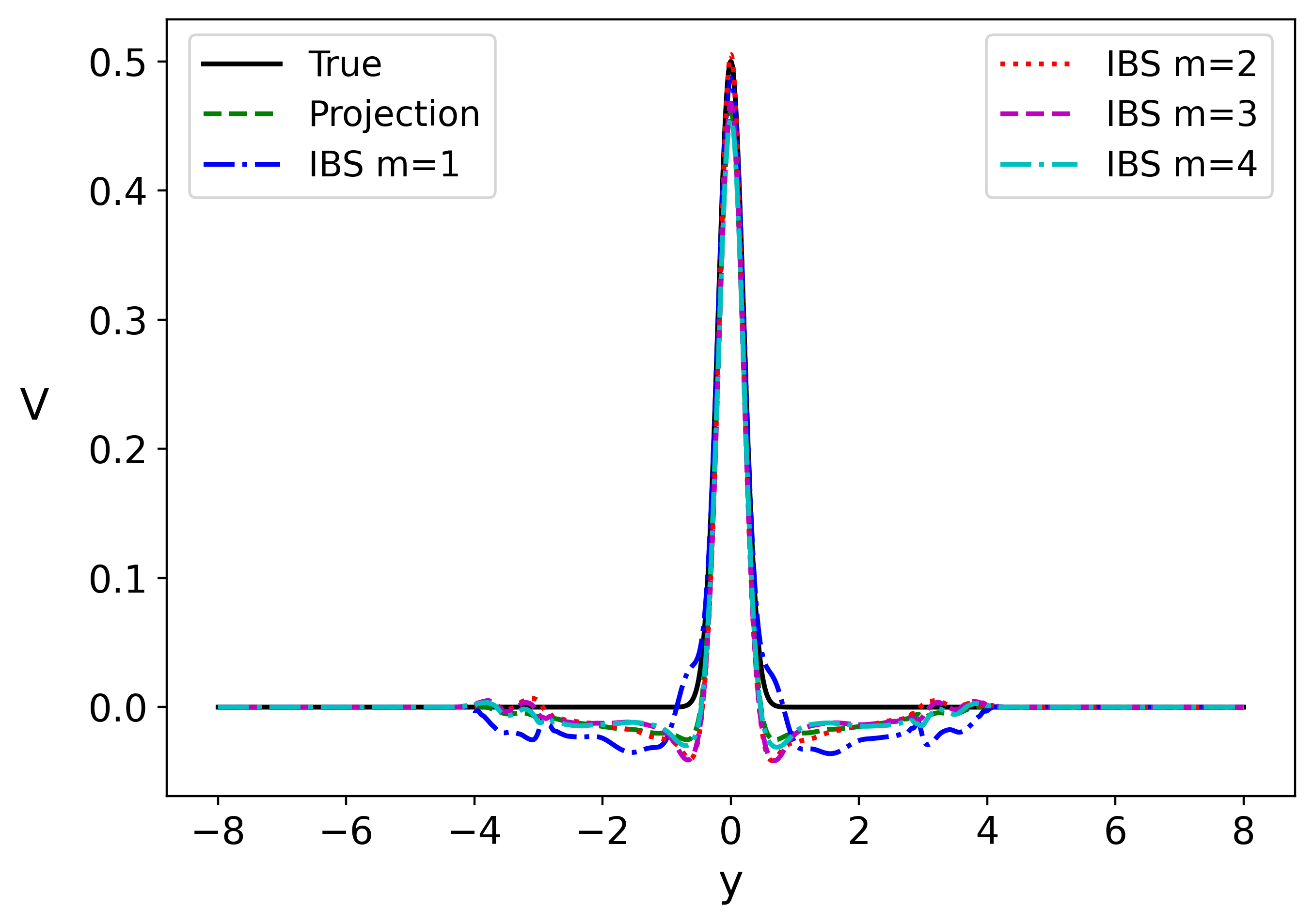}
\caption{IBS cross section at $x=0.5$}
\label{fig:chiral-gauss-mid-ibs-x}
\end{subfigure}\hfill
\begin{subfigure}[b]{0.42\textwidth}
\includegraphics[width=\textwidth]{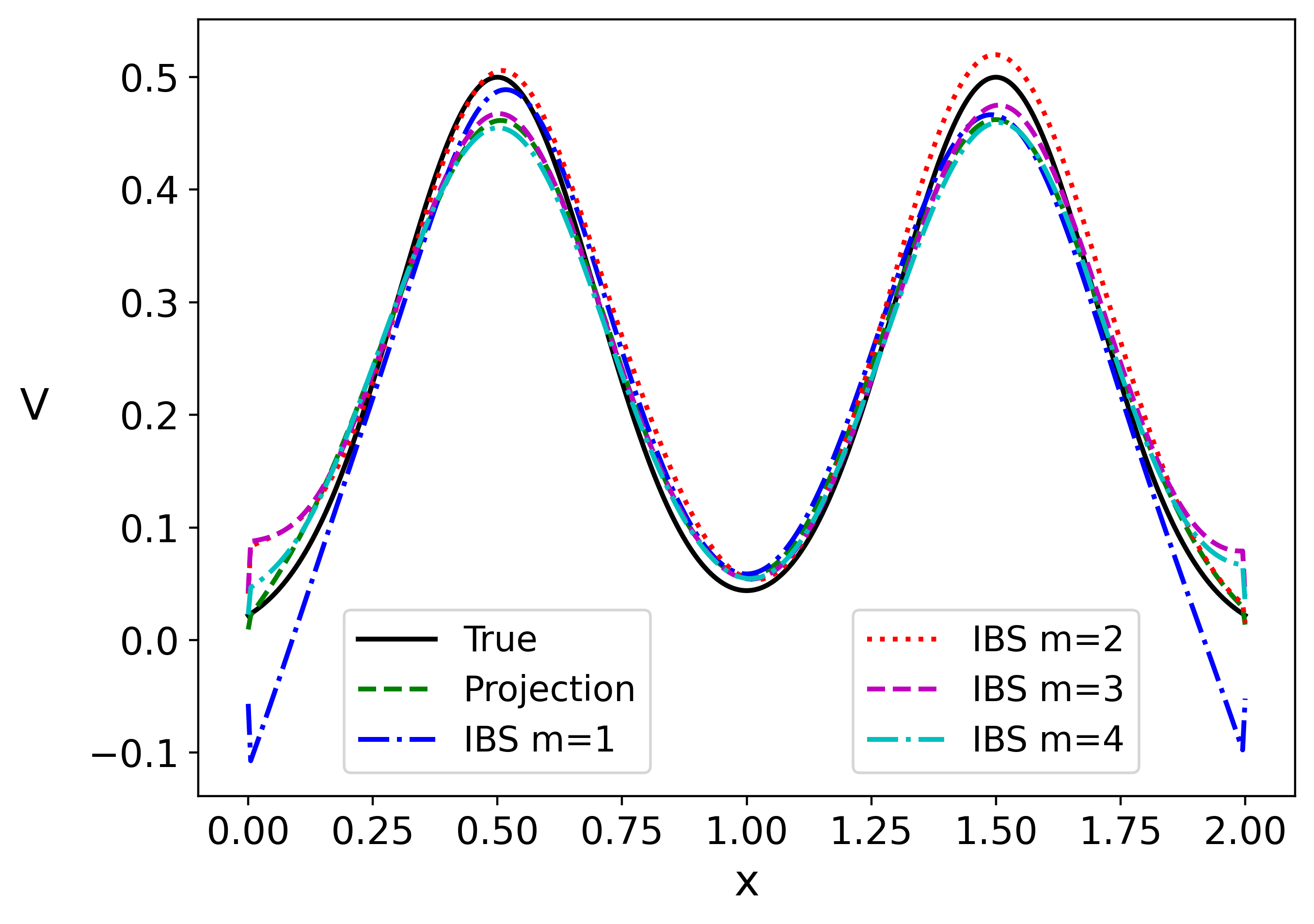}
\caption{IBS cross section at $y=0.0$}
\label{fig:chiral-gauss-mid-ibs-y}
\end{subfigure}

\scriptsize
\setlength{\tabcolsep}{3pt}
\begin{tabular}{@{}l|ccccc|cccc@{}}
\toprule
 & Projection & IBS1 & IBS2 & IBS3 & IBS4 &
 RIBS1 & RIBS2 & RIBS3 & RIBS4 \\
\midrule
Relative error &
0.2027 &   
0.4491 &   
0.2832 &   
0.2599 &   
0.2364 &   
0.4491 &   
0.2832 &   
0.2553 &   
0.2363 \\  
\bottomrule
\end{tabular}

\caption{Reconstructions of two medium contrast Gaussian scatterers (chiral model)}
\label{fig:chiral-gauss-mid}
\end{figure}
\begin{figure}[htbp]
\centering
\begin{subfigure}[b]{0.9\textwidth}
\includegraphics[width=\textwidth]{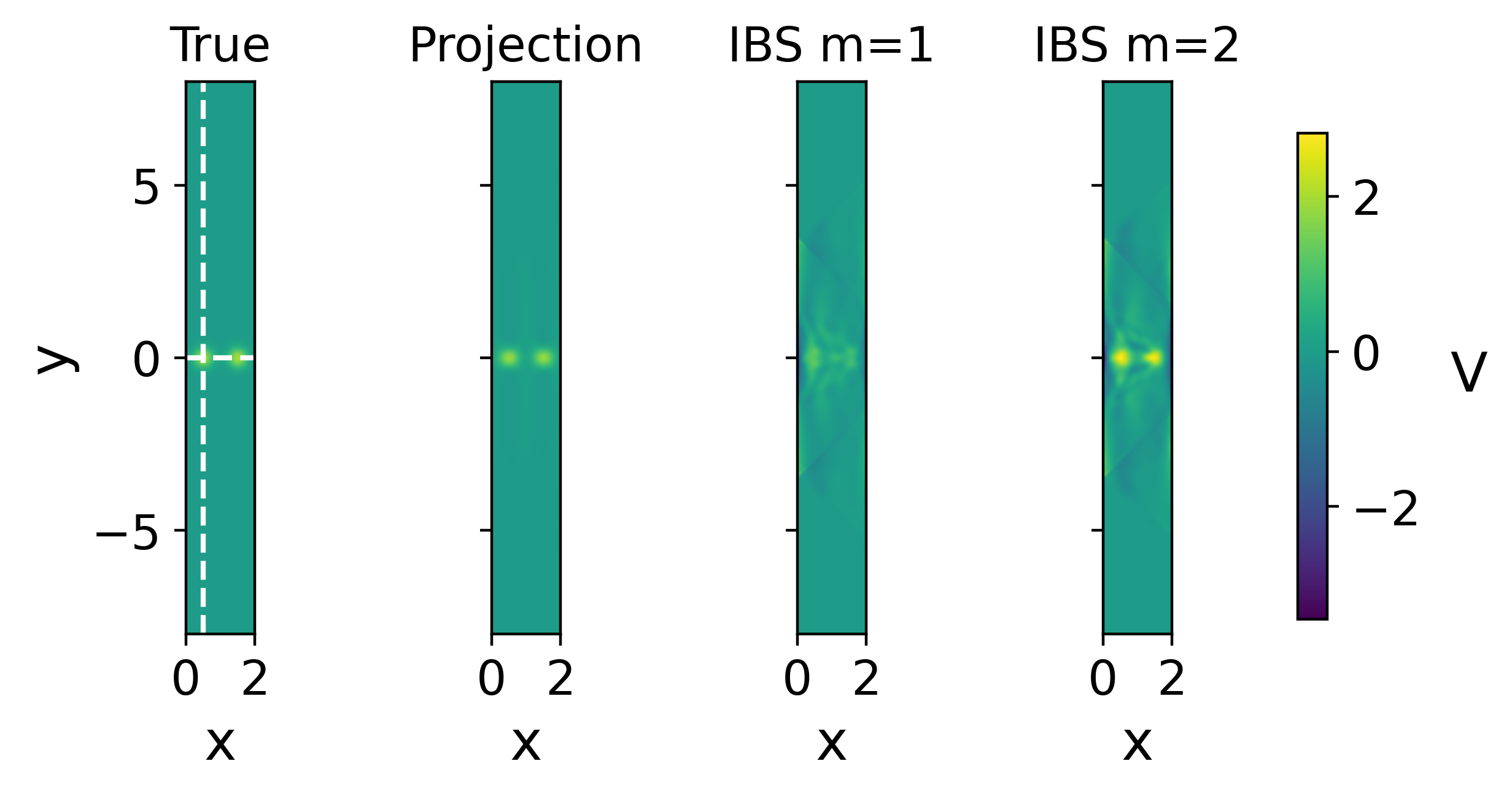}
\caption{Reconstructions of $V$}
\label{fig:chiral-gauss-high-global}
\end{subfigure}

\vspace{0.0em}

\begin{subfigure}[b]{0.45\textwidth}
\includegraphics[width=\textwidth]{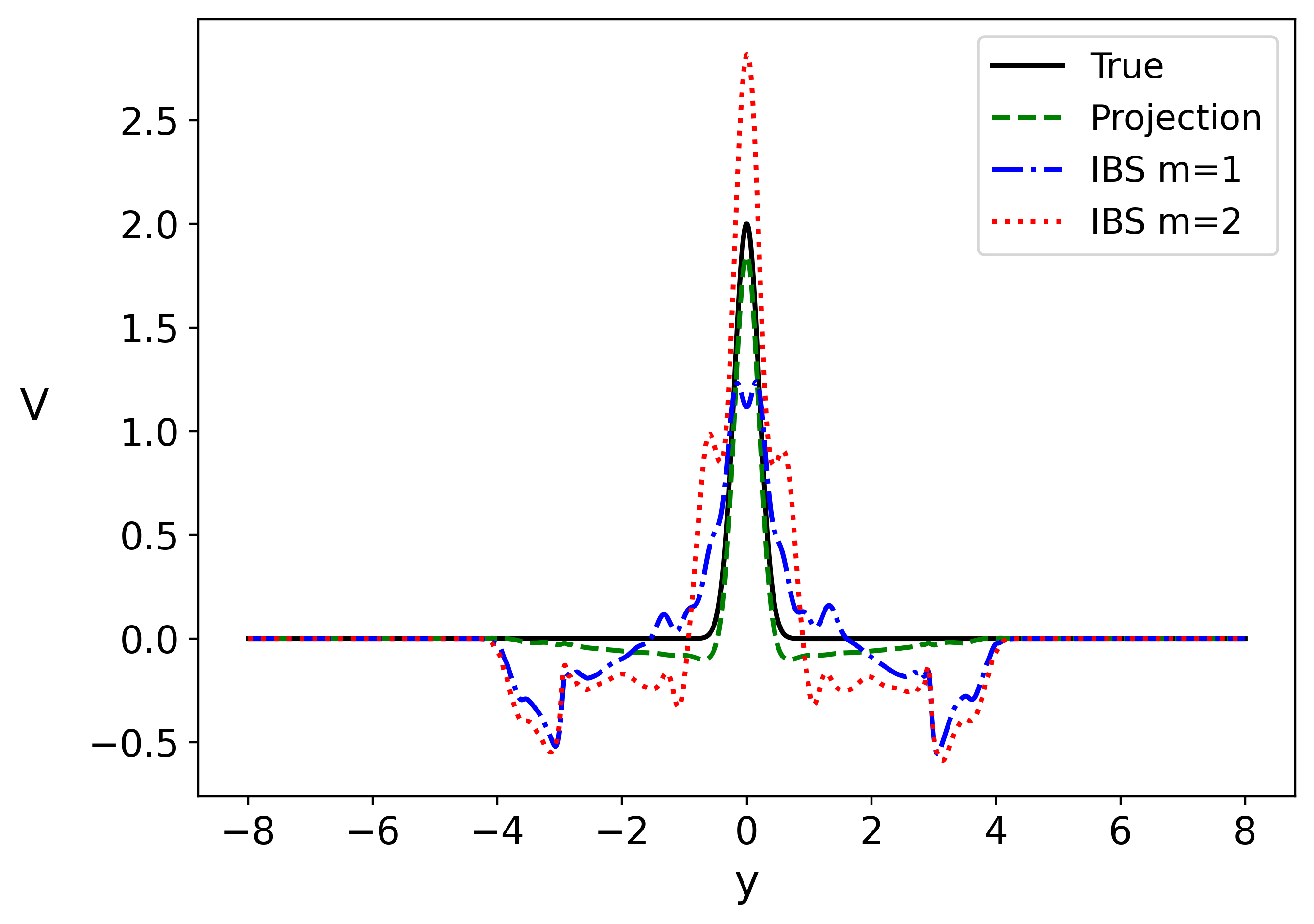}
\caption{Cross section at $x=0.5$}
\label{fig:chiral-gauss-high-slice-x}
\end{subfigure}

\vspace{0.0em}

\begin{subfigure}[b]{0.45\textwidth}
\includegraphics[width=\textwidth]{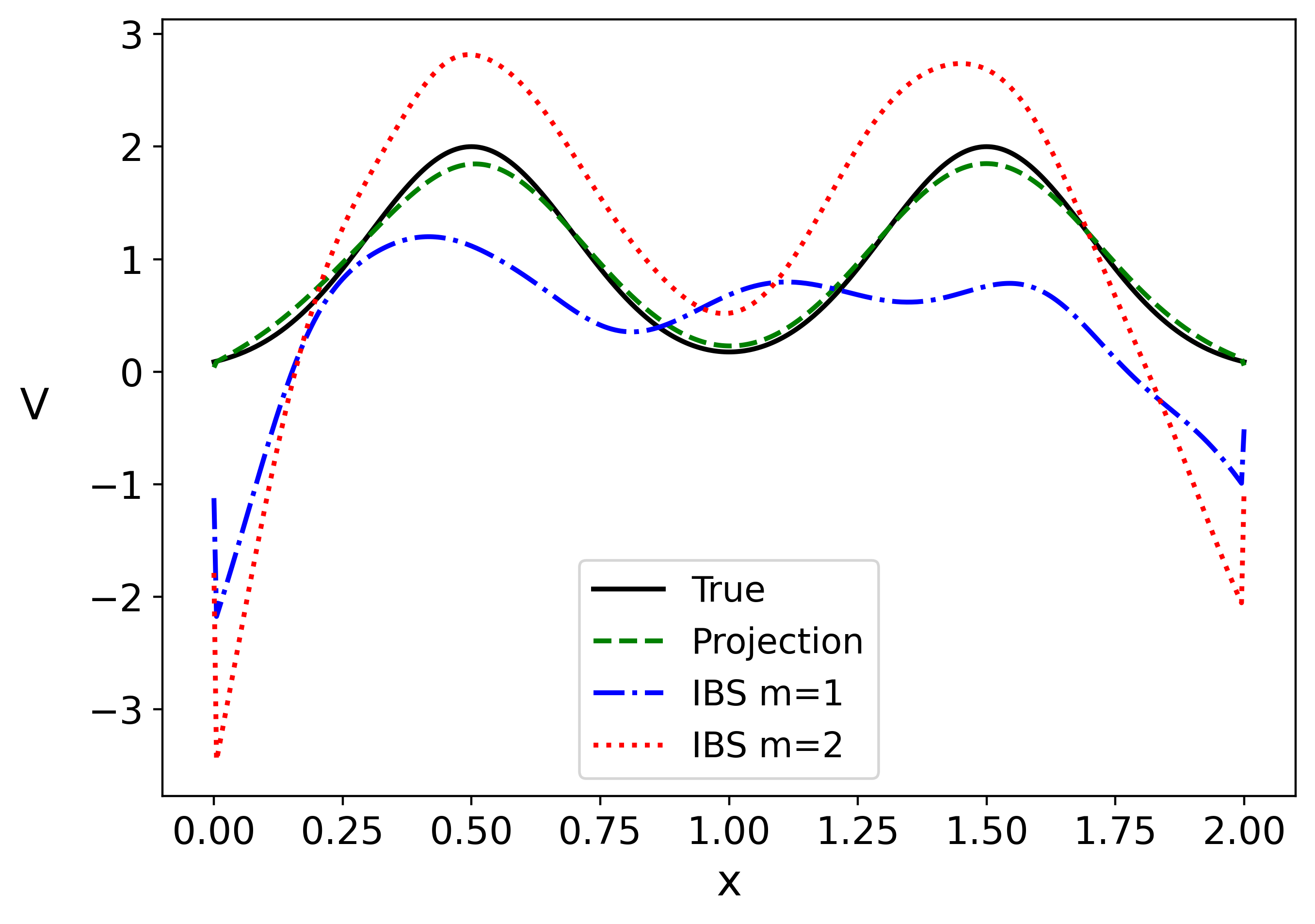}
\caption{Cross section at $y=0.0$}
\label{fig:chiral-gauss-high-slice-y}
\end{subfigure}

\caption{Reconstructions of two high contrast Gaussian scatterers (chiral model)}
\label{fig:chiral-gauss-high}
\end{figure}
\begin{figure}[htbp]
    \centering
    \begin{subfigure}[b]{\textwidth}
        \centering
        \includegraphics[width=\textwidth]{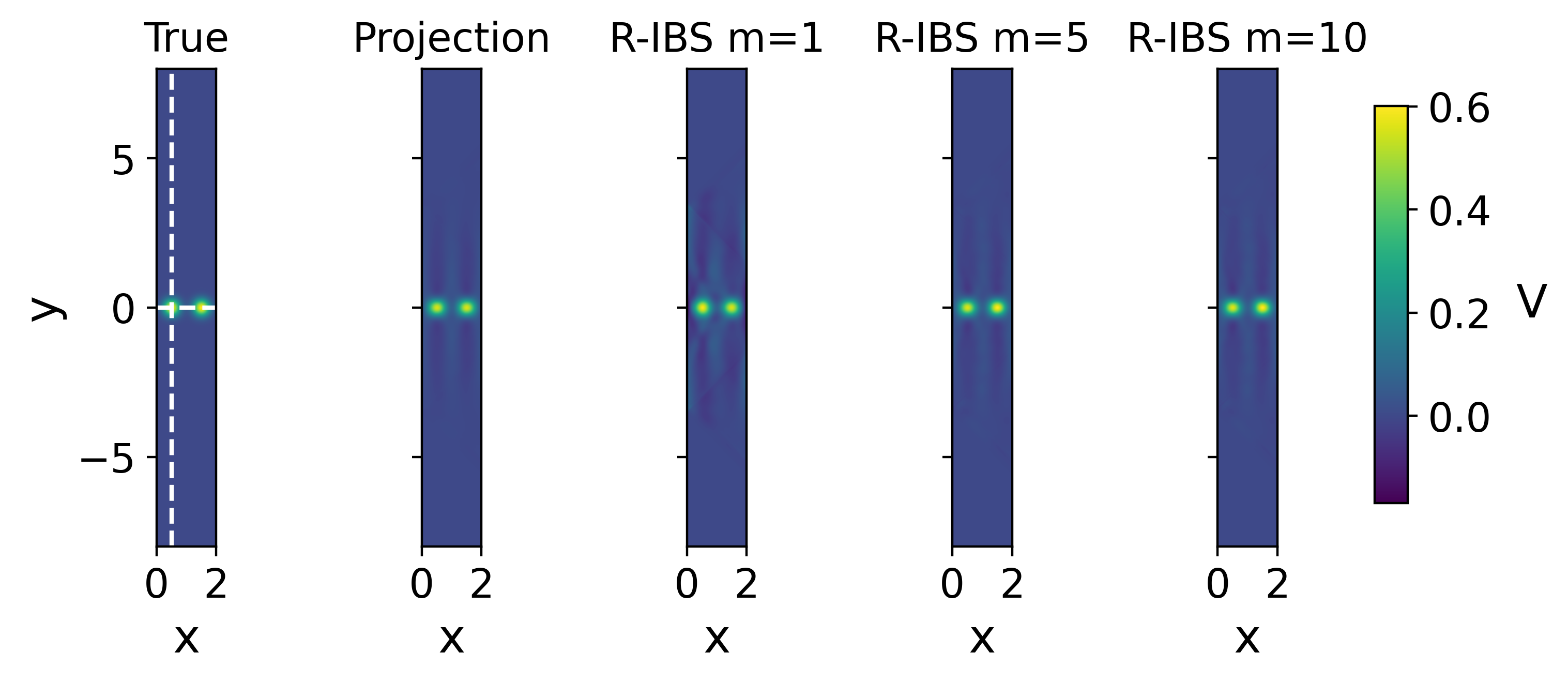}
        \caption{Reconstructions of $V$}
        \label{fig:chiral-fullribs-global}
    \end{subfigure}

    \par\vspace{0.5em}

    \begin{subfigure}[b]{0.5\textwidth}
        \centering
        \includegraphics[width=\textwidth]{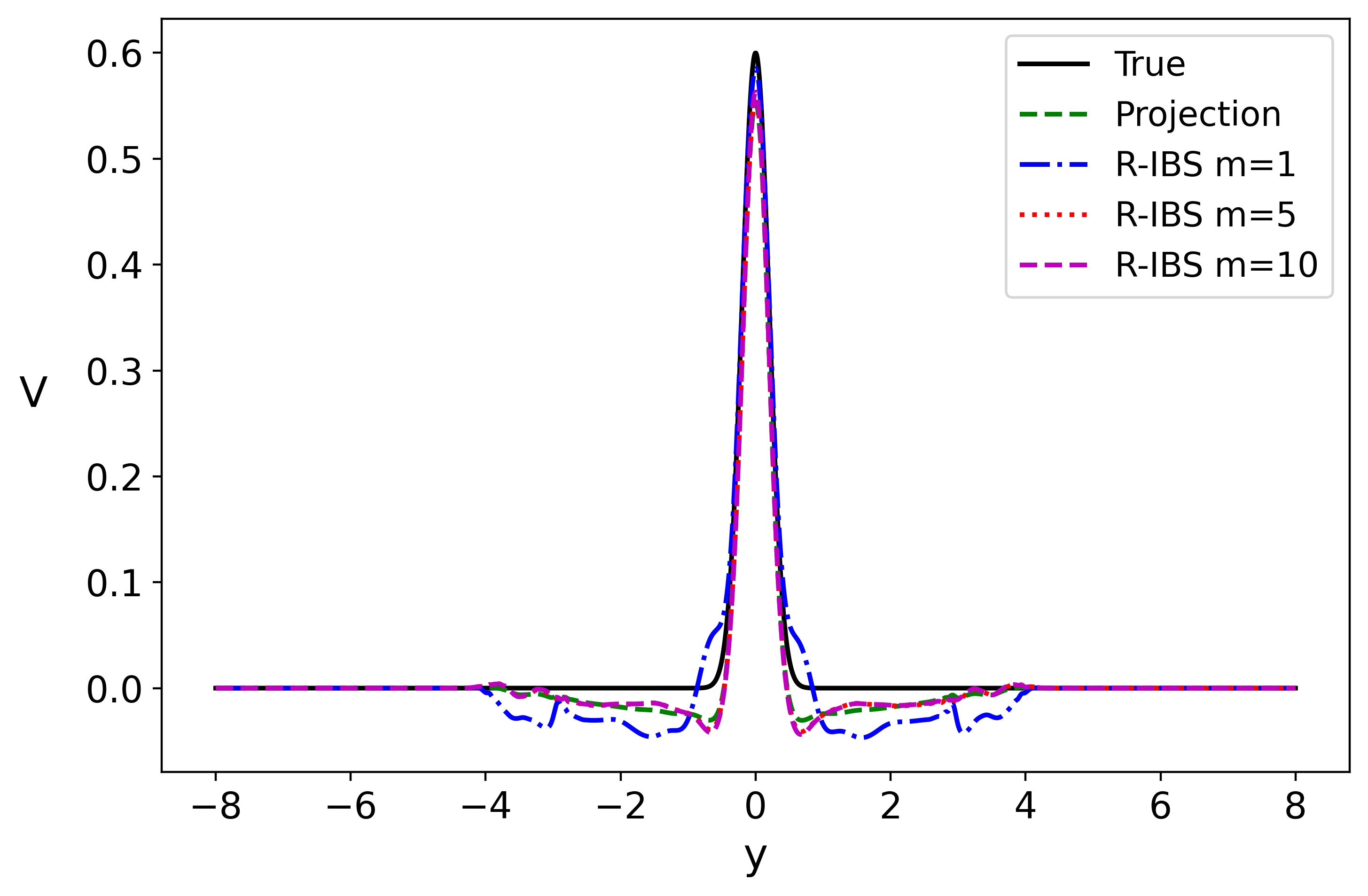}
        \caption{Cross section at $x=0.5$}
        \label{fig:chiral-fullribs-slice-x}
    \end{subfigure}%
    \begin{subfigure}[b]{0.5\textwidth}
        \centering
        \includegraphics[width=\textwidth]{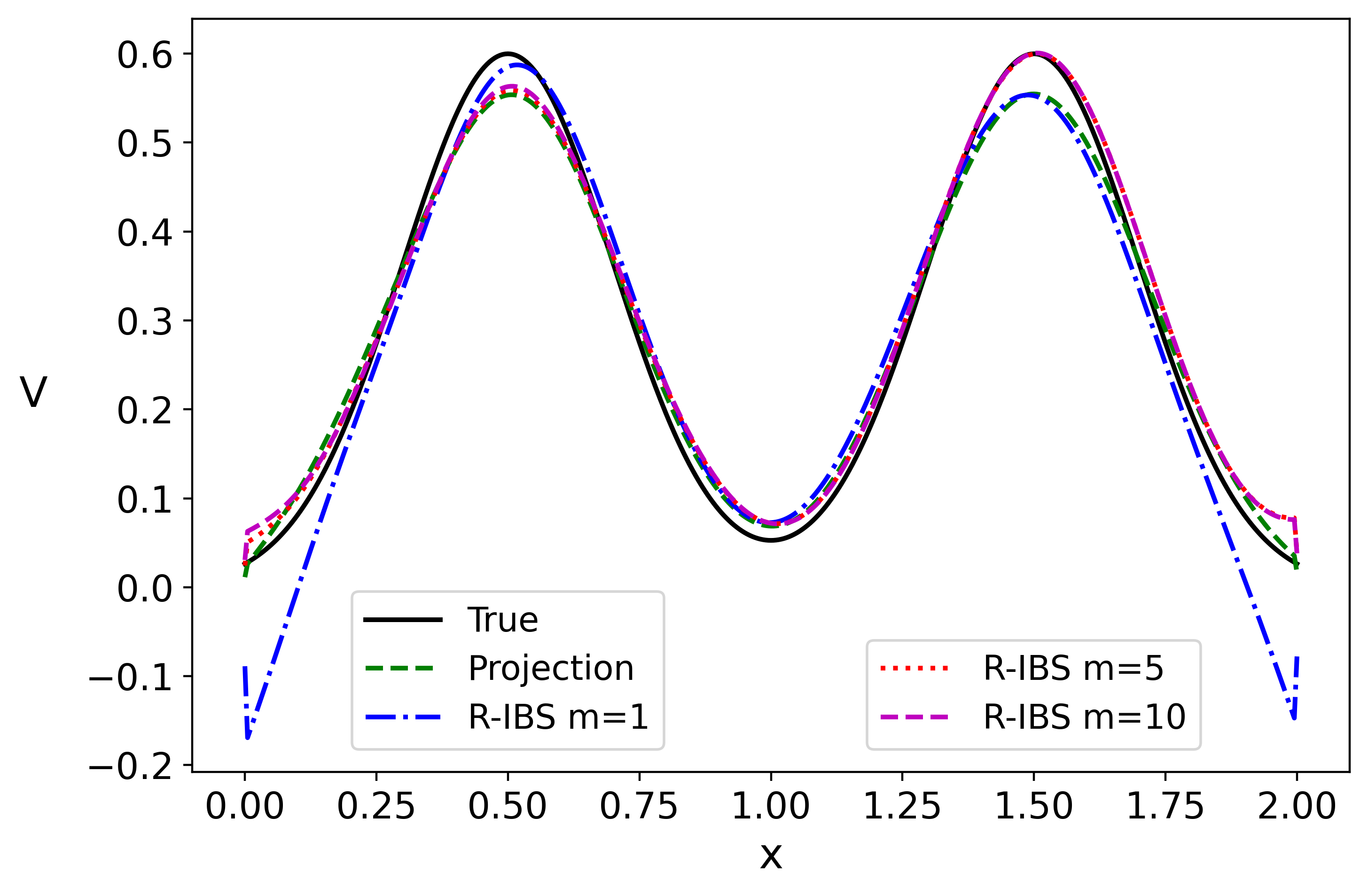}
        \caption{Cross section at $y=0.0$}
        \label{fig:chiral-fullribs-slice-y}
    \end{subfigure}

  \scriptsize
\setlength{\tabcolsep}{3pt}%
\resizebox{\linewidth}{!}{%
\begin{tabular}{@{}l|ccccccccccc@{}}
\toprule
 & Projection &
 RIBS1 & RIBS2 & RIBS3 & RIBS4 & RIBS5 &
 RIBS6 & RIBS7 & RIBS8 & RIBS9 & RIBS10 \\
\midrule
Relative error &
0.2027 &   
0.5319 &   
0.3318 &   
0.2816 &   
0.2471 &   
0.2304 &   
0.2290 &   
0.2310 &   
0.2317 &   
0.2316 &   
0.2315 \\  
\bottomrule
\end{tabular}%
}

    \caption{Reconstructions for the chiral model using the RIBS}
    \label{fig:chiral-fullribs}
\end{figure}

\begin{figure}[htbp]
    \centering
    \begin{subfigure}[b]{0.9\textwidth}
        \centering
        \includegraphics[width=\textwidth]{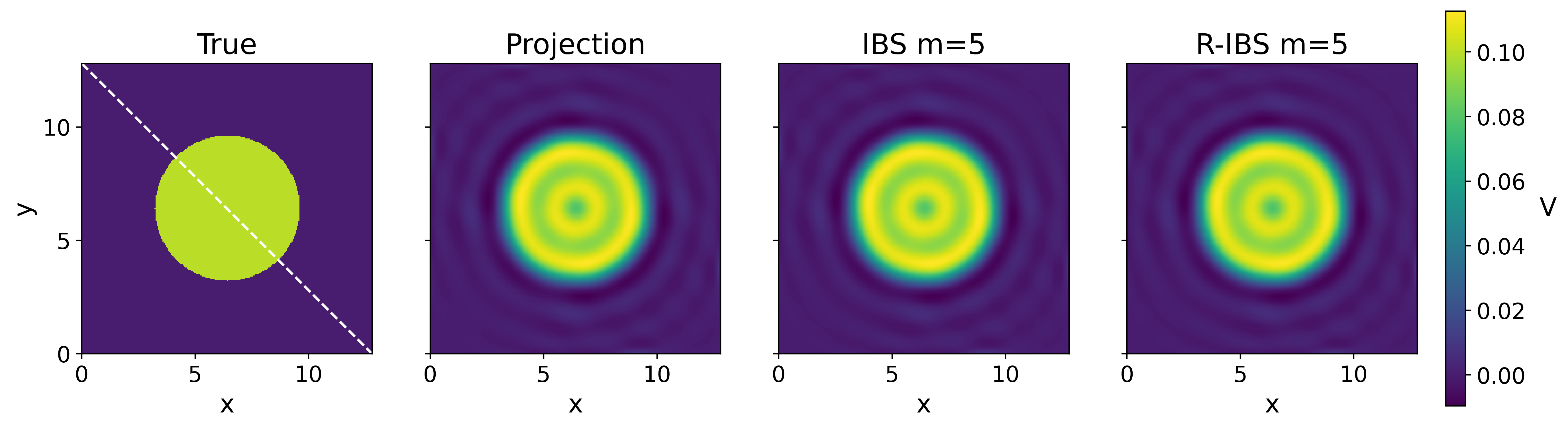}
        \caption{Reconstructions of $V$}
        \label{fig:antichiral-disk-low-global}
    \end{subfigure}

    \vspace{0.6em}

    \begin{subfigure}[b]{0.45\textwidth}
        \centering
        \includegraphics[width=\textwidth]{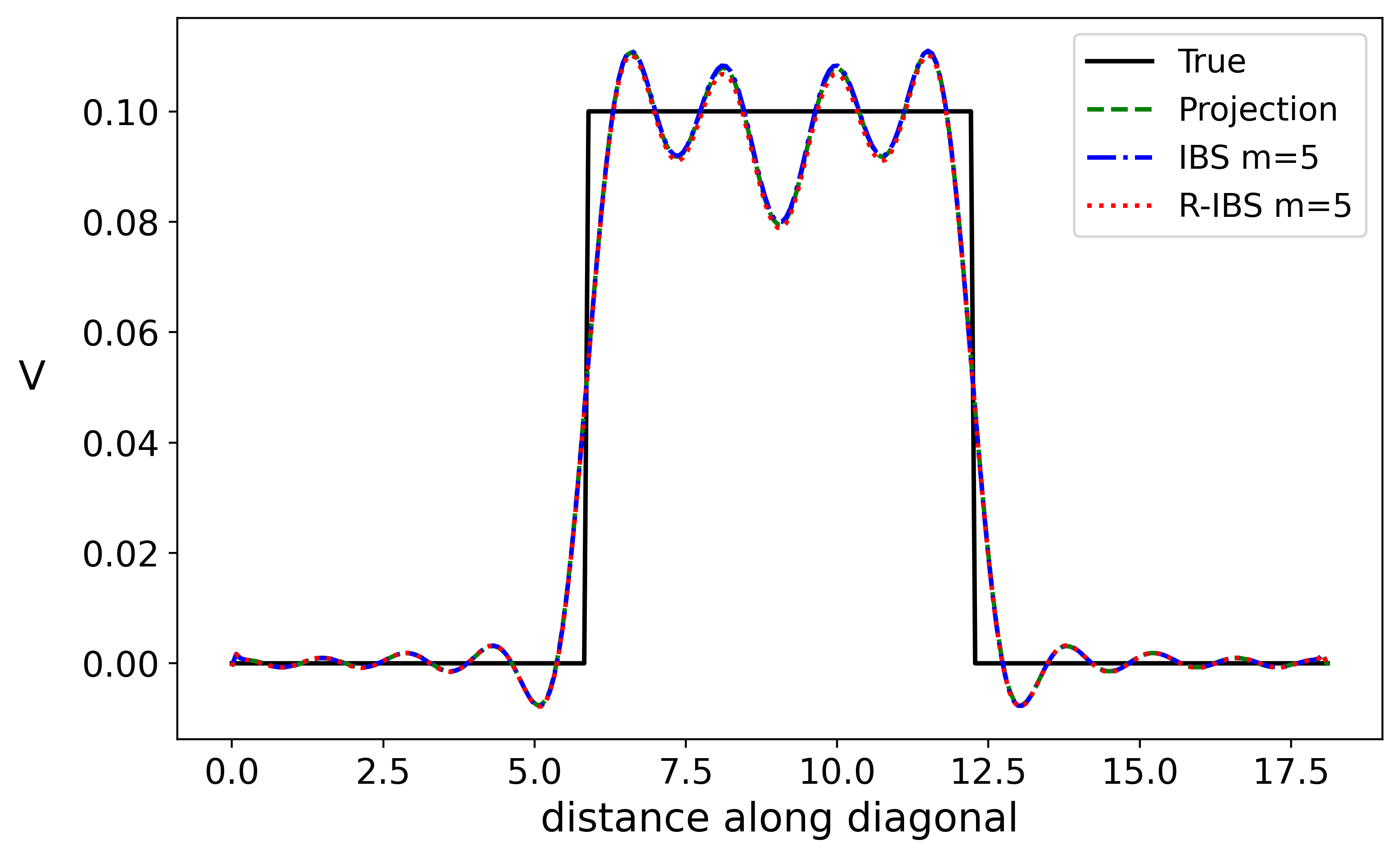}
        \caption{Anti-diagonal cross section}
        \label{fig:antichiral-disk-low-slice-main}
    \end{subfigure}%

    \vspace{0.6em}

    \begin{subfigure}[b]{0.45\textwidth}
        \centering
        \includegraphics[width=\textwidth]{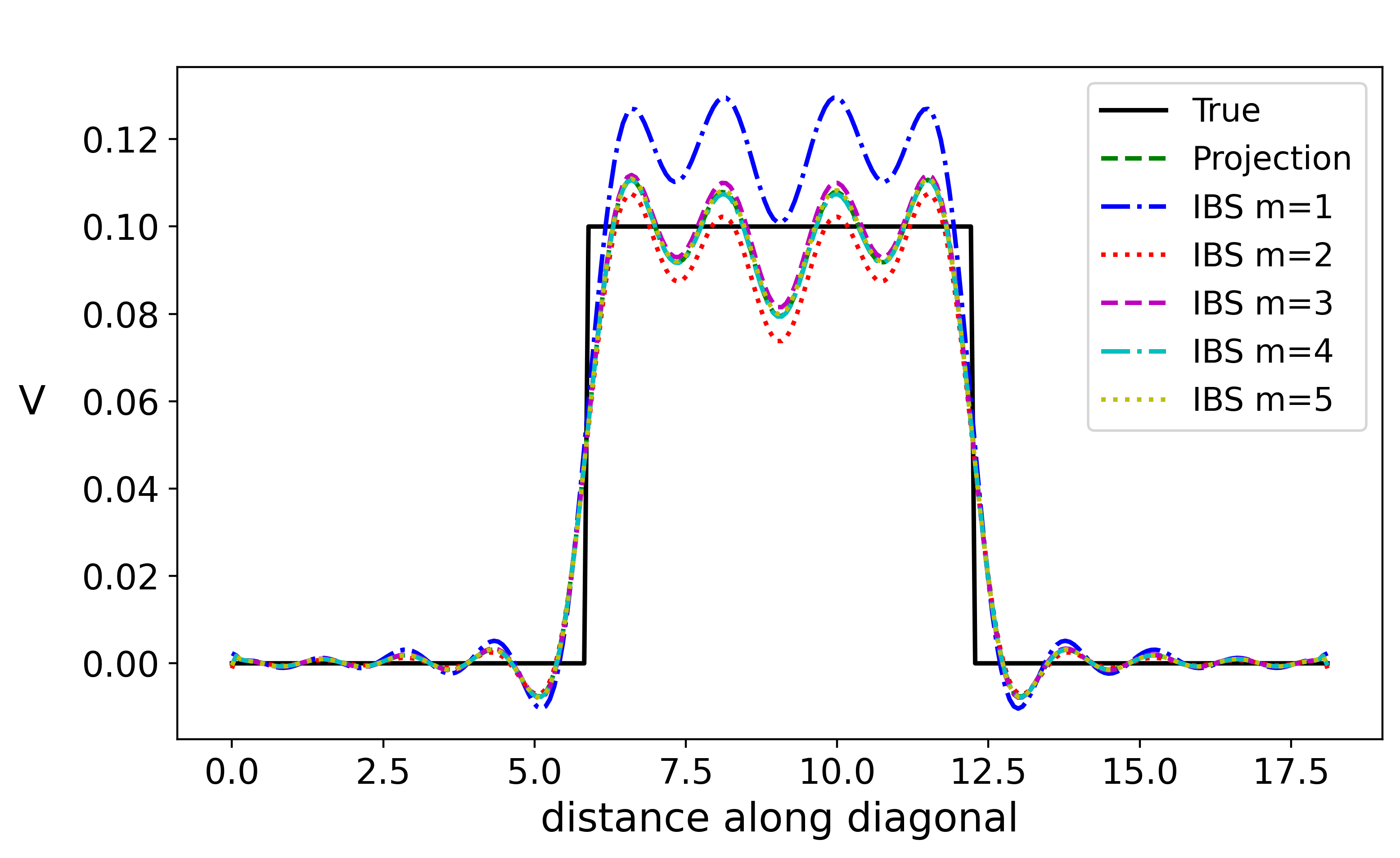}
        \caption{Anti-diagonal cross section of IBS}
        \label{fig:antichiral-disk-low-slice-ibsall}
    \end{subfigure}%

    \scriptsize
\setlength{\tabcolsep}{3pt}%
\resizebox{\linewidth}{!}{
\begin{tabular}{@{}l|cccccc|ccccc@{}}
    \toprule
     & Projection &
     IBS1 & IBS2 & IBS3 & IBS4 & IBS5 &
     RIBS1 & RIBS2 & RIBS3 & RIBS4 & RIBS5 \\
    \midrule
    Relative error &
    0.2186 &   
    0.2754 &   
    0.2234 &   
    0.2186 &   
    0.2186 &   
    0.2185 &   
    0.2754 &   
    0.2234 &   
    0.2185 &   
    0.2188 &   
    0.2187 \\  
    \bottomrule
\end{tabular}}

    \caption{Reconstructions of a low contrast disk (anti-chiral model)}
    \label{fig:antichiral-disk-low}
\end{figure}
\begin{figure}[htbp]
    \centering
    \begin{subfigure}[b]{0.9\textwidth}
        \centering
        \includegraphics[width=\textwidth]{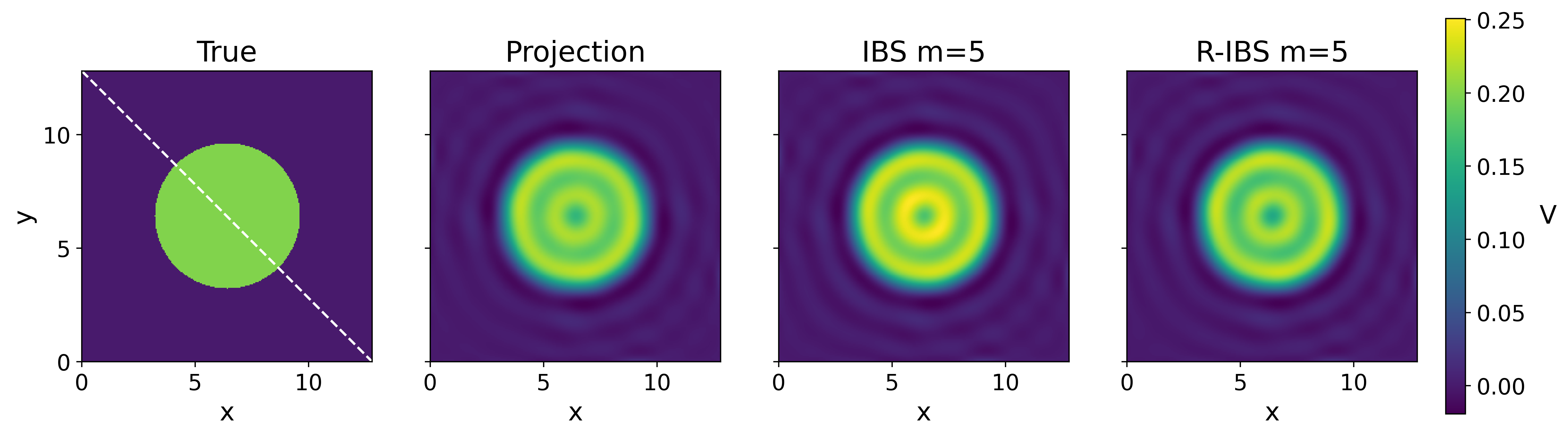}
        \caption{Reconstructions of $V$}
        \label{fig:antichiral-disk-mid-global}
    \end{subfigure}

    \vspace{0.6em}

    \begin{subfigure}[b]{0.45\textwidth}
        \centering
        \includegraphics[width=\textwidth]{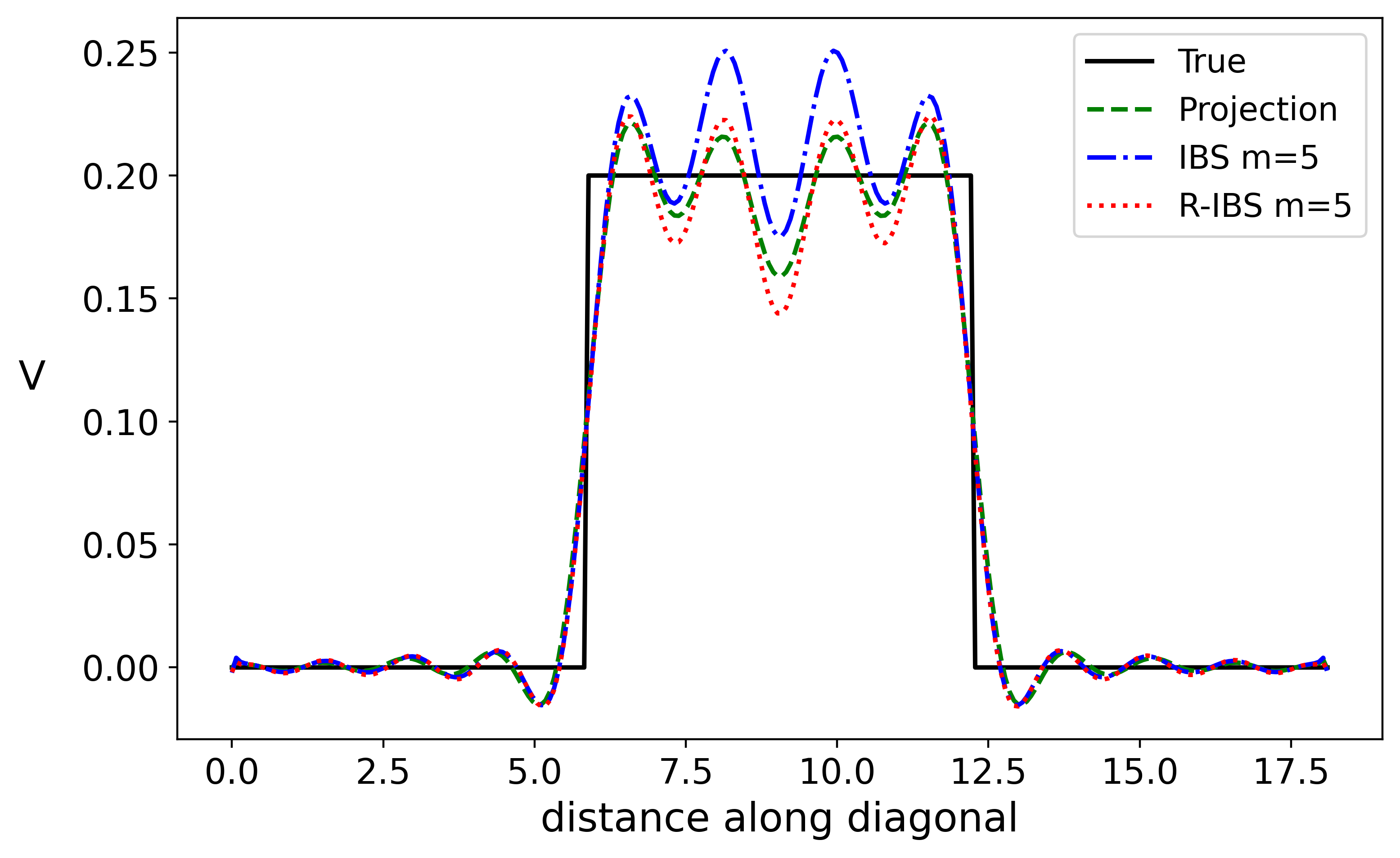}
        \caption{Anti-diagonal cross section}
        \label{fig:antichiral-disk-mid-slice-main}
    \end{subfigure}%

    \vspace{0.6em}

    \begin{subfigure}[b]{0.45\textwidth}
        \centering
        \includegraphics[width=\textwidth]{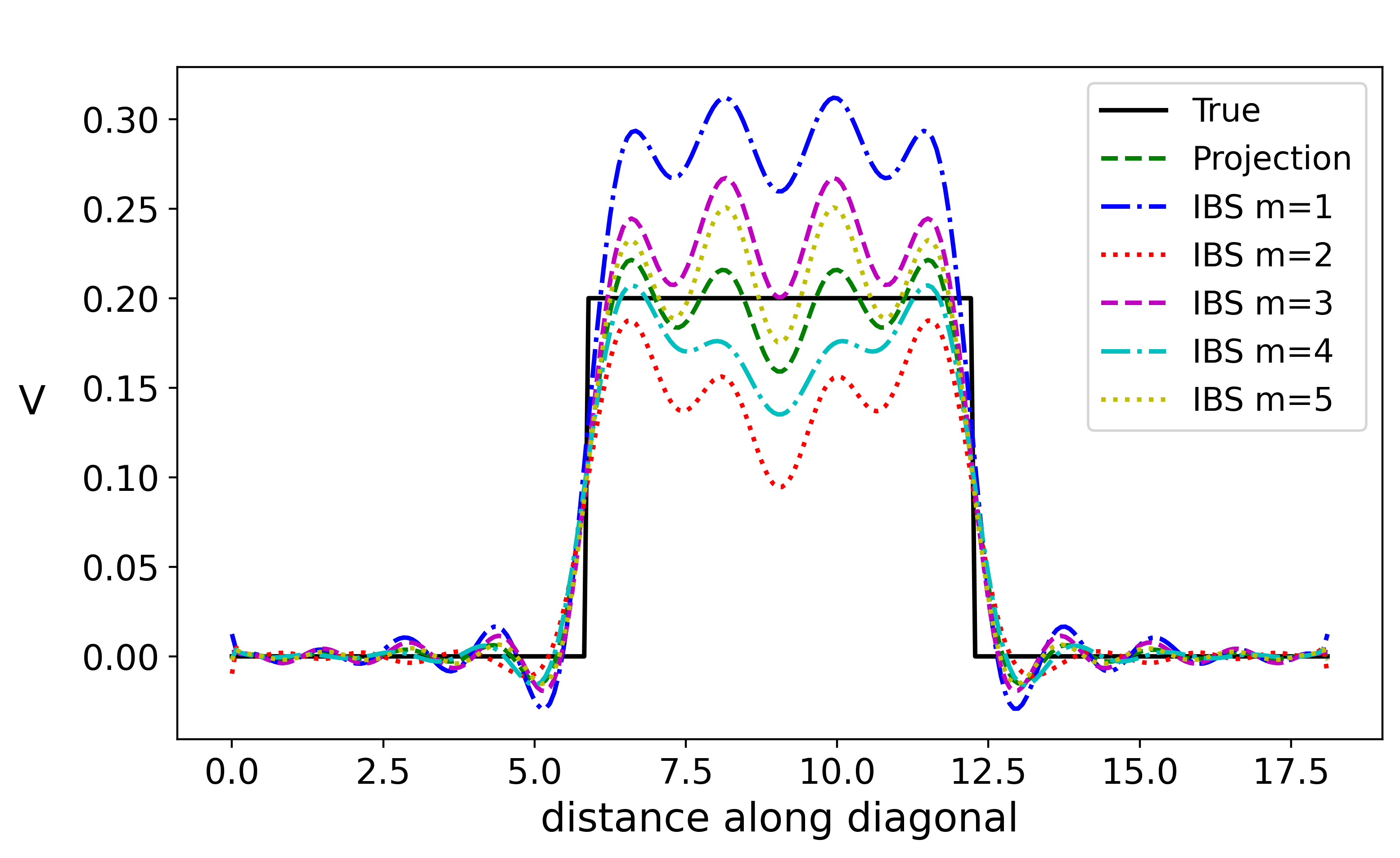}
        \caption{Anti-diagonal cross section for IBS}
        \label{fig:antichiral-disk-mid-slice-ibsall}
    \end{subfigure}%

    \scriptsize
\setlength{\tabcolsep}{3pt}%
\resizebox{\linewidth}{!}{
\begin{tabular}{@{}l|cccccc|ccccc@{}}
    \toprule
     & Projection &
     IBS1 & IBS2 & IBS3 & IBS4 & IBS5 &
     RIBS1 & RIBS2 & RIBS3 & RIBS4 & RIBS5 \\
    \midrule
    Relative error &
    0.2186 &   
    0.4442 &   
    0.3148 &   
    0.2616 &   
    0.2415 &   
    0.2298 &   
    0.4442 &   
    0.3148 &   
    0.2370 &   
    0.2351 &   
    0.2243 \\  
    \bottomrule
\end{tabular}}

    \caption{Reconstructions of a medium contrast disk (anti-chiral model)}
    \label{fig:antichiral-disk-mid}
\end{figure}
\begin{figure}[htbp]
    \centering

    \begin{subfigure}[b]{\textwidth}
        \centering
        \includegraphics[width=\textwidth]{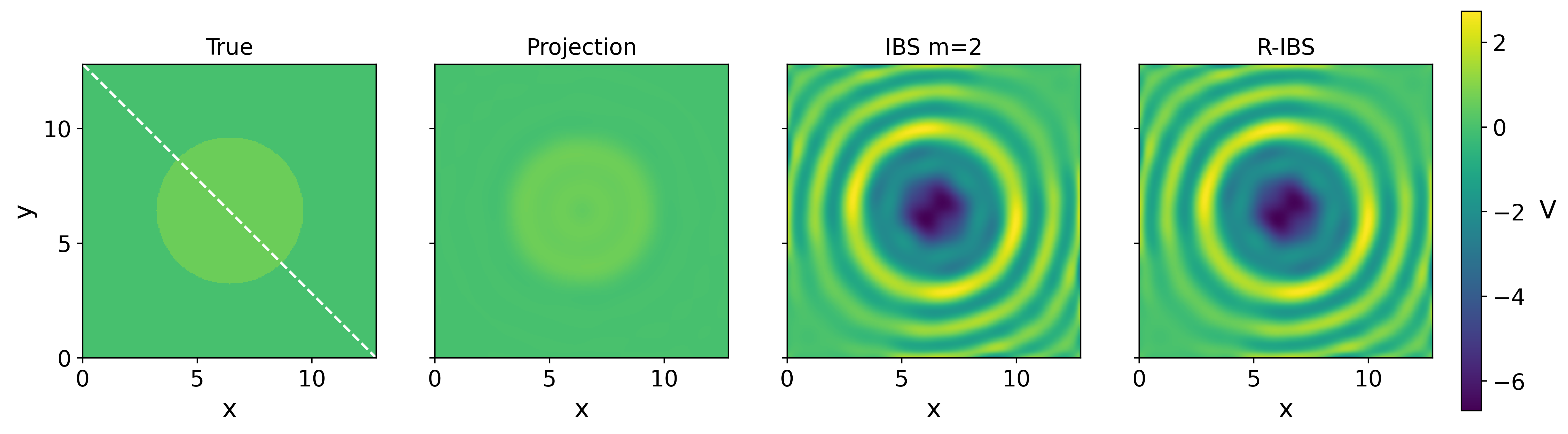}
        \caption{Reconstructions of $V$}
        \label{fig:antichiral-disk-high-global}
    \end{subfigure}

    \par\vspace{0.5em}

    \begin{subfigure}[b]{0.5\textwidth}
        \centering
        \includegraphics[width=\textwidth]{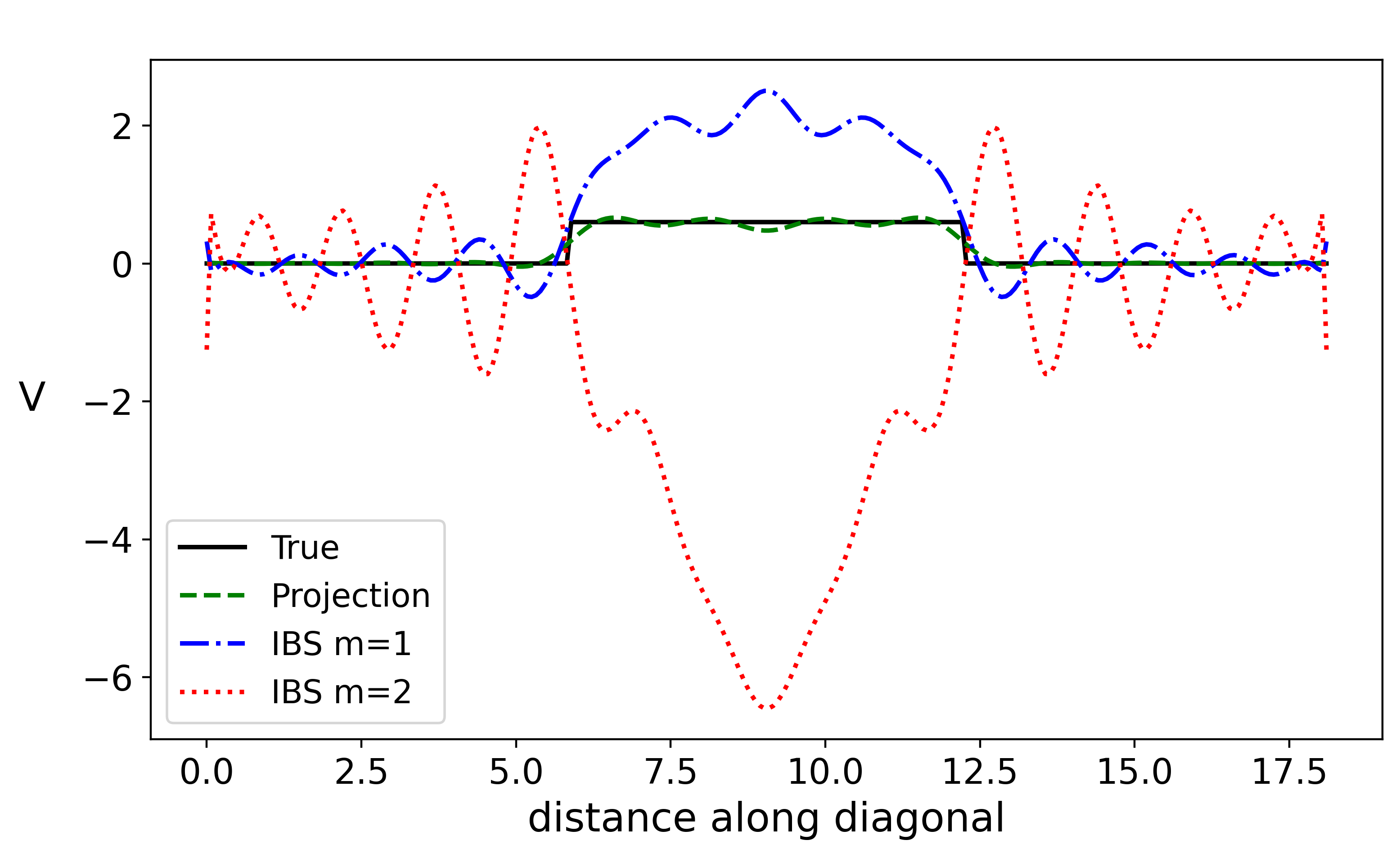}
        \caption{Anti-diagonal cross section}
        \label{fig:antichiral-disk-high-slice}
    \end{subfigure}

    \caption{Reconstructions of a high contrast disk (anti-chiral model)}
    \label{fig:antichiral-disk-high}
\end{figure}
\begin{figure}[htbp]
    \centering
    \begin{subfigure}[b]{0.9\textwidth}
        \centering
        \includegraphics[width=\textwidth]{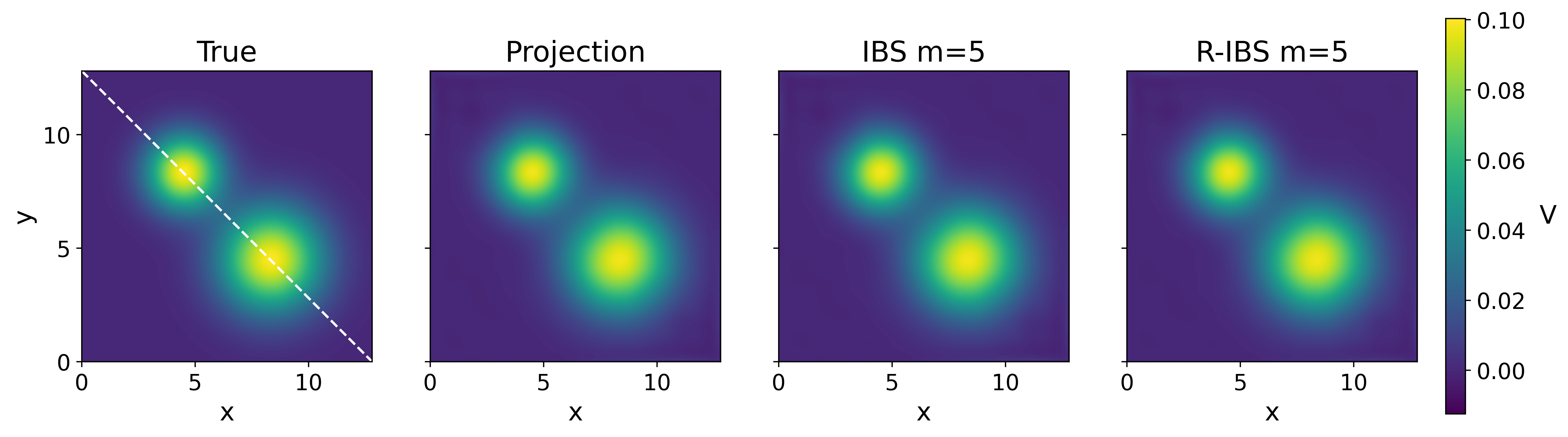}
        \caption{Reconstructions of $V$}
        \label{fig:antichiral-gauss-low-global}
    \end{subfigure}

    \vspace{0.6em}

    \begin{subfigure}[b]{0.45\textwidth}
        \centering
        \includegraphics[width=\textwidth]{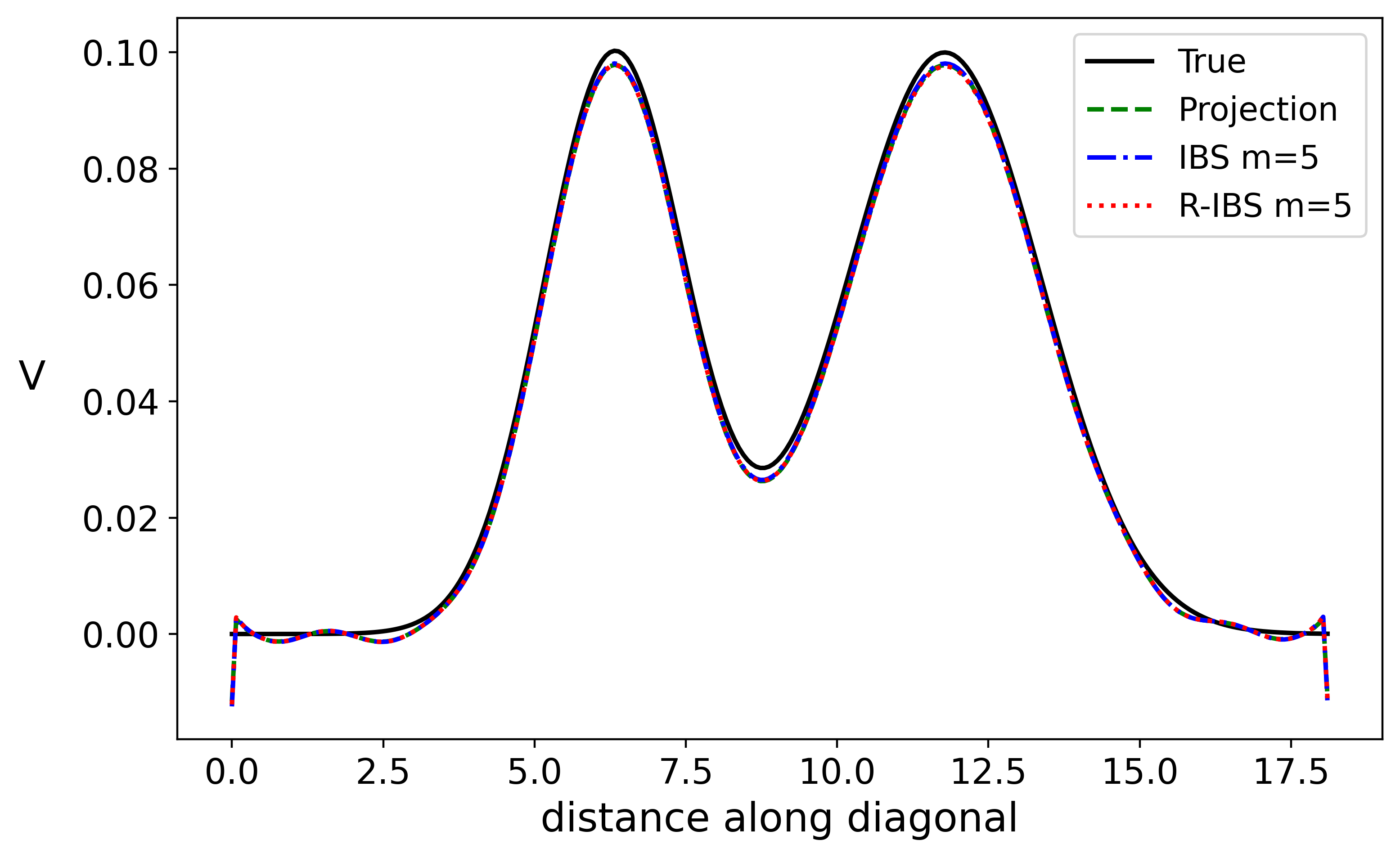}
        \caption{Anti-diagonal cross section}
        \label{fig:antichiral-gauss-low-slice-main}
    \end{subfigure}%

    \vspace{0.6em}

    \begin{subfigure}[b]{0.45\textwidth}
        \centering
        \includegraphics[width=\textwidth]{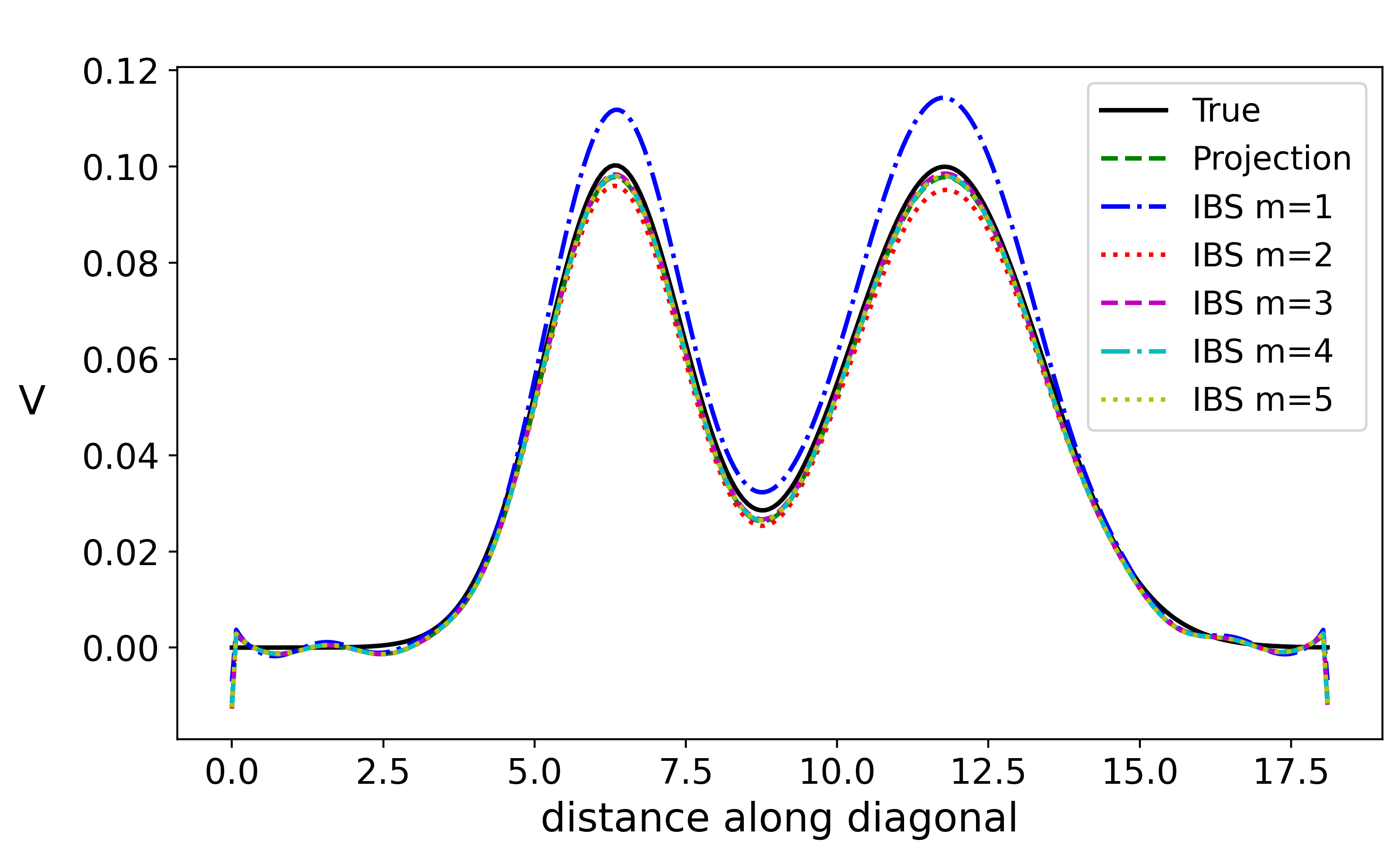}
        \caption{Anti-diagonal cross section for IBS}
        \label{fig:antichiral-gauss-low-slice-ibsall}
    \end{subfigure}%

    \scriptsize
\setlength{\tabcolsep}{3pt}%
\resizebox{\linewidth}{!}{
\begin{tabular}{@{}l|cccccc|ccccc@{}}
    \toprule
     & Projection &
     IBS1 & IBS2 & IBS3 & IBS4 & IBS5 &
     RIBS1 & RIBS2 & RIBS3 & RIBS4 & RIBS5 \\
    \midrule
    Relative error &
    0.0302 &   
    0.1203 &   
    0.0427 &   
    0.0281 &   
    0.0295 &   
    0.0293 &   
    0.1203 &   
    0.0427 &   
    0.0291 &   
    0.0306 &   
    0.0304 \\  
    \bottomrule
\end{tabular}}

    \caption{Reconstructions of two low contrast Gaussians (anti-chiral model)}
    \label{fig:antichiral-gauss-low}
\end{figure}
\begin{figure}[htbp]
    \centering
    \begin{subfigure}[b]{0.9\textwidth}
        \centering
        \includegraphics[width=\textwidth]{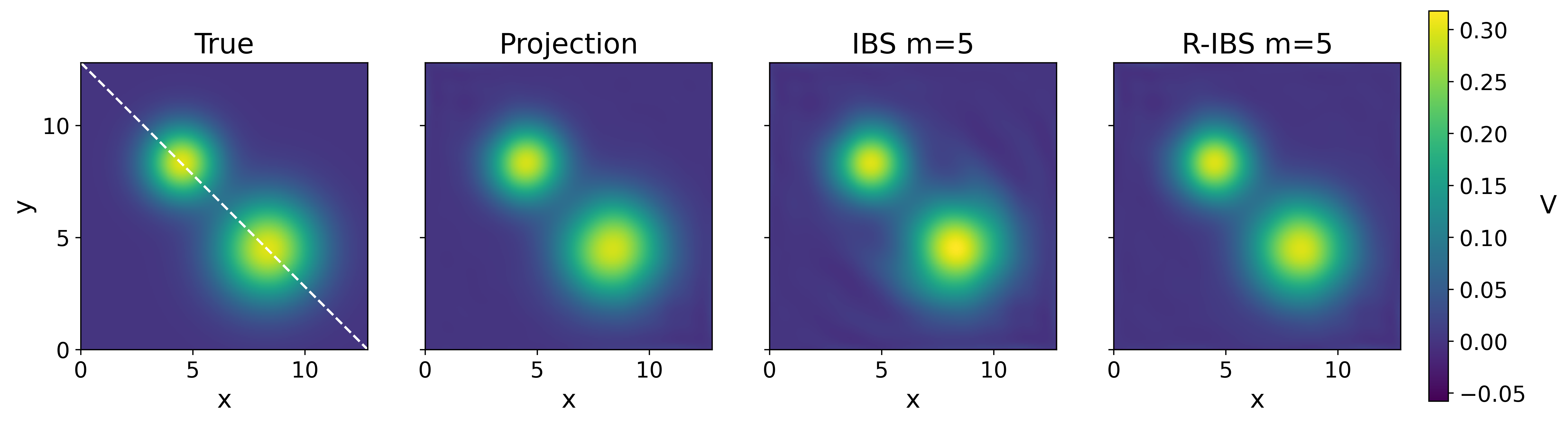}
        \caption{Reconstructions of $V$}
        \label{fig:antichiral-gauss-mid-global}
    \end{subfigure}

    \vspace{0.6em}

    \begin{subfigure}[b]{0.45\textwidth}
        \centering
        \includegraphics[width=\textwidth]{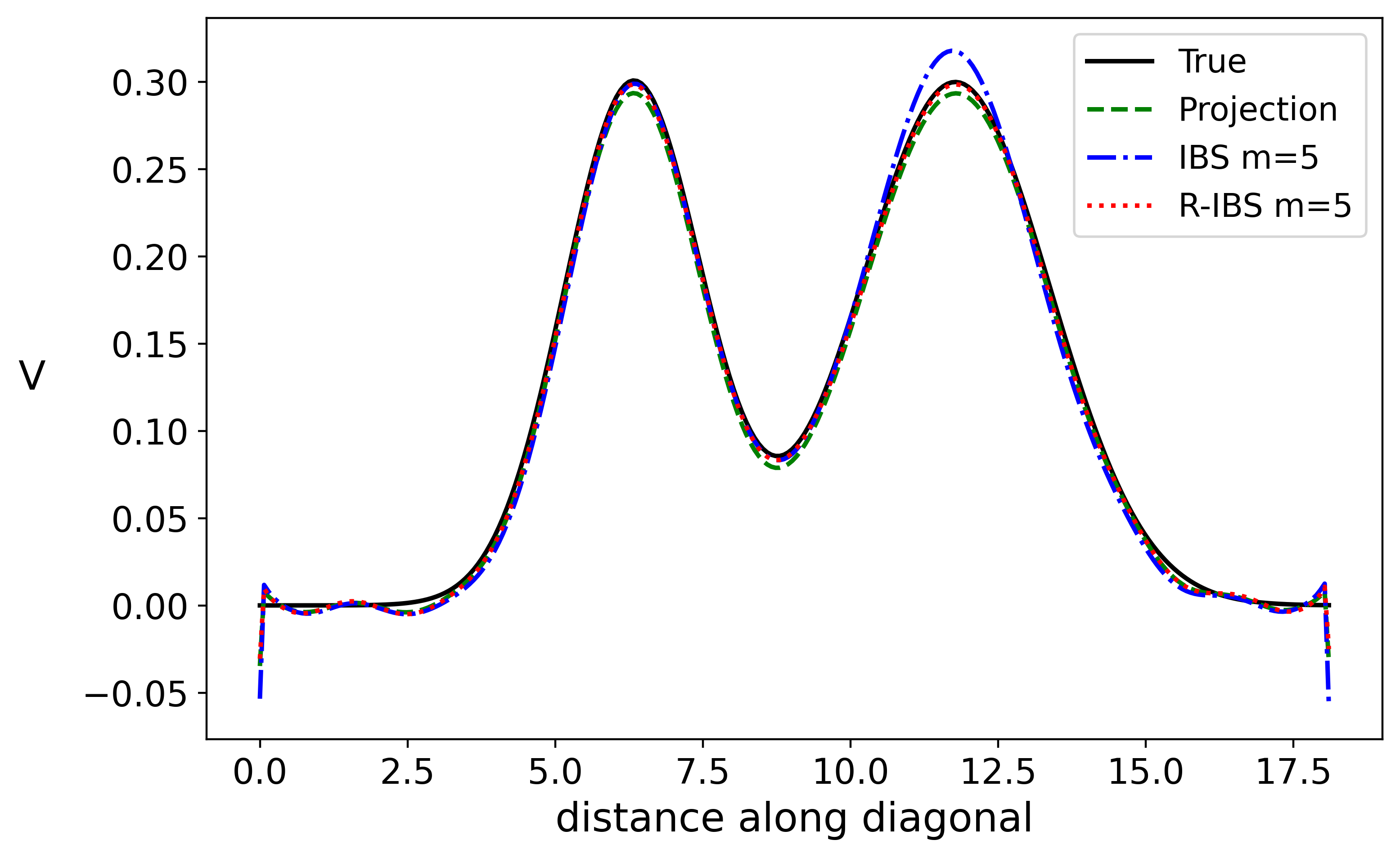}
        \caption{Anti-diagonal cross section }
        \label{fig:antichiral-gauss-mid-slice-main}
    \end{subfigure}%

    \vspace{0.6em}

    \begin{subfigure}[b]{0.45\textwidth}
        \centering
        \includegraphics[width=\textwidth]{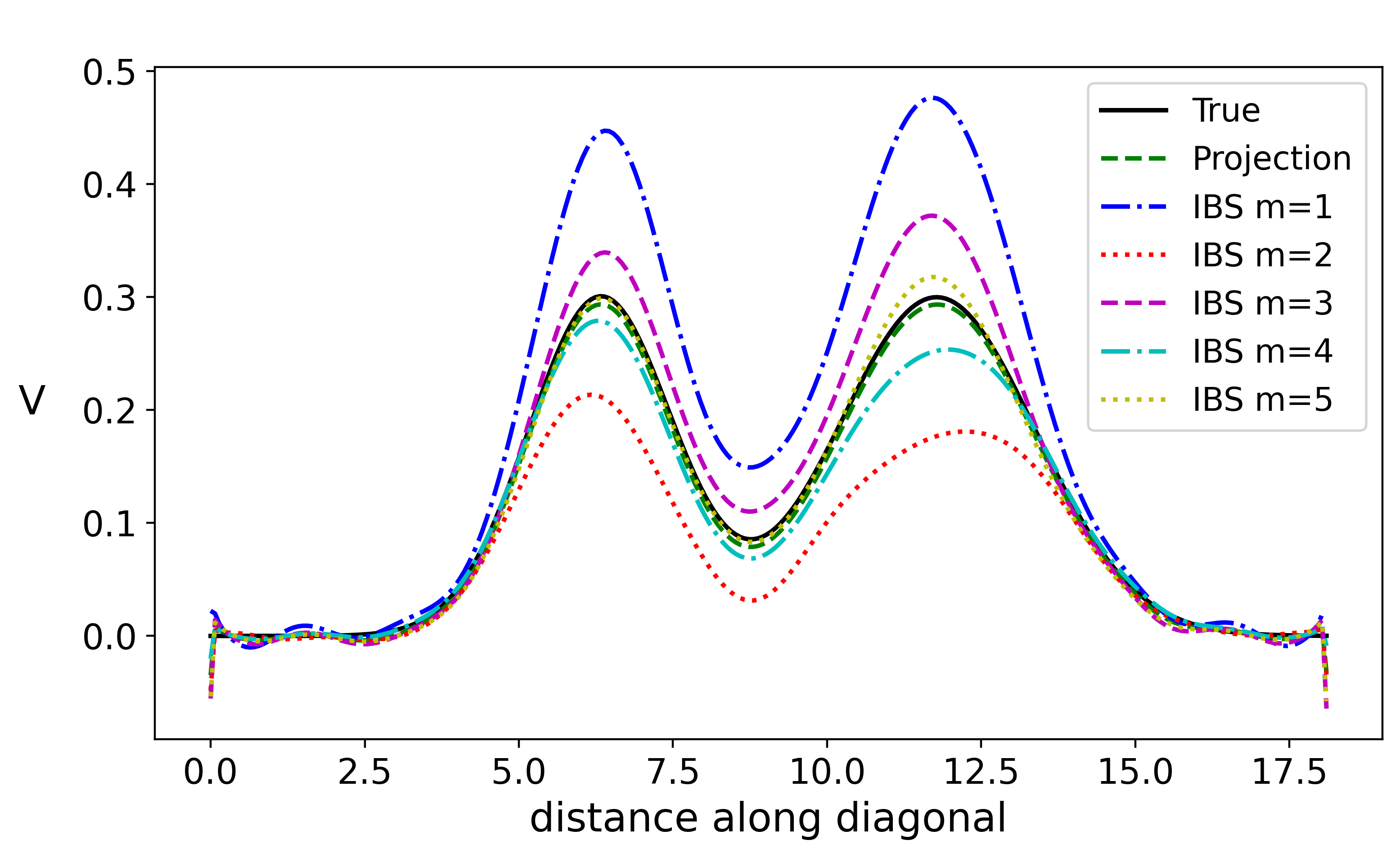}
        \caption{Anti-diagonal cross section for IBS}
        \label{fig:antichiral-gauss-mid-slice-ibsall}
    \end{subfigure}%

    \scriptsize
\setlength{\tabcolsep}{3pt}%
\resizebox{\linewidth}{!}{
\begin{tabular}{@{}l|cccccc|ccccc@{}}
    \toprule
     & Projection &
     IBS1 & IBS2 & IBS3 & IBS4 & IBS5 &
     RIBS1 & RIBS2 & RIBS3 & RIBS4 & RIBS5 \\
    \midrule
    Relative error &
    0.0302 &   
    0.4619 &   
    0.2986 &   
    0.1756 &   
    0.1216 &   
    0.0801 &   
    0.4619 &   
    0.2986 &   
    0.1084 &   
    0.1156 &   
    0.0403 \\  
    \bottomrule
\end{tabular}}

    \caption{Reconstructions of two medium contrast Gaussian scatterers (anti-chiral model)}
    \label{fig:antichiral-gauss-mid}
\end{figure}
\begin{figure}[htbp]
    \centering

    \begin{subfigure}[b]{\textwidth}
        \centering
        \includegraphics[width=\textwidth]{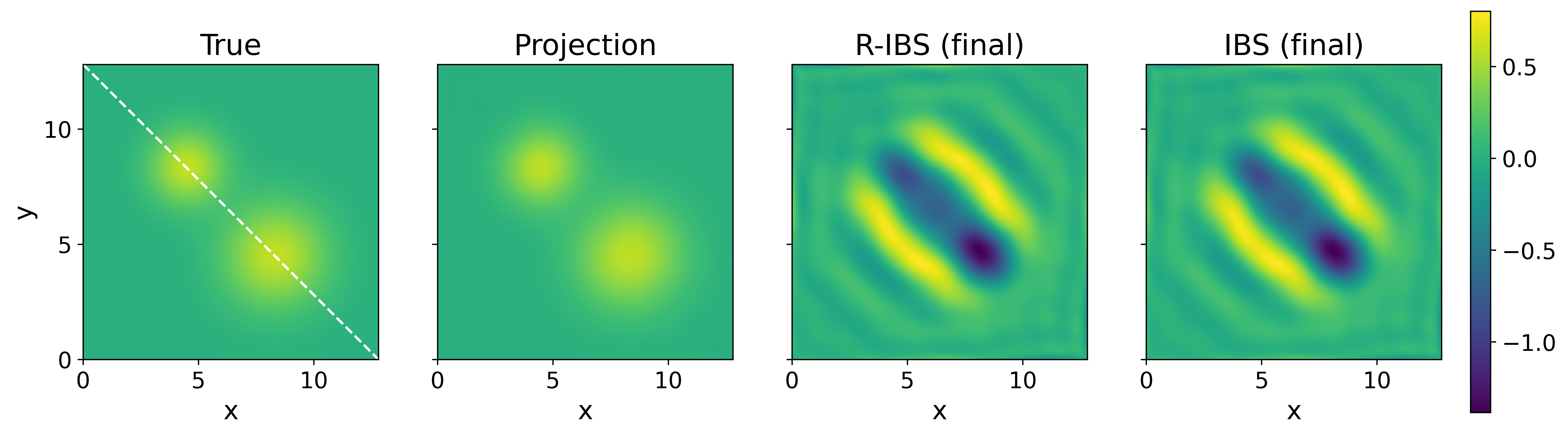}
        \caption{Reconstructions of $V$}
        \label{fig:antichiral-gauss-high-global}
    \end{subfigure}

    \par\vspace{0.5em}

    \begin{subfigure}[b]{0.5\textwidth}
        \centering
        \includegraphics[width=\textwidth]{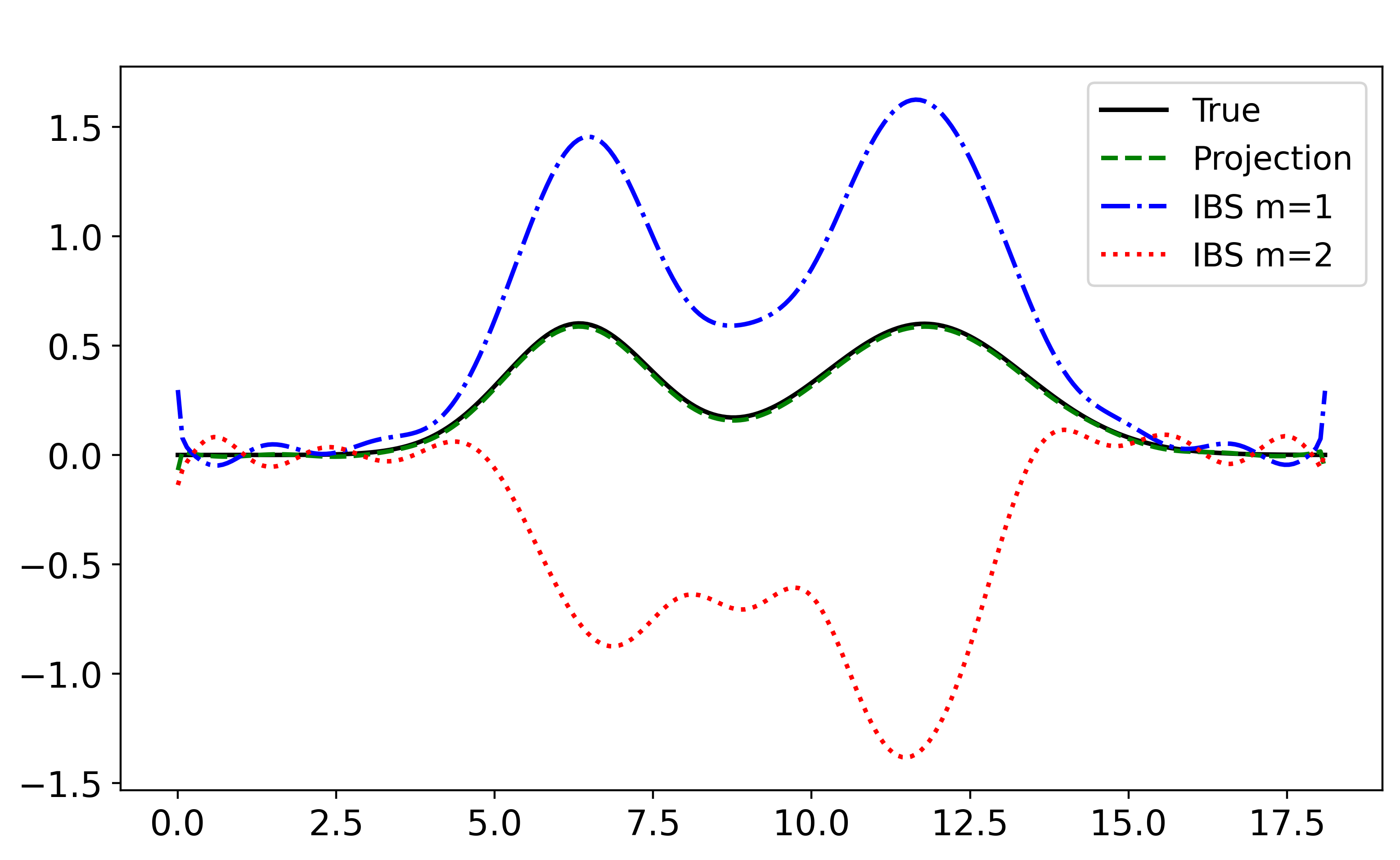}
        \caption{Anti-diagonal cross section}
        \label{fig:antichiral-gauss-high-slice}
    \end{subfigure}

    \caption{Reconstructions of two high contrast Gaussian scatterers (anti-chiral model)}
    \label{fig:antichiral-gauss-high}
\end{figure}
\begin{figure}[htbp]
    \centering
    \begin{subfigure}[b]{\textwidth}
        \centering
        \includegraphics[width=\textwidth]{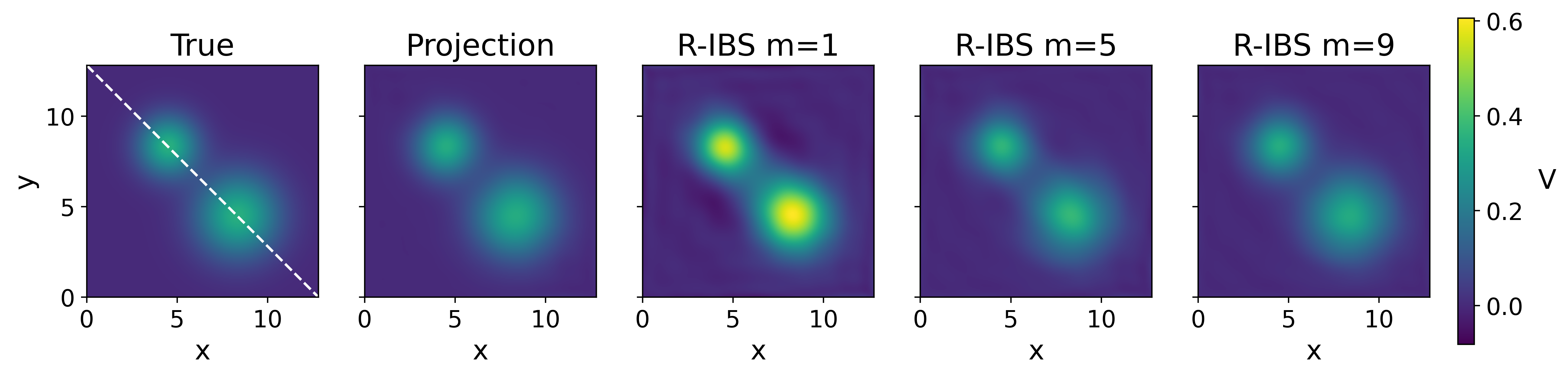}
        \caption{Reconstructions of $V$}
        \label{fig:antichiral-fullribs-global}
    \end{subfigure}

    \par\vspace{0.5em}

    \begin{subfigure}[b]{0.5\textwidth}
        \centering
        \includegraphics[width=\textwidth]{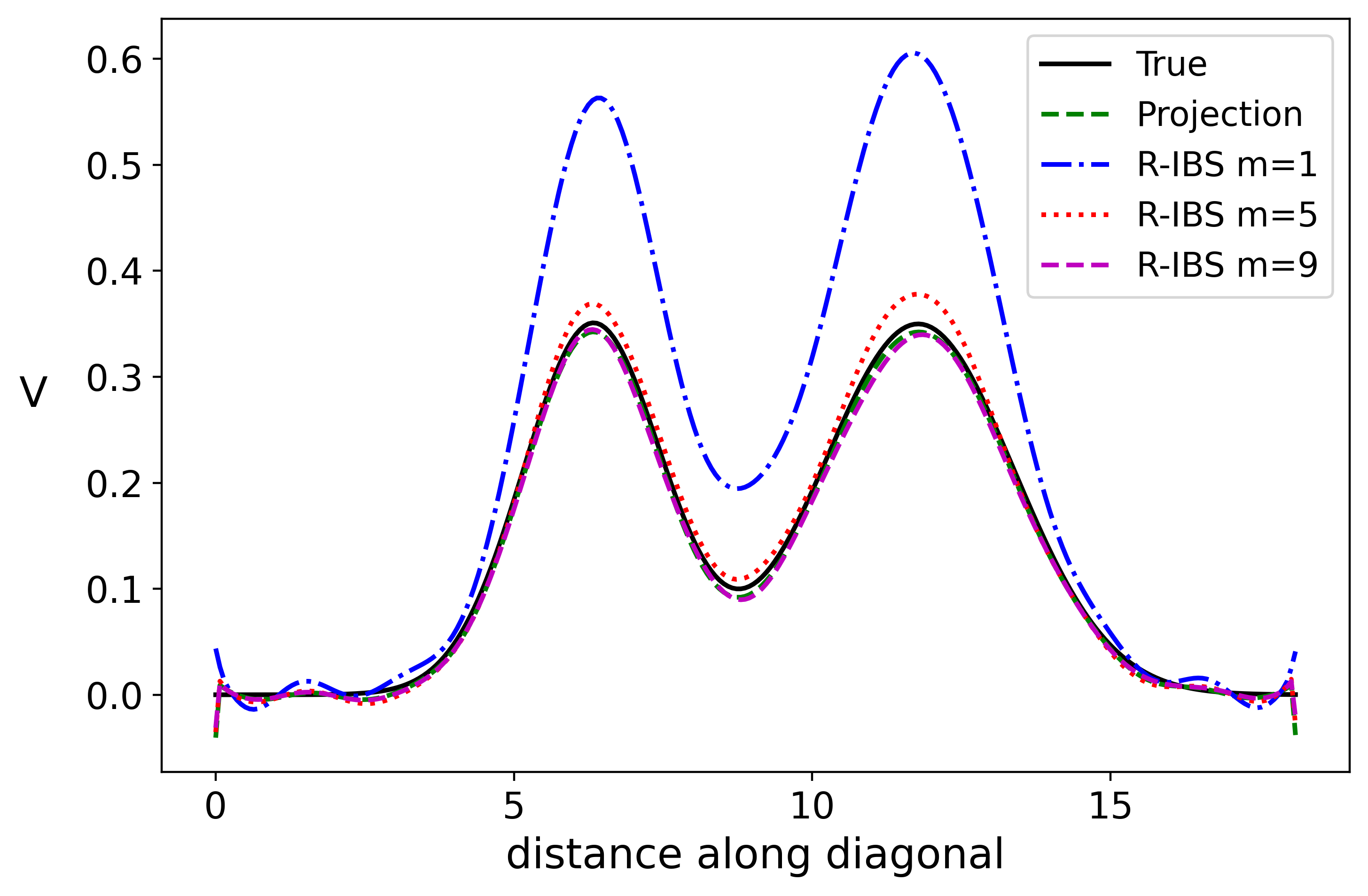}
        \caption{Anti-diagonal cross section}
        \label{fig:antichiral-fullribs-slice-antidiag}
    \end{subfigure}

    \scriptsize
\setlength{\tabcolsep}{3pt}%
\resizebox{\linewidth}{!}{
\begin{tabular}{@{}l|cccccccccc@{}}
    \toprule
     & Projection &
     RIBS1 & RIBS2 & RIBS3 & RIBS4 & RIBS5 &
     RIBS6 & RIBS7 & RIBS8 & RIBS9 \\
    \midrule
    Relative error &
    0.0302 &   
    0.5711 &   
    0.4496 &   
    0.2285 &   
    0.2286 &   
    0.1015 &   
    0.1352 &   
    0.0663 &   
    0.0956 &   
    0.0615 \\  
    \bottomrule
\end{tabular}}

    \caption{Reconstructions of two medium contrast Gaussian scatterers for the anti-chiral model using RIBS}
    \label{fig:antichiral-fullribs}
\end{figure}
The numerical experiments for the chiral model made use of the computational domain
\[
\Omega = [0,2.0]\times[-8.0,8.0],
\]
with wave number \(k=2.0\).
A total of \(64\) incident waves were generated, whose initial conditions are prescribed by
\[
\psi^{(j)}(0,y)=
\begin{pmatrix}
\exp\!\left(-\dfrac{(y-y_c^{(j)})^2}{0.08}\right)\\[6pt]
0
\end{pmatrix},
\quad y\in[-8,8],
\]
where the centers \(y_c^{(j)}\) are uniformly distributed in the interval \([-3.2,3.2]\).
The \(x\)-direction is discretized using \(400\) time steps, while the \(y\)-direction is discretized using \(1600\) spatial grid points. At the measurement location \(x=2\), \(1600\) detectors are uniformly placed along the transverse direction. To avoid the inverse crime, the synthetic data are generated on a finer grid with \(800\) time steps in the \(x\)-direction and \(3200\) spatial grid points in the \(y\)-direction. The Tikhonov regularization parameter is fixed as \(0.001\) throughout all experiments.

We investigate the reconstruction of two types of scattering potentials with low, medium, and high contrast.
The first is a disk defined by
\begin{align*}
V(x,y)
=
\begin{cases}
\sigma, & (x-1.0)^2 + y^2 \le (0.4)^2,\\[6pt]
0, & \text{otherwise},
\end{cases}
\end{align*}
where the contrast parameter \(\sigma\) is chosen as \(0.1\), \(0.5\), and \(2.0\).
The corresponding numerical reconstructions are presented in
Figures~\ref{fig:chiral-disk-low}, \ref{fig:chiral-disk-mid}, and \ref{fig:chiral-disk-high}.

The second potential consists of two Gaussian components,
\[
V(x,y)
=\sigma\left[
\exp\!\left(-\frac{(x-0.5)^2+y^2}{0.08}\right)
+
\exp\!\left(-\frac{(x-1.5)^2+y^2}{0.08}\right)
\right],
\]
where the parameter \(\sigma\) is again selected as \(0.1\), \(0.5\), and \(2.0\).
The corresponding results are shown in
Figures~\ref{fig:chiral-gauss-low}, \ref{fig:chiral-gauss-mid}, and \ref{fig:chiral-gauss-high}.
In addition, we apply the RIBS with 10 terms  with \(\sigma=0.5\).
The results are displayed in Figure~\ref{fig:chiral-fullribs}.

For the anti-chiral model, the computational domain is chosen as
\[
\Omega = [0,25.6]\times[0,25.6],
\]
with wave number \(k=1.0\).
The computational domain is discretized by a \(256\times256\) uniform grid, with all grid points on the boundary serving as detectors. To avoid the inverse crime, the synthetic data are generated independently on a finer \(512\times512\) uniform grid.
We again employ \(64\) incident waves of the form
\[
\psi_i(\theta)(x,y)
=
\begin{pmatrix}
\sin\!\left(\dfrac{\theta}{2}\right)\\[6pt]
-\cos\!\left(\dfrac{\theta}{2}\right)
\end{pmatrix}
\exp\!\left(\mathrm{i}k\big(x\cos\theta + y\sin\theta\big)\right),
\]
where the incident angles \(\theta\) are uniformly sampled from \([0,2\pi)\). The Tikhonov regularization parameter is fixed as \(0.05\) throughout all experiments.

Two classes of scattering potentials are considered.
The first is a circular disk given by
\begin{align*}
V(x,y)
=
\begin{cases}
\sigma, & (x-12.8)^2 + (y-12.8)^2 \le (6.4)^2,\\[6pt]
0, & \text{otherwise},
\end{cases}
\end{align*}
where the contrast  \(\sigma\) is set to \(0.1\), \(0.2\), and \(0.4\).
The corresponding reconstructions are shown in
Figures~\ref{fig:antichiral-disk-low}, \ref{fig:antichiral-disk-mid}, and \ref{fig:antichiral-disk-high}.

The second consists of two Gaussian-shaped scatterers,
\begin{align*}
V(x,y)
&=\sigma\Bigg[
\exp\!\left(
-\frac{(x-16.64)^2+(y-8.96)^2}{20.48}
\right)
+
\exp\!\left(
-\frac{(x-8.96)^2+(y-16.64)^2}{10.83}
\right)
\Bigg],
\end{align*}
where the intensity  \(\sigma\) is chosen as \(0.1\), \(0.3\), and \(0.6\).
The corresponding numerical results are presented in
Figures~\ref{fig:antichiral-gauss-low}, \ref{fig:antichiral-gauss-mid}, and \ref{fig:antichiral-gauss-high}.
Finally, we test the RIBS algorithm for this configuration with \(\sigma=0.35\).
The results are reported in Figure~\ref{fig:antichiral-fullribs}.

For both models, the numerical results exhibit the same qualitative behavior.
In the low-contrast regime, one or two IBS terms are sufficient to accurately recover the scattering potential.
As the contrast increases to a moderate level, a larger number of IBS terms is required for convergence, leading to improved reconstruction quality. 
In the high-contrast regime, however, the reconstruction fails.
These observations are consistent with the theoretical analysis in \cite{hoskinsAnalysisInverseBorn2022}.

Overall, the anti-chiral model yields better reconstruction quality than the chiral model.
This may seem counterintuitive, since the anti-chiral model is elliptic while the chiral
model is hyperbolic, and elliptic inverse problems are typically considered more ill-posed
than hyperbolic ones. The improved performance of the anti-chiral model is therefore not
due to ellipticity itself, but rather to the richer measurement data available in this
case. Specifically, the chiral model only provides final-time data, which may be viewed
as one-sided boundary measurements, whereas the anti-chiral model provides full boundary
data.

Another noteworthy observation is that the RIBS performs well for both models, producing reconstruction results that are comparable to, or even better than, those obtained using the IBS.
This suggests that using the RIBS can significantly reduce the computational cost without sacrificing reconstruction accuracy.
\section{Discussion}
In conclusion, we have investigated the IBS and the RIBS for two Dirac equations arising from the quantum optics of chiral and anti-chiral quantum waveguide arrays. We established the convergence properties of the proposed series and conducted extensive numerical experiments to demonstrate the effectiveness of the reconstruction algorithms. For both models, accurate reconstructions of the scattering potentials are achieved. Furthermore, the numerical results for the RIBS indicate that the cancellation phenomenon described in holds in both models. A deeper theoretical understanding of this behavior, as well as its implications for stability and convergence, will be the subject of future investigations.

\appendix
In the chiral model, the data consist of measurements of the scattered field at the distance $\psi_s(L_x,y)$, and the objective is to reconstruct the scattering potential $V$. In the anti-chiral model, measurements are taken on the boundary $\psi_s|_{\partial\Omega}$, and the goal is to recover $V$.

\section{Quantitative Estimate of $\mu_a$}\label{Quantitative Estimate}

The constants $\nu_c$ and $\nu_a$ can be computed explicitly; in our numerical experiments, they both take the value one. For $\mu_a$ defined in \eqref{def of mua}, we give an estimate here. To proceed, we require several estimates for Hankel functions. 
\begin{lem}\cite{freitasSharpBoundsModulus2018}
     For all $x>0$ and $l\geq 0$, define
     \begin{equation*}
        M_l(x) = |H^{(1)}_l(x)|^2.
     \end{equation*}
     Then the following estimates hold:
     \begin{itemize}
         \item $\dfrac{d}{dx}\big(M_0(x)\big)<0$.
         \item $M_1(x)\leq\dfrac{4}{\pi^2 x^{2}}+\dfrac{2}{\pi x}$.
         \item $M_0(x)\leq 1+\dfrac{4}{\pi^2}\big(\gamma+\log (x / 2)\big)^2$, where $\gamma$ denotes Euler’s constant.
         \item $M_0(x)\leq \dfrac{2}{\pi x}$, for $x\geq 1$.
     \end{itemize}
\end{lem}

\begin{remark}
    The asymptotic behavior of $|H^{(1)}_{l}(x)|$ is 
    \[
    |H_l^{(1)}(x)| \sim \sqrt{\frac{2}{\pi x}}.
    \]
    Hence, the above estimates are sharp for large $x$.
\end{remark}
\begin{thm}
    Suppose that the spatial domain $\Omega$ is contained in a disk of radius $R>1$. Then we have the estimate:
$$
\mu_a\leq\frac{k^2}{4}\left(I(k)
    +\frac{4}{3} \sqrt{\frac{2 \pi}{k}}\big(R^{3 / 2}-1\big)+\frac{2\pi^2 k}{3}\left(\left(\frac{2}{\pi k} R+\frac{4}{\pi^2 k^2}\right)^{3 / 2}-\left(\frac{2}{k\pi}\right)^{3 }\right)\right) ,
$$
where 
\[
I(k) :=
\begin{cases}
\pi + 2\left(\gamma+\log\frac{k}{2}-\tfrac{1}{2}\right) + \dfrac{8}{k^2}e^{-2\gamma}, & k \;\geq\; 2e^{-\gamma}, \\[1.2em]
(\pi+1) - 2\gamma - 2\log\frac{k}{2}, & k \;<\; 2e^{-\gamma}.
\end{cases}
\]
\end{thm}
\begin{proof}
We denote the four entries of the Green’s function matrix by
\begin{equation*}
    G(\mathbf{x},\mathbf{x}^{\prime}) 
    = \begin{pmatrix}
        G_{11}(\mathbf{x},\mathbf{x}^{\prime}) & G_{12}(\mathbf{x},\mathbf{x}^{\prime}) \\
        G_{21}(\mathbf{x},\mathbf{x}^{\prime}) & G_{22}(\mathbf{x},\mathbf{x}^{\prime})
    \end{pmatrix}.
\end{equation*}
We first estimate 
\(|G_{11}(\mathbf{x},\mathbf{x}^{\prime})|+|G_{12}(\mathbf{x},\mathbf{x}^{\prime})|\).
For $\mathbf{x}\neq \mathbf{x}^{\prime}$, define $r=|\mathbf{x}-\mathbf{x}^{\prime}|$, then
\begin{align*}
    \frac{4}{k}\Big(|G_{11}(\mathbf{x},\mathbf{x}^{\prime})|+|G_{12}(\mathbf{x},\mathbf{x}^{\prime})|\Big)
    &=\Big|\mathrm{i}H^{(1)}_0(k r)+\frac{x-x^{\prime}}{r}H^{(1)}_1(k r)\Big| \\
    &\quad +\Big|\frac{y-y^{\prime}}{r}H^{(1)}_1(k r)\Big| \\
    &\leq |H^{(1)}_0(k r)|+\sqrt{2}\,|H^{(1)}_1(k r)|.
\end{align*}
From the lemma, $|H^{(1)}_0(k r)|$ is monotone decreasing for $r>0$, hence
\begin{align*}
    \sup_{\mathbf{x}\in\Omega}\int_{\Omega}|H^{(1)}_0(k|\mathbf{x}-\mathbf{x}^{\prime}|)|\,d \mathbf{x}^{\prime}
    &\leq2\pi\int_{0}^{R}|H^{(1)}_0(k r)|r\,dr \\
    &\leq 2\pi\int_{0}^{1}r\sqrt{1+\frac{4}{\pi^2}\big(\gamma+\log (k r / 2)\big)^2}\,dr \\
    &\quad+2\pi\int_{1}^{R}r\sqrt{\frac{2}{\pi k r}}\,dr \\
    &\leq 2\pi\int_{0}^{1}\Big(r+|\frac{2 r}{\pi}\big(\gamma+\log (k r / 2)\big)|\Big)\,dr \\
    &\quad+\frac{4}{3} \sqrt{\frac{2 \pi}{k}}\big(R^{3 / 2}-1\big)\\
    &= I(k)
    +\frac{4}{3} \sqrt{\frac{2 \pi}{k}}\big(R^{3 / 2}-1\big).
\end{align*}
On the other hand, since $\dfrac{4}{\pi^2 x^{2}}+\dfrac{2}{\pi x}$ is decreasing for $x>0$, it follows that
\begin{align*}
    \sup_{\mathbf{x}\in\Omega}\int_{\Omega}|H^{(1)}_1(k|\mathbf{x}-\mathbf{x}^{\prime}|)|\,d \mathbf{x}^{\prime}
    &\leq\sup_{\mathbf{x}\in\Omega} \int_{\Omega}\sqrt{\frac{4}{\pi^2 (k|\mathbf{x}-\mathbf{x}^{\prime}|)^{2}}+\frac{2}{\pi k|\mathbf{x}-\mathbf{x}^{\prime}|}}\,d \mathbf{x}^{\prime}\\
    &\leq 2\pi\int_{0}^{R}r\sqrt{\frac{4}{\pi^2 (k r)^{2}}+\frac{2}{\pi k r}}\,dr\\
    &=\frac{2\pi^2 k}{3}\left(\left(\frac{2}{\pi k} R+\frac{4}{\pi^2 k^2}\right)^{3 / 2}-\left(\frac{2}{k\pi}\right)^{3 }\right).
\end{align*}
The estimate for \(|G_{21}(\mathbf{x},\mathbf{x}^{\prime})|+|G_{22}(\mathbf{x},\mathbf{x}^{\prime})|\) is similar, establishing the required results for $\mu_a$. 
\end{proof}
\bibliographystyle{elsarticle-num} 
\bibliography{ref}

\end{document}